\documentclass[notheoremnums,9pt]{dpreprint}
\begin{document}

\newtheorem{theorem}{Theorem}[section]
\newtheorem*{theorem*}{Theorem}
\newtheorem{corollary}[theorem]{Corollary}
\newtheorem*{corollary*}{Corollary}
\newtheorem*{remark}{Remark}
\newtheorem*{question}{Question}
\newtheorem*{example}{Example}
\newtheorem*{conjecture}{Conjecture}
\newtheorem{proposition}{Proposition}[subsection]
\newtheorem{lemma}[proposition]{Lemma}
\newtheorem*{lemma*}{Lemma}
\theoremstyle{definition}
\newtheorem{definition}[theorem]{Definition}
\newtheorem*{definition*}{Definition}

\dtitle[Stabilizers of Ergodic Actions of Lattices and Commensurators]{Stabilizers of Ergodic Actions of Lattices and Commensurators}
\dauthorone[D.~Creutz]{Darren Creutz}{darren.creutz@vanderbilt.edu}{Vanderbilt University}{}
\dauthortwo[J.~Peterson]{Jesse Peterson}{jesse.d.peterson@vanderbilt.edu}{Vanderbilt University}{%
Partially supported by NSF Grant 0901510 and a grant from the Alfred P.~Sloan Foundation
}
\datewritten{25 May 2015}
\dabstract{%
\large
We prove that any ergodic measure-preserving action of an irreducible lattice in a semisimple group, with finite center and each simple factor having rank at least two, either has finite orbits or has finite stabilizers.  The same dichotomy holds for many commensurators of such lattices.

The above are derived from more general results on groups with the Howe-Moore property and property $(T)$.  We prove similar results for commensurators in such groups and for irreducible lattices (and commensurators) in products of at least two such groups, at least one of which is totally disconnected.
}

\makepreprint

\section{Introduction}

A groundbreaking result in the theory of lattices in semisimple groups is the Margulis Normal Subgroup Theorem \cite{Ma79},\cite{Ma91}: any nontrivial normal subgroup of an irreducible lattice in a center-free higher-rank semisimple group has finite index.  In the case of real semisimple Lie groups, Stuck and Zimmer \cite{SZ94} generalized this result to ergodic measure-preserving actions of such lattices: any irreducible ergodic measure-preserving action of a semisimple real Lie group, each simple factor having rank at least two, is either essentially  free or essentially transitive; and any ergodic measure-preserving action of an irreducible lattice in such a semisimple real Lie group either has finite orbits or has finite stabilizers.

More recently, Bader and Shalom \cite{BS04} proved a Normal Subgroup Theorem for irreducible lattices in products of locally compact groups: any infinite normal subgroup of an irreducible integrable lattice in a product of nondiscrete just noncompact locally compact second countable compactly generated groups, not both isomorphic to $\mathbb{R}$, has finite index.

Our purpose is to extend this dichotomy for ergodic measure-preserving actions to irreducible lattices and commensurators of lattices in semisimple groups, each factor having higher-rank, and more generally to lattices in products of at least two groups with the Howe-Moore property and property $(T)$.

The Stuck-Zimmer result follows from an Intermediate Factor Theorem, a generalization of the Factor Theorem of Margulis, due to Zimmer \cite{zimmer2} and Nevo-Zimmer \cite{nevozimmer}.  
 A key step in the work of Bader-Shalom is a similar Intermediate Factor Theorem for product groups which they use to show: any irreducible ergodic measure-preserving action of a product of two locally compact second countable just noncompact groups with property $(T)$ is either essentially free or essentially transitive.  One of the main ingredients in our work is an Intermediate Factor Theorem for relatively contractive actions which we will discuss presently.

The methods of Bader-Shalom do not easily yield the same result for actions of irreducible lattices in products of groups.  The major issue is that inducing an action of a lattice may yield an action of the ambient group which is not irreducible but the Bader-Shalom (and Stuck-Zimmer) Intermediate Factor Theorems only apply to irreducible actions.  The same issue arises when attempting to apply the Stuck-Zimmer methods to lattices in semisimple groups with $p$-adic parts.

Our techniques are a generalization of the Normal Subgroup Theorem for Commensurators due to the first author and Shalom \cite{CS14},\cite{CrD11}: if $\Lambda$ is a dense commensurator of a lattice in a locally compact second countable group that is not a compact extension of an abelian group such that $\Lambda$ does not infinitely intersect any noncocompact normal subgroup then any infinite normal subgroup of $\Lambda$ contains the lattice up to finite index; the commensurability classes of infinite normal subgroups of such a commensurator are in a one-one onto correspondence with the commensurability classes of open normal subgroups of the relative profinite completion.

The difficulty for lattices does not arise using our techniques as we do not need to induce the action of the lattice to the ambient group, but rather analyze the action directly by treating the lattice as a commensurator in a proper subproduct.  For this, we require precisely the object that is the obstruction in the Stuck-Zimmer approach: a totally disconnected factor.  In this sense, our methods complement those of Stuck and Zimmer and combining results we are able to handle all $S$-arithmetic lattices.  Our methods also lead to results on actions of commensurators and allow us to prove the corresponding generalization of the Normal Subgroup Theorem of Bader-Shalom to actions of lattices (provided one group in the product is totally disconnected).  For this generalization, we impose the requirement of the ambient groups having the Howe-Moore property as the measurable analogue of just noncompactness.

\subsection{Main Results}

We now state the main results of the paper.  Recall that an action is weakly amenable when the corresponding equivalence relation is amenable.

\begin{theorem*}[Theorem \ref{T:main} and Corollary \ref{C:maintheorem1}]
Let $G$ be a noncompact nondiscrete locally compact second countable group with the Howe-Moore property.  Let $\Gamma < G$ be a lattice and let $\Lambda < G$ be a countable dense subgroup such that $\Lambda$ contains and commensurates $\Gamma$ and such that $\Lambda$ has finite intersection with every compact normal subgroup of $G$.

Then any ergodic measure-preserving action of $\Lambda$ either has finite stabilizers or the restriction of the action to $\Gamma$ is weakly amenable.

If, in addition, $G$ has property $(T)$ then any ergodic measure-preserving action of $\Lambda$ either has finite stabilizers or the restriction of the action to $\Gamma$ has finite orbits.
\end{theorem*}

The result on weak amenability applies to all noncompact simple Lie groups--even those without higher-rank--and also to automorphism groups of regular trees, both of which have the Howe-Moore property \cite{HM79}, \cite{LM92}.  One consequence of this is that any ergodic measure-preserving action of the commensurator on a nonatomic probability space, which is strongly ergodic when restricted to the lattice, has finite stabilizers.

The additional assumption that the ambient group have property $(T)$ allows us to conclude that any weakly amenable action of the lattice, which will also have property $(T)$, has, in fact, finite orbits.  This also accounts for the requirement in the Stuck-Zimmer Theorem that each simple factor has higher-rank.

Treating lattices in products of groups, at least one of which is totally disconnected, as commensurators of lattices sitting in proper subproducts, we obtain the following generalization of the Bader-Shalom Theorem to actions:
\begin{theorem*}[Theorem \ref{T:BSbetter}]
Let $G$ be a product of at least two simple nondiscrete noncompact locally compact second countable groups with the Howe-Moore property, at least one of which has property $(T)$, at least one of which is totally disconnected and such that every connected simple factor has property $(T)$.  Let $\Gamma < G$ be an irreducible lattice.

Then any ergodic measure-preserving action of $\Gamma$ either has finite orbits or has finite stabilizers.
\end{theorem*}

Specializing to Lie groups:
\begin{theorem*}[Corollary \ref{C:latticeLie2}]
Let $G$ be a semisimple Lie group (real or $p$-adic or both) with no compact factors, trivial center, at least one factor with rank at least two and such that each real simple factor has rank at least two.  Let $\Gamma < G$ be an irreducible lattice.  Then any ergodic measure-preserving action of $\Gamma$ on a nonatomic probability space is essentially free.
\end{theorem*}
%

In particular, we obtain examples of groups without property $(T)$ having uncountably many subgroups that admit only essentially free actions:
\begin{theorem*}[Corollary \ref{C:algK}]
Let $\mathbf{G}$ be a simple algebraic group defined over a global field $K$ such that $\mathbf{G}$ has $v$-rank at least two for some place $v$ and has $v_{\infty}$-rank at least two for every infinite place $v_{\infty}$.  Then every nontrivial ergodic measure-preserving action of $\mathbf{G}(K)$ is essentially free.
\end{theorem*}

One consequence of this fact is that results from the theory of orbit equivalence, which often require that the actions in question be essentially free, apply to all actions of such groups.  For example, since any nonamenable group cannot act freely and give rise to the hyperfinite $\mathrm{II}_{1}$ equivalence relation \cite{dye1},\cite{Zi84}:
\begin{corollary*}
Let $\mathbf{G}$ be a simple algebraic group defined over a global field $K$ such that $\mathbf{G}$ has $v$-rank at least two for some place $v$ and has $v_{\infty}$-rank at least two for every infinite place $v_{\infty}$.  
Then there is no nontrivial homomorphism of $\mathbf{G}(K)$ to the full group of the hyperfinite $\mathrm{II}_{1}$ equivalence relation.
\end{corollary*}

The above result also holds if we replace the hyperfinite $\mathrm{II}_1$ equivalence relation with any measure preserving equivalence relation which is treeable, or more generally which has the Haagerup property (see \cite{MR2122917} and \cite{MR2215135} Theorem 5.4).  More generally, if $\mathbf{G}$ is a simple algebraic group defined over $\mathbb{Q}$ and $S$ is a set of primes as in the statement of Corollary \ref{C:rationalLie} then any homomorphism of $\mathbf{G}(\mathbb{Z}_{S})$ into the full group of the hyperfinite $\mathrm{II}_{1}$ equivalence relation is precompact (when $S$ is finite this also follows from Robertson \cite{robertson}).  It seems plausible that the above result still holds if we replace the full group of the hyperfinite $\mathrm{II}_{1}$ equivalence relation with the unitary group of the hyperfinite $\mathrm{II}_{1}$ factor (see Bekka \cite{bekka} for results in this direction).

\subsection{Stabilizers of Actions and Invariant Random Subgroups}

The results described above can be suitably interpreted in terms of invariant random subgroups.  Invariant random subgroups are conjugation-invariant probability measures on the space of closed subgroups and naturally arise from the stabilizer subgroups of measure-preserving actions.  This notion was introduced in \cite{AGV12} where it is shown that, conversely, every invariant random subgroup arises in this way (\cite{AGV12} Proposition 13, see also Section \ref{S:rs} below).  From this perspective the Stuck-Zimmer Theorem \cite{SZ94} then states that semisimple real Lie groups, with each factor having higher-rank, and their irreducible lattices, admit no nonobvious invariant random subgroups and the Bader-Shalom result states the same for irreducible invariant random subgroups of products of nondiscrete locally compact second countable groups with property $(T)$.  Our results can likewise be interpreted in this context.

The study of stabilizers of actions dates back to the work of Moore, \cite{moore66} Chapter 2, and Ramsay, \cite{ramsay} Section 9 (see also Adams and Stuck \cite{adamsstuck} Section 4).
Bergeron and Gaboriau \cite{gaboriau} observed that invariant random subgroups behave similarly to normal subgroups and this topic has attracted much recent attention: \cite{seven}, \cite{AGV12}, \cite{bowen12}, \cite{grig1}, \cite{grig2}, \cite{vershik11}, \cite{vershiktotallynonfree}.

Our work, like that of Stuck and Zimmer, rules out the existence of nonobvious invariant random subgroups for certain groups.  This stands in stark contrast to nonabelian free groups, which admit a large family of invariant random subgroups \cite{bowen12}.  Even simple groups can admit large families of nonfree actions: Vershik showed that the infinite alternating group admits many such actions \cite{vershiktotallynonfree}.  Another class of examples can be found by considering the commutator subgroup of the topological full group of Cantor minimal systems, which were shown to be simple by Matsui \cite{matsui} (and more recently to also be amenable by Juschenko and Monod \cite{monod}).

The main contribution to the theory we make here is introducing a technique, based on joinings, that allows us to formulate meaningful definitions of notions such as containment and commensuration for invariant random subgroups that extend the usual notions for subgroups:
\begin{definition*}[Definition \ref{def:commrs}]
Two invariant random subgroups are commensurate when there exists a joining such that almost surely the intersection has finite index in both.
\end{definition*}
The joinings technique may be of independent interest and should allow for more general definitions of properties of invariant random subgroups akin to those of subgroups.
We use this definition to formulate a one-one correspondence:
\begin{theorem*}[Theorem \ref{T:semisimplecomm}]
Let $G$ be a semisimple Lie group (real or $p$-adic or both) with finite center where each simple factor has rank at least two.  Let $\Gamma < G$ be an irreducible lattice and let $\Lambda < G$ be a countable dense subgroup such that $\Lambda$ contains and commensurates $\Gamma$ and such that $\Lambda$ has finite intersection with every proper subfactor of $G$.

Then any ergodic measure-preserving action of $\Lambda$ either has finite stabilizers or the restriction of the action to $\Gamma$ has finite orbits.

Moreover, the commensurability classes of infinite ergodic invariant random subgroups of $\Lambda$ are in one-one, onto correspondence with the commensurability classes of open ergodic invariant random subgroups of the relative profinite completion $\rpf{\Lambda}{\Gamma}$.
\end{theorem*}

\subsection{Relatively Contractive Maps}

Strongly approximately transitive (SAT) actions, introduced by Jaworski in \cite{Ja94}, \cite{Ja95}, are the extreme opposite of measure-preserving actions: an action of a group $G$ on a probability space $(X,\nu)$, with $\nu$ quasi-invariant under the $G$-action, such that for any measurable set $B$ of less than full measure there exists a sequence $g_{n} \in G$ which ``contracts" $B$, that is $\nu(g_{n}B) \to 0$.

We introduce a relative version of this property, akin to relative measure-preserving, by
saying that
a $G$-equivariant map $\pi : (X,\nu) \to (Y,\eta)$ between $G$-spaces with quasi-invariant probability measures is relatively contractive when it is ``contractive over each fiber" (see Section \ref{S:relativecontractive} for a precise definition).  This is a generalization of the notion of proximal maps, which can also be thought of as ``relatively boundary maps" in the context of stationary dynamical systems.

Following \cite{FG10}, we say a continuous action of a group $G$ on a compact metric space $X$ with quasi-invariant Borel probability measure $\nu \in P(X)$ is contractible when for every $x \in X$ there exists $g_{n} \in G$ such that $g_{n}\nu \to \delta_{x}$ in weak*.  Furstenberg and Glasner \cite{FG10} showed that an action is SAT if and only if every continuous compact model is contractible.

We generalize this to the relative case and obtain that a $G$-space is a relatively contractive extension of a point if and only if it is SAT.  For this reason, we adopt the somewhat more descriptive term contractive to refer to such spaces.

Contractive spaces are the central dynamical concept in the proof of the amenability half of the Normal Subgroup Theorem for Commensurators \cite{CS14}, \cite{CrD11} and have been studied in the context of stationary dynamical systems by Kaimanovich \cite{kaimanovichSAT}.
Jaworski introduced the notion as a stronger form of the approximate transitivity property of Connes and Woods \cite{conneswood} to study the Choquet-Deny property on groups and showed that Poisson boundaries are contractive.  The main benefit contractive spaces offer over boundaries is greater flexibility in that one need not impose a measure on the group.

We show that relatively contractive maps are essentially unique, which is crucial for the Intermediate Contractive Factor Theorem:
\begin{theorem*}[Theorem \ref{T:contractiveIFT}]
Let $\Gamma < G$ be a lattice in a locally compact second countable group and let $\Lambda$ contain and commensurate $\Gamma$ and be dense in $G$.

Let $(X,\nu)$ be a $G$-space that is $\Gamma$-contractive and $(Y,\eta)$ be a measure-preserving $G$-space.  Let $\pi : (X \times Y, \nu \times \eta) \to (Y,\eta)$ be the natural projection map from the product space with the diagonal action.

Let $(Z,\zeta)$ be a $\Lambda$-space such that there exist $\Gamma$-maps $\varphi : (X\times Y,\nu\times\eta) \to (Z,\zeta)$ and $\rho : (Z,\zeta) \to (Y,\eta)$ with $\rho \circ \varphi = \pi$.  Then $\varphi$ and $\rho$ are $\Lambda$-maps and $(Z,\zeta)$ is $\Lambda$-isomorphic to a $G$-space and over this isomorphism the maps $\varphi$ and $\rho$ become $G$-maps.
\end{theorem*}
We will actually need a stronger version of the Intermediate Factor Theorem (Theorem \ref{T:groupoidcontractive1}) which can be viewed as a ``piecewise" or groupoid version in line with the virtual groups of Mackey \cite{mackey}.

\subsection{Acknowledgments}

The authors would like to thank N.~Monod and C.~Houdayer for some helpful remarks on an initial draft of the paper, A.~Salehi~Golsefidy for pointing out the example of lattices and commensurators in algebraic groups over fields with positive characteristic, and R.~Tucker-Drob for allowing us to present an argument of his \cite{tuckerdrob} in the proof of Corollary \ref{C:howemoore}, which simplifies our approach in the non-torsion-free case.  The authors would also like to thank the referees for many helpful remarks in ways to simplify the paper and improve its presentation.

\section{Preliminaries}\label{S:prelims}

\subsection{Lattices and Commensuration}

Let $G$ be a locally compact second countable group.  A subgroup $\Gamma < G$ is a \textbf{lattice} when it is discrete and there exists a fundamental domain for $G / \Gamma$ with finite Haar measure.  A lattice is \textbf{irreducible} when the projection modulo any noncocompact closed normal subgroup is dense.

A subgroup $\Lambda < G$ \textbf{commensurates} another subgroup $\Gamma < G$ when for every $\lambda \in \Lambda$ the group $\Gamma \cap \lambda \Gamma \lambda^{-1}$ has finite index in both $\Gamma$ and $\lambda \Gamma \lambda^{-1}$.  When $\Gamma < \Lambda$ is a subgroup of $\Lambda$ we will write
\[
\Gamma <_{c} \Lambda
\]
to mean that $\Gamma$ is a commensurated subgroup of $\Lambda$.

Let $\Gamma < G$ be a lattice in a locally compact second countable group.  Then
\[
\mathrm{Comm}_{G}(\Gamma) = \{ g \in G : [\Gamma : \Gamma \cap g\Gamma g^{-1}] < \infty \text{ and } [g\Gamma g^{-1} : \Gamma \cap g\Gamma g^{-1}] < \infty \}
\]
is the \textbf{commensurator} of $\Gamma$ in $G$.

\subsection{Group Actions on Measure Spaces}

Throughout the paper, we will always assume groups are locally compact second countable and that measure spaces are standard probability spaces (except when otherwise stated).

\begin{definition}
A group $G$ \textbf{acts on} a space $X$ when there is a map $G \times X \to X$, written $gx$, such that $g(hx) = (gh)x$.  For $\nu \in P(X)$ a Borel probability measure on $X$, we say that $\nu$ is \textbf{quasi-invariant} when the $G$-action preserves the measure class and \textbf{invariant} or \textbf{measure-preserving} when $G$ preserves $\nu$.

We write $G \actson (X,\nu)$ and refer to $(X,\nu)$ as a \textbf{$G$-space} when $G$ acts on $X$ and $\nu \in P(X)$ is quasi-invariant and the action map $G \times X \to X$ is $\Haar \times \nu$-measurable.
\end{definition}

\begin{definition}
Let $G \actson (X,\nu)$.  The \textbf{stabilizer} subgroups are written
\[
\mathrm{stab}_{G}(x) = \{ g \in G : gx = x \}
\]
and when $\Gamma < G$ is a subgroup we write
$\mathrm{stab}_{\Gamma}(x) = \{ \gamma \in \Gamma : \gamma x = x \} = \mathrm{stab}_{G}(x) \cap \Gamma$
for the stabilizer of $x$ when the action is restricted to $\Gamma$.
\end{definition}

\begin{definition}
$G \actson (X,\nu)$ is \textbf{essentially transitive} when $Gx$ is conull in $X$ for some $x$;\\
$G \actson (X,\nu)$ is \textbf{essentially free} when $\mathrm{stab}_{G}(x)$ is trivial for a.e.~$x$;\\
$G \actson (X,\nu)$ has \textbf{finite stabilizers} when $\mathrm{stab}_{G}(x)$ is finite for a.e.~$x$;\\
$G \actson (X,\nu)$ has \textbf{finite orbits} when the orbit $Gx$ is finite for a.e.~$x$;\\
$G \actson (X,\nu)$ is \textbf{ergodic} when every $G$-invariant measurable set is null or conull; and\\
$G \actson (X,\nu)$ is \textbf{irreducibly ergodic} (\textbf{irreducible}) when it is ergodic for every nontrivial closed normal subgroup of $G$.
\end{definition}

%
%
%
%
%
%


\begin{definition}
Let $G$ be a locally compact second countable group and $\pi : (X,\nu) \to (Y,\eta)$ a measurable map such that $\pi_{*}\nu = \eta$ and $\pi(gx) = g\pi(x)$ for all $g \in G$ and almost every $x \in X$.  Such a map $\pi$ is a \textbf{$G$-map} of $G$-spaces.
\end{definition}

\begin{definition}
Given a measurable map $\pi : X \to Y$ the \textbf{push-forward map} $\pi_{*} : P(X) \to P(Y)$, mentioned above, is defined by $(\pi_{*}\nu)(B) = \nu(\pi^{-1}(B))$ for $B \subseteq Y$ measurable.
\end{definition}

In order to relativize properties of $G$-spaces to $G$-maps, we will need to focus on the behavior of the disintegration measures over a $G$-map.  Recall that:
\begin{definition}
Let $\pi : (X,\nu) \to (Y,\eta)$ be a $G$-map of $G$-spaces.  Then there exist almost surely unique measures $D_{\pi}(y)$, called the \textbf{disintegration measures}, such that $D_{\pi}(y)$ is supported on $\pi^{-1}(y)$ and $\int D_{\pi}(y)~d\eta(y) = \nu$.
\end{definition}

Of course, the disintegration measures correspond to the conditional expectation at the level of the function algebras: if $\pi : (X,\nu) \to (Y,\eta)$ then the algebra $\{ f \circ \pi : f \in L^{\infty}(Y,\eta) \}$ is a subalgebra of $L^{\infty}(X,\nu)$ and the disintegration measures define the conditional expectation to this subalgebra.

\subsection{Continuous Compact Models}

We will need a basic fact about the existence of compact models.  This result does not appear explicitly in the literature but the proof is essentially contained in \cite{Zi84}.

\begin{definition}
Let $(X,\nu)$ be a (measurable) $G$-space.  A compact metric space $X_{0}$ and fully supported Borel probability measure $\nu_{0} \in P(X_{0})$ is a \textbf{continuous compact model} of $(X,\nu)$ when $G$ acts continuously on $X_{0}$ and there exists a $G$-equivariant measure space isomorphism $(X,\nu) \to (X_{0},\nu_{0})$.
\end{definition}

\begin{definition}
Let $\pi : (X,\nu) \to (Y,\eta)$ be a measurable $G$-map of (measurable) $G$-spaces.  Let $X_{0}$ and $Y_{0}$ be compact metric spaces on which $G$ acts continuously and let $\pi_{0} : X_{0} \to Y_{0}$ be a continuous $G$-equivariant map.  Let $\nu_{0} \in P(X_{0})$ and $\eta_{0} \in P(Y_{0})$ be fully supported Borel probability measures such that $(\pi_{0})_{*}\nu_{0} = \eta_{0}$.  The map and spaces $\pi_{0} : (X_{0},\nu_{0}) \to (Y_{0},\eta_{0})$ is a \textbf{continuous compact model} for the $G$-map $\pi$ and $G$-spaces $(X,\nu)$ and $(Y,\eta)$ when there exist $G$-equivariant measure space isomorphisms $\Phi : (X,\nu) \to (X_{0},\nu_{0})$ and $\Psi : (Y,\eta) \to (Y_{0},\eta_{0})$ such that the resulting diagram commutes: $\pi = \Psi^{-1} \circ \pi_{0} \circ \Phi$.
\end{definition}

\begin{lemma}[Varadarjan \cite{vara}]\label{L:compactmodels}
Let $G$ be a locally compact second countable group and $\pi : (X,\nu) \to (Y,\eta)$ a $G$-map of $G$-spaces.  Then there exists a continuous compact model for $\pi$.
\end{lemma}
\begin{proof}
Let $\mathcal{X}$ be a countable collection of functions in $L^{\infty}(X,\nu)$ that separates points and let $\mathcal{Y}$ be a countable collection in $L^{\infty}(Y,\eta)$ that separates points.  Let $\mathcal{F} = \mathcal{X} \cup \{ f \circ \pi : f \in \mathcal{Y} \}$.  Let $B$ be the unit ball in $L^{\infty}(G,\Haar)$ which is a compact metric space in the weak* topology (as the dual of $L^{1}$).

Define $X_{00} = \prod_{f \in \mathcal{F}} B$ and $Y_{00} = \prod_{f \in \mathcal{Y}} B$, both of which are compact metric spaces using the product topology.  Define $\pi_{00} : X_{00} \to Y_{00}$ to be the restriction map: for $f \in \mathcal{Y}$ take the $f^{th}$ coordinate of $\pi_{00}(x_{00})$ to be the $(f \circ \pi)^{th}$ coordinate of $x_{00}$.  Then $\pi_{00}$ is continuous.

Define the map $\Phi : X \to X_{00}$ by $\Phi(x) = (\varphi_{f}(x))_{f \in \mathcal{F}}$ where $(\varphi_{f}(x))(g) = f(gx)$.  Then $\Phi$ is an injective map (since $\mathcal{F}$ separates points).  Observe that $(\varphi_{f}(hx))(g) = f(ghx) = (\varphi_{f}(x))(gh)$.  Consider the $G$-action on $X_{00}$ given by the right action on each coordinate.  Then $G$ acts on $X_{00}$ continuously (and likewise on $Y_{00}$ continuously) and $\Phi$ is $G$-equivariant.  Similarly, define $\Psi : Y \to Y_{00}$ by $\Psi(y) = (\psi_{f}(y))_{f \in \mathcal{Y}}$ where $(\psi_{f}(y))(g) = f(gy)$.

Let $X_{0} = \overline{\Phi(X)}$, let $\nu_{0} = \Phi_{*}\nu$, let $Y_{0} = \overline{\Psi(Y)}$, let $\eta_{0} = \Psi_{*}\eta$ and let $\pi_{0}$ be the restriction of $\pi_{00}$ to $X_{0}$.  Then $\Phi : (X,\nu) \to (X_{0},\nu_{0})$ and $\Psi : (Y,\eta) \to (Y_{0},\eta_{0})$ are $G$-isomorphisms.  Since $(\psi_{f}(\pi(x)))(g) = f(g\pi(x)) = f \circ \pi (gx) = (\varphi_{f \circ \pi}(x))(g)$, $\pi_{0}(X_{0}) = Y_{0}$ and $\Psi^{-1} \circ \pi_{0} \circ \Phi = \pi$.
\end{proof}

\section{Invariant Random Subgroups}\label{S:rs}

Invariant random subgroups are the natural context for the presentation of some of our results.  We recall here the definition and a basic construction of them and introduce a definition of commensurability for invariant random subgroups.

\begin{definition}
Let $G$ be a group.  The space of closed subgroups $S(G)$ is a compact topological space (with the Chabauty topology) and $G$ acts on it by conjugation.  An \textbf{invariant random subgroup} of $G$ is a probability measure $\nu \in P(S(G))$ that is invariant under the conjugation action.
\end{definition}

\subsection{Measure-Preserving Actions}

Let $G$ be a group and $G \actson (X,\nu)$ be a measure-preserving action.  Then the mapping $x \mapsto \mathrm{stab}_{G}(x)$ sending each point to its stabilizer subgroup defines a Borel map $X \to S(G)$ (\cite{moore66} Chapter 2, Proposition 2.3).  Let $\eta$ be the pushforward of $\nu$ under this map.  Observe that $\mathrm{stab}_{G}(gx) = g\hspace{0.1em}\mathrm{stab}_{G}(x)\hspace{0.1em}g^{-1}$ so the mapping is a $G$-map and therefore $\eta$ is an invariant measure on $S(G)$.  Hence $G \actson (X,\nu)$ gives rise in a canonical way to an invariant random subgroup of $G$ defined by the stabilizer subgroups.

\subsection{Invariant Random Subgroups Always Arise From Actions}

In fact the converse of this is also true: any invariant random subgroup can be realized as the stabilizer subgroups of some measure-preserving action:
\begin{theorem}[Abert-Glasner-Vir\'{a}g \cite{AGV12}]\label{T:silly}
Let $\eta \in P(S(G))$ be an invariant random subgroup of a countable group $G$.  Then there exists a measure-preserving $G$-space $(X,\nu)$ such that $\eta$ is the invariant random subgroup that arises from the stabilizers of the action.
\end{theorem}

In our setting, we consider invariant random subgroups of nondiscrete locally compact groups and so we need to generalize the result of Abert, Glasner and Vir\'{a}g to the locally compact case (see also \cite{seven} Theorem 2.4).
We make use of the Gaussian action construction: for a separable Hilbert space $H$ one can associate a probability space $(Y_{H},\nu_{H})$ and an embedding $\rho : H \to L^{2}(Y_{H},\nu_{H})$ such that for any orthogonal $T : H \to K$ between Hilbert spaces there is a measure-preserving map $V_{T} : (Y_{H},\nu_{H}) \to (Y_{K}, \nu_{K})$ such that $\rho(T(\xi)) = \rho(\xi) \circ V_{T}^{-1}$ and that for $T : H \to K$ and $S: K \to L$, $V_{S} \circ V_{T} = V_{S \circ T}$ almost everywhere for each fixed pair $S,T$.  The reader is referred to Schmidt \cite{schmidtcohomology} Section 3 for details.
\begin{theorem}\label{T:gaussian}
Let $G$ be a locally compact second countable group.  Given an invariant random subgroup $(S(G),\eta)$ there exists a measure-preserving $G$-space $(X,\nu)$ such that the $G$-equivariant mapping $x \mapsto \mathrm{stab}_{G}(x)$ pushes $\nu$ to $\eta$.
\end{theorem}
\begin{proof}
Fix a probability measure $\rho$ on $G$ in the class of Haar measure.
For each $H \in S(G)$ let $(Y_{H},\eta_{H})$ be the Gaussian probability space corresponding to an infinite direct sum of $L^{2}(G/H,\rho_{H})$ where $\rho_{H}$ is the pushforward of $\rho$ via the quotient map $q_{H} : G \to G/H$.  Note that $L^{2}(G/H,\rho_{H})$ gives a measurable field of Hilbert spaces where a measurable sequence of vector fields is given by $\zeta_{n}(H) = \bbone_{O_{n}H}$ where $\{ O_{n} \}$ is a countable basis for the topology of $G$ (see Folland \cite{folland} Section 7.4).
Let $Y = ((Y_{H},\eta_{H}))_{H \in S(G)}$ be the field of measure spaces just constructed.

Define the cocycle $\alpha : G \times S(G) \to Y$ such that $\alpha(g,H) \in \mathrm{Aut}(Y_{H},Y_{gHg^{-1}})$ to be the measure-preserving isomorphism from $Y_{H}$ to $Y_{gHg^{-1}}$ induced by the orthogonal operator $T_{g,H}$ from the infinite direct sum of $L^{2}(G/H,\rho_{H})$ to the direct sum of $L^{2}(G/gHg^{-1},\rho_{gHg^{-1}})$ given by $(T_{g,H}f)(kgHg^{-1}) = f(kgH) \sqrt{\frac{d(q_{H})_{*}(\rho g^{-1})}{d\rho_{H}}(kgH)}$.  For each $g,h \in G$, the cocycle identity holds almost everywhere by the nature of the Gaussian construction.  Define the measure space
\[
(X,\nu) = \big{(}\bigsqcup Y_{H}, \int \eta_{H}~d\eta(H)\big{)}
\]
equipped with the measure-preserving cocycle action of $G$ coming from $\alpha$.  By Mackey's point realization \cite{Ma62} (see Appendix B of \cite{Zi84}), as $G$ is locally compact second countable by removing a null set we may assume, the cocycle identity holds everywhere.

For each fixed $H \in S(G)$ the map $g \mapsto \alpha(g,H)$ defines an action of $N_{G}(H)/H$ on $(Y_{H},\eta_{H})$ which is essentially free (as per the proof of Proposition 1.2 in \cite{AEG}).  For $g \in G$ and $(H,x) \in X$ we see that $g(H,x) = (gHg^{-1},\alpha(g,H)x)$ and therefore $(H,x) = g(H,x)$ if and only if $g \in N_{G}(H)$ and $\alpha(g,H)x = x$ hence if and only if $g \in H$.  That is to say, $\mathrm{stab}_{G}(H,x) = H$ for almost every $(H,x)$.  Therefore the $G$-action on $(X,\nu)$ gives rise to the invariant random subgroup $\eta$ as required.
\end{proof}

\subsection{Ergodic Invariant Random Subgroups}

\begin{definition}
An invariant random subgroup $\nu \in P(S(G))$ is \textbf{ergodic} when $\nu$ is an ergodic measure.
\end{definition}

We remark that ergodic invariant random subgroups are precisely the same as the extremal invariant measures in the (weak*) compact convex set of invariant random subgroups.

\begin{proposition}\label{P:gaussian}
Let $G$ be a locally compact second countable group.  Given an ergodic invariant random subgroup $(S(G),\eta)$ there exists an ergodic measure-preserving $G$-space $(X,\nu)$ such that the $G$-equivariant mapping $x \mapsto \mathrm{stab}_{G}(x)$ pushes $\nu$ to $\eta$.
\end{proposition}
\begin{proof}
Let $(Z,\zeta)$ be the $G$-action constructed in Theorem \ref{T:gaussian} such that $z \mapsto \mathrm{stab}_{G}(z)$ pushes $\zeta$ to $\eta$.  Consider the ergodic decomposition $\pi : (Z,\zeta) \to (W,\rho)$.  Then $G$ acts trivially on $(W,\rho)$ and almost every fiber $(\pi^{-1}(w), D_{\pi}(w))$ is an ergodic $G$-space.  Observe that
\[
\int_{W} \mathrm{stab}_{*}D_{\pi}(w)~d\rho(w) = \mathrm{stab}_{*} \int_{W} D_{\pi}(w)~d\rho(w) = \mathrm{stab}_{*}\zeta = \eta.
\]
Since $\eta$ is ergodic, it is extremal in the set of invariant random subgroups.  The above convex combination of invariant random subgroups must then almost surely be constantly equal to $\eta$.  That is, $\mathrm{stab}_{*}D_{\pi}(w) = \eta$ for $\rho$-almost every $w \in W$.  Let $(X,\nu)$ be one such fiber.  Then $(X,\nu)$ is an ergodic $G$-space with the required properties.
\end{proof}

\subsection{Compact and Open Invariant Random Subgroups}

\begin{definition}
An invariant random subgroup $\nu \in P(S(G))$ is a \textbf{finite (compact) invariant random subgroup} when $\nu$ is supported on the finite (compact) subgroups of $G$ and is an \textbf{infinite (noncompact) invariant random subgroup} when it is supported on the infinite (noncompact) subgroups of $G$.
\end{definition}

We remark that in the case of ergodic invariant random subgroups, infinite is equivalent to not finite.

\begin{definition}
An invariant random subgroup $\nu \in P(S(G))$ is an \textbf{open invariant random subgroup} when $\nu$ is supported on the open subgroups of $G$.
\end{definition}

\subsection{Commensurate Invariant Random Subgroups}\label{S:commclassesrs}

Recall that two subgroups are commensurate when their intersection has finite index in each.  We introduce a definition of commensurability for invariant random subgroups that generalizes this notion to invariant random subgroups.  We remark that one regains the usual definition in the case when the invariant random subgroups are point masses.  Before introducing the definition, we recall the notion of joinings of $G$-spaces.

\begin{definition}
Let $(X,\nu)$ and $(Y,\eta)$ be $G$-spaces.  Let $\alpha \in P(X \times Y)$ such that $(\mathrm{pr}_{X})_{*}\alpha = \nu$, $(\mathrm{pr}_{Y})_{*}\alpha = \eta$ and $\alpha$ is quasi-invariant under the diagonal $G$-action.  The space $(X \times Y, \alpha)$ with the diagonal $G$-action is called a \textbf{joining} of $(X,\nu)$ and $(Y,\eta)$.
A joining $\alpha$ of $G$-spaces is \textbf{$G$-invariant} when $\alpha$ is $G$-measure-preserving under the diagonal action.
\end{definition}

%

\begin{definition}[\cite{glasner} Definition 6.9]
Let $(X,\nu)$, $(Y,\eta)$ and $(Z,\zeta)$ be $G$-spaces and let $\alpha$ be a joining of $(X,\nu)$ and $(Y,\eta)$ and $\beta$ be a joining of $(Y,\nu)$ and $(Z,\zeta)$.  Let $\alpha_{y} \in P(X)$ and $\beta_{y} \in P(Z)$ be the projections of the disintegrations of $\alpha$ and $\beta$ over $\eta$.  The measure $\rho \in P(X \times Z)$ by
\[
\rho = \int_{Y} \alpha_{y} \times \beta_{y}~d\eta(y)
\]
is the \textbf{composition} of $\alpha$ and $\beta$.
\end{definition}

\begin{proposition}[\cite{glasner} Proposition 6.10]\label{P:composejoining}
The composition of two joinings is a joining.  If two joinings are $G$-invariant then so is their composition.
\end{proposition}

We can now state the definition of commensurate invariant random subgroups:
\begin{definition}\label{def:commrs}
Let $G$ be a group and $\eta_{1}$ and $\eta_{2}$ be invariant random subgroups of $G$.  If there exists a $G$-invariant joining $\alpha \in P(S(G) \times S(G))$ of $\eta_{1}$ and $\eta_{2}$ such that for $\alpha$-almost every $H,L \in S(G) \times S(G)$ the subgroup $H \cap L$ has finite index in both $H$ and $L$ then $\eta_{1}$ and $\eta_{2}$ are \textbf{commensurate invariant random subgroups}.  The \textbf{commensurability classes of invariant random subgroups} of $G$ are the classes of invariant random subgroups equivalent under commensuration (see Theorem \ref{T:commequiv}).
\end{definition}

\begin{theorem}\label{T:commequiv}
The property of being commensurate is an equivalence relation on the space of invariant random subgroups.
\end{theorem}
\begin{proof}
Let $\eta_{1},\eta_{2},\eta_{3}$ be invariant random subgroups of $G$ such that $\eta_{1}$ and $\eta_{2}$ are commensurate and $\eta_{2}$ and $\eta_{3}$ are commensurate.  Let $\alpha$ be a joining of $\eta_{1}$ and $\eta_{2}$ and $\beta$ be a joining of $\eta_{2}$ and $\eta_{3}$ witnessing the commensuration.  Let $D$ be the disintegration of $\alpha$ over $\eta_{2}$.  Then for almost every $K \in S(G)$ we have that $D(K) = \alpha_{K} \times \delta_{K}$ for some $\alpha_{K} \in P(S(G))$ and likewise the disintegration of $\beta$ over $\eta_{2}$ is of the form $\delta_{K} \times \beta_{K}$ for some $\beta_{K} \in P(S(G))$.

Let $\rho \in P(S(G) \times S(G))$ be the composition of the joinings $\alpha$ and $\beta$ (see Glasner \cite{glasner}):
\[
\rho = \int_{S(G)} \alpha_{K} \times \beta_{K}~d\eta_{2}(K).
\]
Then $\rho$ is a joining of $\eta_{1}$ and $\eta_{3}$ (Proposition \ref{P:composejoining}) and
for $\rho$-almost every $(H,L)$ we have that for $\eta_{2}$-almost every $K$ the subgroup $H \cap K$ has finite index in $H$ and $K$ and the subgroup $K \cap L$ has finite index in both $K$ and $L$.  Then $H \cap K \cap L$ has finite index in $H$, $K$ and $L$ and so $H \cap L$ has finite index in $H$ and $L$ (that is, commensurability is an equivalence relation on subgroups).  Therefore $\rho$ makes $\eta_{1}$ and $\eta_{3}$ commensurate invariant random subgroups.
\end{proof}

\section{Relatively Contractive Maps}\label{S:relativecontractive}

We now introduce the notion of relatively contractive maps and develop the machinery needed to study actions of commensurators and lattices.  We first spend some time developing basic facts about relatively contractive maps which we then use to recover most known results about contractive actions.  We also take a short detour to study joinings of contractive spaces and show that relatively contractive is indeed the opposite of relatively measure-preserving in some very strong senses.

We will always assume the group $G$ is locally compact second countable in what follows.

\begin{definition}[Jaworski \cite{Ja94}]
A $G$-space $(X,\nu)$ is \textbf{contractive}, also called \textbf{SAT (strongly approximately transitive)}, when for all measurable sets $B \subseteq X$ of less than full measure and all $\epsilon > 0$ there exists $g \in G$ such that
\[
\nu(gB) < \epsilon.
\]
\end{definition}

\subsection{Conjugates of Disintegration Measures}

The principal notion in formulating the idea of relatively contractive maps is to ``conjugate" the disintegration measures.
For a $G$-map of $G$-spaces $\pi : (X,\nu) \to (Y,\eta)$, the disintegration of $\nu$ over $\eta$ can be summarized as saying that for almost every $y \in Y$ there is a unique measure $D_{\pi}(y) \in P(X)$ such that $D_{\pi}(y)$ is supported on the fiber over $y$ and $\int_{Y} D_{\pi}(y)~d\eta(y) = \nu$.

For $g \in G$ and $y \in Y$, we have that $D_{\pi}(gy)$ is supported on the fiber over $gy$, that is, on $\pi^{-1}(gy) = g \pi^{-1}(y)$, and that for any Borel $B \subseteq X$, we have that $gD_{\pi}(y)(B) = D_{\pi}(y)(g^{-1}B)$
meaning that $gD_{\pi}(y)$ is supported on $g\pi^{-1}(y)$.  Therefore we can formulate the following:
\begin{definition}
Let $\pi : (X,\nu) \to (Y,\rho)$ be a $G$-map of $G$-spaces.
The \textbf{conjugated disintegration measure} over $\pi$ at a point $y \in Y$ by the group element $g \in G$ is
\[
D_{\pi}^{(g)}(y) = g^{-1}D_{\pi}(gy).
\]
\end{definition}

The preceding discussion shows that $D_{\pi}^{(g)}(y)$ is supported on $g^{-1} g \pi^{-1}(y) = \pi^{-1}(y)$.  Hence:
\begin{proposition}\label{P:conjugateddisintegration}
Let $\pi : (X,\nu) \to (Y,\eta)$ be a $G$-map of $G$-spaces and fix $y \in Y$.  The conjugated disintegration measures
\[
\mathcal{D}_{y} = \{ g^{-1}D_{\pi}(gy) : g \in G \}
\]
are all supported on $\pi^{-1}(y)$.
\end{proposition}

Another approach to the conjugates of disintegration measures is to observe that:
\begin{proposition}
Let $\pi : (X,\nu) \to (Y,\eta)$ be a $G$-map of $G$-spaces.  For any $g \in G$ then $\pi : (X, g^{-1}\nu) \to (Y, g^{-1}\eta)$ is also a $G$-map of $G$-spaces.  Let $D_{\pi} : Y \to P(X)$ be the disintegration of $\nu$ over $\eta$.  Then $D_{\pi}^{(g)}$ is the disintegration of $g^{-1}\nu$ over $g^{-1}\eta$.
\end{proposition}

\begin{proof}
To see that $\pi$ maps $(X,g^{-1}\nu)$ to $(Y,g^{-1}\eta)$ follows from $\pi$ being $G$-equivariant.

We have already seen that $g^{-1} D_{\pi}(gy)$ is supported on $\pi^{-1}(y)$ so to prove the proposition it remains only to show that $\int g^{-1} D_{\pi}(gy)~dg^{-1}\eta(y) = g^{-1}\nu$.  This is clear as
\begin{align*}
\int_{Y} g^{-1} D_{\pi}(gy)~dg^{-1}\eta(y) &= g^{-1} \int_{Y} D_{\pi}(g g^{-1} y)~d\eta(y) \\ &= g^{-1} \int_{Y} D_{\pi}(y)~d\eta(y) = g^{-1}\nu
\end{align*}
since $D_{\pi}$ disintegrates $\nu$ over $\eta$.
\end{proof}

A basic fact we will need in what follows is that  the conjugated disintegration measures are mutually absolutely continuous to one another (over a fixed point $y$ of course, as $y$ varies they have disjoint supports):
\begin{proposition}
Let $\pi : (X,\nu) \to (Y, \eta)$ be a $G$-map of $G$-spaces.  For almost every $y$ the set
\[
\mathcal{D}_{y} = \{ g^{-1}D_{\pi}(gy) : g \in G \}
\]
is a collection of mutually absolutely continuous probability measures supported on $\pi^{-1}\{ y \}$.
\end{proposition}
\begin{proof}
For $g \in G$ write
\[
A_{g} = \{ y \in Y : D_{\pi}(y)~\text{and}~g^{-1}D_{\pi}(gy) \text{ are not in the same measure class} \}.
\]
Then $A_{g}$ is a Borel set for each $g \in G$ since $D_{\pi} : Y \to P(X)$ is a Borel map and the equivalence relation on $P(X)$ given by $\alpha \sim \beta$ if and only if $\alpha$ and $\beta$ is in the same measure class is Borel.

Since $g^{-1}D_{\pi}(gy)$ is the disintegration of $g^{-1}\nu$ over $g^{-1}\eta$ and $g^{-1}\nu$ is in the same measure class as $\nu$, Lemma \ref{L:abscontmeas} (following the proof) gives that $\eta(A_{g}) = 0$ for each $g \in G$.
Therefore
\[
\eta(\bigcup_{g \in G_{0}} A_{g}) = 0
\]
where the union is taken over a countable dense subset $G_{0}$ (the existence of such a subset is a consequence of the second countability of $G$).  When $G$ is itself countable the claim is then proven.

Suppose now that there is some $g$ such that $\eta(A_{g}) > 0$.  Take a continuous compact model for $\pi$ via Lemma \ref{L:compactmodels}.  Define the sets, for $g \in G$ and $\epsilon > 0$ and $f \in C(X)$ with $f \geq 0$,
\[
A_{g,\epsilon,f} = \{ y \in Y : D_{\pi}(y)(f) = 0, D_{\pi}^{(g)}(y)(f) \geq \epsilon \}.
\]
These sets are Borel since $y \mapsto D_{\pi}(y)(f)$ is Borel.
Now $A_{g} = \bigcup_{\epsilon > 0,f} A_{g,\epsilon,f}$ and since $\eta(A_{g}) > 0$, (taking a countable sequence $\epsilon \to 0$ and a countable dense set of $C(X)$) there is some $\epsilon > 0$ and $f \in C(X)$ with $f \geq 0$ such that
\[
\eta(A_{g,\epsilon,f}) > 0.
\]
But now for fixed $\epsilon > 0$ and $f \in C(X)$ with $f \geq 0$ we see that
\[
g^{-1}\nu(f) \geq \int_{A_{g,\epsilon,f}} D_{\pi}^{(g)}(y)(f)~dg^{-1}\eta(y) \geq g^{-1}\eta(A_{g,\epsilon,f})\epsilon > 0
\]
by the quasi-invariance of $\eta$.  

Consider the function $F : G \to \mathbb{R}$ given by
\[
F(h) = \int_{A_{g,\epsilon,f}} D_{\pi}^{(h)}(y)(f)~dh^{-1}\eta(y) = h^{-1}\nu(\bbone_{\pi^{-1}(A_{g,\epsilon,f})} f).
\]
Then $F(g) > 0$ by the above.  Now $F$ is continuous since $f \in C(X)$ and $G \actson X$ continuously.  Hence there is some open neighborhood $U$ of $g$ in $G$ such that $F(u) > 0$ for all $u \in U$.  

For $g_{0} \in G_{0}$, however, we know that $g_{0}^{-1}\nu(f) = 0$ and so, as $f \geq 0$, then $F(g_{0}) = 0$.
But $U \cap G_{0} \ne \emptyset$ since $G_{0}$ is dense and $U$ is open, leading to a contradiction.  Hence when $G$ is locally compact second countable the claim also holds.
\end{proof}

\begin{lemma}\label{L:abscontmeas}
Let $(X,\nu)$ be a probability space and $\pi : (X,\nu) \to (Y,\pi_{*}\nu)$ a measurable map to a probability space.  Let $\alpha$ be a probability measure in the same measure class as $\nu$.  Let $D(y)$ denote the disintegration of $\nu$ over $\pi_{*}\nu$ via $\pi$ and let $D^{\prime}(y)$ denote the disintegration of $\alpha$ over $\pi_{*}\alpha$ via $\pi$.  Then for almost every $y \in Y$, $D(y)$ and $D^{\prime}(y)$ are in the same measure class.
\end{lemma}
\begin{proof}
Since $\alpha$ and $\nu$ are in the same measure class, the Radon-Nikodym derivative $\frac{d\alpha}{d\nu}$ exists and is in $L^{1}(X,\nu)$.  Likewise, $\pi_{*}\alpha$ and $\pi_{*}\nu$ are in the same measure class so $\frac{d\pi_{*}\nu}{d\pi_{*}\alpha}$ exists in $L^{1}(X,\pi_{*}\nu)$.

For $y \in Y$, define the measure $\alpha_{y}$ by, for $B \subseteq X$ measurable,
\[
\alpha_{y}(B) = \int_{B} \frac{d\alpha}{d\nu}(x)~dD(y)(x)~\frac{d\pi_{*}\nu}{d\pi_{*}\alpha}(y).
\]
Note that the Radon-Nikodym derivatives are always positive so these are positive measures.  Also $\alpha_{y}(X) = 1$ since
\[
\frac{d\pi_{*}\alpha}{d\pi_{*}\nu}(y) = \int_{X} \frac{d\alpha}{d\nu}(x)~dD(y)(x)
\]
which can be verified directly (using the uniqueness of the Radon-Nikodym derivative).

Now the support of $\alpha_{y}$ is contained in the support of $D(y)$ which is contained in $\pi^{-1}(y)$, hence $\alpha_{y}$ is supported on $\pi^{-1}(y)$.  For $B \subseteq X$ measurable,
\begin{align*}
\int_{Y} \alpha_{y}(B)~d\pi_{*}\alpha(y)
&= \int_{Y} \int_{B} \frac{d\alpha}{d\nu}(x)~dD(y)(x)~\frac{d\pi_{*}\nu}{d\pi_{*}\alpha}(y)~d\pi_{*}\alpha(y) \\
&= \int_{Y} \int_{X} \bbone_{B}(x) \frac{d\alpha}{d\nu}(x)~dD(y)(x)~d\pi_{*}\nu(y) \\
&= \int_{X} \bbone_{B}(x) \frac{d\alpha}{d\nu}(x)~d\nu(x) \\
&= \int_{X} \bbone_{B}(x)~d\alpha(x) = \alpha(B).
\end{align*}
Therefore, by uniqueness of disintegration, $D^{\prime}(y) = \alpha_{y}$ for almost every $y$.

Suppose that $D(y)(B) = 0$ for some $y$ and some measurable $B \subseteq X$.  Then 
\[
\alpha_{y}(B) = \int_{B} \frac{d\alpha}{d\nu}(x)~dD(y)(x)~\frac{d\pi_{*}\nu}{d\pi_{*}\alpha}(y) = 0
\]
since $D(y)(B) = 0$.  So $\alpha_{y}$ is absolutely continuous with respect to $D(y)$.

Therefore $D^{\prime}(y)$ is absolutely continuous with respect to $D(y)$ for almost every $y \in Y$.  The symmetric argument (reversing the roles of $\nu$ and $\alpha$) shows that $D(y)$ is also absolutely continuous with respect to $D(y)$ almost everywhere.
\end{proof}

%
%
%
%
%

\subsection{Definition of Relatively Contractive Maps}

We now define relatively contractive factor maps, which are the counterpart of relatively measure-preserving factor maps.

\subsubsection{Relatively Measure-Preserving}

We first recall the definition of relative measure-preserving:
\begin{definition}
Let $\pi : (X,\nu) \to (Y,\eta)$ be a $G$-map of $G$-spaces.  Then $\pi$ is \textbf{relatively measure-preserving} when for almost every $y \in Y$ the disintegration map $D_{\pi}$ is $G$-equivariant: $D_{\pi}(gy) = gD_{\pi}(y)$.
\end{definition}

In terms of conjugating disintegration measures, relative measure-preserving means that $D_{\pi}^{(g)}(y) = D_{\pi}(y)$ almost surely.

We also remark that a $G$-space $(X,\nu)$ is measure-preserving if and only if the map from $(X,\nu)$ to the trivial (one-point) space is relatively measure-preserving (the disintegration over the trivial space is $D^{(g)}(0) = g^{-1}D(g \cdot 0) = g^{-1}D(0) = g^{-1}\nu$).

\subsubsection{Relatively Contractive}

\begin{definition}
Let $\pi : (X, \nu) \to (Y, \eta)$ be a $G$-map of $G$-spaces.  We say $\pi$ is \textbf{relatively contractive} when for almost every $y \in Y$ and any measurable $B \subseteq X$ with $D_{\pi}(y)(B) < 1$ and any $\epsilon > 0$ there exists $g \in G$ such that $g^{-1}D_{\pi}(gy)(B) < \epsilon$.
\end{definition}

This is also stated as saying that $(X,\nu)$ is a \textbf{relatively contractive extension} or \textbf{contractive extension} of $(Y,\eta)$ or that $(Y,\eta)$ is a \textbf{relatively contractive factor} or just a \textbf{contractive factor} of $(X,\nu)$.

We have the following easy reformulation of the above definition:
\begin{proposition}
A $G$-map $\pi : (X,\nu) \to (Y,\eta)$ of $G$-spaces is relatively contractive if and only if for almost every $y$ and any measurable $B \subseteq Y$ with $D_{\pi}(y)(B) > 0$ we have
\[
\sup_{g \in G} D_{\pi}^{(g)}(y)(B) = 1.
\]
\end{proposition}

\subsubsection{Contractive Extensions of a Point}

We now show that contractive can be defined in terms of relatively contractive extensions of a point (just as measure-preserving can be defined as being a relatively measure-preserving extension of a point).

\begin{theorem}
A $G$-space $(X, \nu)$ is contractive if and only if it is a relatively contractive extension of a point.
\end{theorem}
\begin{proof}
In the case where $(Y, \eta) = 0$ is the trivial one point system, the disintegration measure is always $\nu$ and so being a relatively contractive extension reduces to the definition of contractive: $g^{-1}D_{\pi}(g \cdot 0) = g^{-1}\nu$ for all $g \in G$ since $g \cdot 0 = 0$ and therefore $\sup_{g} D_{\pi}^{(g)}(0)(B) = 1$ implies $\sup_{g} g^{-1}D_{\pi}(0)(B) = 1$ so $\sup_{g} g^{-1}\nu(B) = 1$ for all measurable $B$ with $\nu(B) > 0$.
\end{proof}

\subsection{The Algebraic Characterization}

Generalizing Jaworksi \cite{Ja94}, we characterize relatively contractive maps algebraically:
\begin{theorem}\label{T:contractiveexte}
Let $\pi : (X, \nu) \to (Y, \rho)$ be a $G$-map of $G$-spaces.  Then $\pi$ is relatively contractive if and only if the map $f \mapsto D_{\pi}^{(g)}(y)(f)$ is an isometry between $L^{\infty}(X, D_{\pi}(y))$ and $L^{\infty}(G, \mathrm{Haar})$ for almost every $y \in Y$ (here $D_{\pi}^{(g)}(y)(f)$ is a function of $g$).
\end{theorem}
\begin{proof}
Assume $\pi$ is relatively contractive.  Take $y$ in the measure one set where the disintegration measures are relatively contractive.  Let $f$ be any simple function $f = \sum a_{n} \bbone_{B_{n}}$ with $B_{n} \subseteq \pi^{-1}(y)$.  Choose $N$ such that $|a_{N}| = \max_{n} |a_{n}| = \| f \|_{\infty}$.  For $\epsilon > 0$ choose $g \in G$ such that $D_{\pi}^{(g)}(y)(B_{N}) > 1 - \epsilon$.  Then $D_{\pi}^{(g)}(y)(B_{N}^{C}) < \epsilon$ and since the $B_{n}$ are disjoint then $D_{\pi}^{(g)}(y)(B_{n}) < \epsilon$ for $n \ne N$.  This means that
\[
\big{|}D_{\pi}^{(g)}(y)(f) - a_{N}\big{|} = \big{|} \sum_{n} a_{n} D_{\pi}^{(g)}(y)(B_{n}) - a_{N}\big{|}
\leq \sum_{n \ne N} |a_{n}| \epsilon + |a_{N}| |1 - \epsilon - 1| = \epsilon \sum_{n} |a_{n}|
\]
and since $\epsilon > 0$ was arbitrary then $\sup_{g} |D_{\pi}^{(g)}(y)(f)| = |a_{N}| = \| f \|$.  As simple functions are uniformly dense in $L^{\infty}(X,D_{\pi}(y))$ and the map is a contraction this proves one direction.

Conversely, assume the map is an isometry for almost every $y$.  For such a $y$, let $B \subseteq \pi^{-1}(y)$ with $D_{\pi}(y)(B) > 0$ and then
$1 = \| \bbone_{B} \|_{\infty} = \sup_{g} D_{\pi}^{(g)}(y)(B)$
so $\pi$ relatively contractive.
\end{proof}

Note that $\pi$ is relatively measure-preserving if and only if the map that would be isometric for relatively contractive, $f \mapsto D_{\pi}^{(g)}(y)(f)$, is simply the map $f \mapsto D_{\pi}(y)(f)$ which is the projection to the ``constants" on each fiber.

We remark that in effect there is a zero-one law for relatively contractive extensions.  Namely, if $\pi : (X,\nu) \to (Y,\eta)$ is a $G$-map of ergodic $G$-spaces then the set of $y$ such that $D_{\pi}^{(g)}(y)$ induces an isometry $L^{\infty}(X,D_{\pi}(y)) \to L^{\infty}(G,\mathrm{Haar})$ has either measure zero or measure one.  This follows from the fact that the set of such $y$ must be $G$-invariant and hence follows by ergodicity: if $D_{\pi}^{(g)}(y)$ induces an isometry then for any $h \in G$ and $f \in L^{\infty}(X,\nu)$
\[
\sup_{g \in G} \big{|}D_{\pi}^{(g)}(hy)(f)\big{|} = \sup_{g \in G} \big{|}D_{\pi}^{(gh)}(y)(h \cdot f)\big{|}
= \sup_{g \in G} \big{|}D_{\pi}^{(g)}(y)(h \cdot f)\big{|} = \| h \cdot f \| = \| f \|.
\]

Specializing to the case of a contractive extension of the trivial one point system we obtain:
\begin{corollary}[Jaworski \cite{Ja94}]
A $G$-space $(X,\nu)$ is contractive if and only if the map $L^{\infty}(X,\nu) \to L^{\infty}(G,\mathrm{Haar})$ given by $f \mapsto g\nu(f)$ is an isometry.
\end{corollary}

One can also characterize relatively contractive maps in terms of convex combinations of measures:
\begin{theorem}
A $G$-map of $G$-spaces $\pi : (X,\nu) \to (Y,\eta)$ is relatively contractive if and only if for almost every $y \in Y$ the space of absolutely continuous measures $L_{1}^{1}(\pi^{-1}(y),D_{\pi}(y)) \subseteq \overline{\mathrm{conv}}~\mathcal{D}_{y}$.
\end{theorem}
\begin{proof}
An immediate consequence of Theorem \ref{T:contractiveexttopo} (in the following subsection).
\end{proof}

Specializing to the one point system:
\begin{corollary}[Jaworski \cite{Ja94}]
A $G$-space $(X,\nu)$ is contractive if and only if the space of absolutely continuous measures $L_{1}^{1}(X,\nu) \subseteq \overline{\mathrm{conv}}~G\nu$.
\end{corollary}

\subsection{Relatively Contractible Spaces}

\begin{definition}[Furstenberg-Glasner \cite{FG10}]
A continuous compact model $(X_{0},\nu_{0})$ of a $G$-space $(X,\nu)$ is \textbf{contractible} when for every $x \in X_{0}$ there exists $g_{n} \in G$ such that $g_{n}\nu_{0} \to \delta_{x}$ in weak*.
\end{definition}

\begin{definition}
Let $\pi : (X,\nu) \to (Y,\eta)$ be a $G$-map of $G$-spaces.  A continuous compact model $\pi_{0} : (X_{0},\nu_{0}) \to (Y_{0},\eta_{0})$ for this map is \textbf{relatively contractible} when for $\eta_{0}$-almost every $y \in Y_{0}$ and every $x \in X_{0}$ such that $\pi_{0}(x) = y$ there exists a sequence $g_{n} \in G$ such that $D_{\pi_{0}}^{(g_{n})}(y) \to \delta_{x}$ in weak*.
\end{definition}

\begin{theorem}\label{T:contractiveexttopo}
Let $\pi : (X,\nu) \to (Y,\eta)$ be a $G$-map of $G$-spaces.  Then $\pi$ is relatively contractive if and only if  every continuous compact model of $\pi$ is relatively contractible.
\end{theorem}
\begin{proof}
Recall that a continuous compact model for $\pi$ means compact models for $X$ and $Y$ such that $G \actson X$ and $G \actson Y$ are continuous and the map $\pi$ is continuous (Lemma \ref{L:compactmodels}).

Assume that $\pi$ is relatively contractive.  By Theorem \ref{T:contractiveexte}, there is a measure one set of $y$ such that $f \mapsto D_{\pi}^{(g)}(y)(f)$ is an isometry between $L^{\infty}(X, D_{\pi}(y))$ and $L^{\infty}(G,\mathrm{Haar})$. Fix $y$ in that set and let $x \in X$ such that $\pi(x) = y$.  Choose $f_{n} \in C(X)$ such that $0 \leq f_{n} \leq 1$, $\| f_{n} \| = 1$ and $f_{n} \to \bbone_{\{x\}}$ (possible since $C(X)$ separates points) and such that $f_{n+1} \leq f_{n}$.  Since $\pi$ is relatively contractive, $\sup_{g} D_{\pi}^{(g)}(y)(f_{n}) = 1$ for each $n$.  Choose $g_{n} \in G$ such that
\[
1 - \frac{1}{n} < D_{\pi}^{(g_{n})}(y)(f_{n})
\]
and observe then that, since $f_{n+1} \leq f_{n}$,
\[
1 - \frac{1}{n+1} < D_{\pi}^{(g_{n+1})}(y)(f_{n+1}) \leq D_{\pi}^{(g_{n+1})}(y)(f_{n})
\]
and therefore $\lim_{m\to\infty} D_{\pi}^{(g_{m})}(f_{n}) = 1$ for each fixed $n$.

Now $P(X)$ is compact so there exists a limit point $\zeta \in P(X)$ such that $\zeta = \lim_{j} D_{\pi}^{(g_{n_{j}})}(y)$ along some subsequence.  Now $\zeta(f_{n}) = 1$ for each $n$ by the above and $f_{n} \to \bbone_{\{x\}}$ is pointwise decreasing so by bounded convergence $\zeta(\{ x \}) = \lim \zeta(f_{n}) = 1$.  This means that for almost every $y$, the conclusion holds for all $x \in \pi^{-1}(y)$.

For the converse, first consider any continuous compact model such that for almost every $y \in Y$ and every $x \in \pi^{-1}(y)$ there exists a sequence $\{ g_{n} \}$ such that $D_{\pi}^{(g_{n})}(y) \to \delta_{x}$.  Let $f \in C(X)$.  Then the supremum of $f$ on $\pi^{-1}(y)$ is attained at some $x \in \pi^{-1}(y)$ since $\pi^{-1}(y)$ is a closed, hence compact, set.  Take $g_{n}$ such that $g_{n}^{-1} D_{\pi}(g_{n}y) \to \delta_{x}$.   Then $g_{n}^{-1}D_{\pi}(g_{n}y)(f) \to f(x) = \| f \|_{L^{\infty}(\pi^{-1}(y))}$.  Hence for $f \in C(X)$ the map is an isometry.

Now assume that for every continuous compact model for $\pi$ and for almost every $y$ and every $x \in \pi^{-1}(y)$ there is a sequence $g_{n} \in G$ such that $g_{n}^{-1}D_{\pi}(g_{n}y) \to \delta_{x}$.

Suppose that $\pi$ is not relatively contractive.  Then there exists a measurable set $A \subseteq X$ with $\nu(A) > 0$ and $1 > \delta > 0$ such that 
\[
B = \{ y \in Y : D_{\pi}(y)(A) > 0 \text{ and } \sup_{g} D_{\pi}^{(g)}(y)(A) \leq 1 - \delta \} > 0
\]
has $\eta(B) > 0$.

Fix $\epsilon > 0$.  Let $\psi_{n} \in C_{c}(G)$ be an approximate identity ($\psi_{n}$ are nonnegative continuous functions with $\int \psi_{n} dm = 1$ where $m$ is a Haar measure on $G$ such that the compact supports of the $\psi_{n}$ are a decreasing sequence and $\cap_{n} \mathrm{supp}~\psi_{n} = \{ e \}$; the reader is referred to \cite{FG10} Corollary 8.7).  Define $f_{n} = \bbone_{A} * \psi_{n} = \int_{G} \bbone_{A}(hx) \psi_{n}(h)~dm(h)$.  Then the $f_{n}$ are $G$-continuous functions by \cite{FG10} Lemma 8.6.

By Proposition \ref{P:fghelp} (below),
\[
\lim_{n} \| \bbone_{A} * \psi_{n} \|_{L^{\infty}(X,D_{\pi}(y))} = 1
\]
for all $y \in B$.

There then exists a set $B_{1} \subseteq B$ with $\eta(B_{1}) > \eta(B) - \epsilon$ and $N \in \mathbb{N}$ such that for all $y \in B_{1}$ and all $n \geq N$, $\| \bbone_{A} * \psi_{n} \|_{L^{\infty}(X,D_{\pi}(y))} > 1 - \epsilon$.
Let $V$ be a compact set neighborhood of the identity in $G$ such that $| \eta(B_{1} \cap h^{-1}B_{1}) - \eta(B_{1})| < \epsilon$ for all $h \in V$ (possible as the $G$-action is continuous on the algebra of measurable sets).  Choose $n \geq N$ such that the support of $\psi = \psi_{n}$ is contained in $V$.

Set $f = \bbone_{A} * \psi$.
Since $f$ is $G$-continuous there exists a continuous compact model on which $f \in C(X)$ by \cite{FG10} Theorem 8.5.  Hence, for almost every $y \in Y$,
\[
\sup_{g} D_{\pi}^{(g)}(y)(f) = \| f \|_{L^{\infty}(X,D_{\pi}(y))}.
\]
Removing a null set from $B_{1}$, then for all $y \in B_{1}$ there exists $g_{y} \in G$ such that
\[
D_{\pi}^{(g_{y})}(y)(f) > \| f \|_{L^{\infty}(X,D_{\pi}(y))} - \epsilon > 1 - 2\epsilon.
\]
Observe that
\begin{align*}
(1 - 2\epsilon)\eta(B_{1}) &\leq \int_{B_{1}} D_{\pi}^{(g_{y})}(f)~d\eta(y) \\
&= \int_{B_{1}} \int_{X} f(g_{y}^{-1}x)~dD_{\pi}(g_{y}y)(x)~d\eta(y) \\
&= \int_{B_{1}} \int_{X} \int_{G} \bbone_{A}(hg_{y}^{-1}x) \psi(h)~dm(h)~dD_{\pi}(g_{y}y)~d\eta(y) \\
&= \int_{G} \int_{B_{1}} D_{\pi}(g_{y}y)(g_{y}h^{-1}A)~d\eta(y) \psi(h)~dm(h) \\
&= \int_{G} \int_{hB_{1}} D_{\pi}(g_{y}h^{-1}y)(g_{y}h^{-1}A)~dh\eta(y) \psi(h)~dm(h) \\
&= \int_{G} \int_{hB_{1}} D_{\pi}^{(g_{y}h^{-1})}(y)(A)~dh\eta(y) \psi(h)~dm(h) \\
&\leq \int_{G} \int_{hB_{1}} \sup_{g} D_{\pi}^{(g)}(y)(A)~dh\eta(y) \psi(h)~dm(h) \\
&= \int_{G} \Big{(} \int_{hB_{1} \setminus B_{1}} \sup_{g} D_{\pi}^{(g)}(y)(A)~dh\eta(y) + \int_{hB_{1} \cap B_{1}} \sup_{g} D_{\pi}^{(g)}(y)(A)~dh\eta(y) \Big{)} \psi(h)~dm(h) \\
&\leq \int_{G} \big{(} h\eta(hB_{1} \setminus B_{1}) + (1 - \delta)h\eta(hB_{1} \cap B_{1}) \big{)} \psi(h)~dm(h) \\
&= \int_{G} \big{(} h\eta(hB_{1}) - \delta h\eta(hB_{1} \cap B_{1}) \big{)} \psi(h)~dm(h) \\
&= \eta(B_{1}) - \delta \int_{G} \eta(B_{1} \cap h^{-1}B_{1}) \psi(h)~dm(h).
\end{align*}
Now the support of $\psi$ is contained in $V$ and $| \eta(B_{1} \cap h^{-1}B_{1}) - \eta(B_{1})| < \epsilon$ for all $h \in V$.  Therefore
\[
-2\epsilon\eta(B_{1}) \leq - \delta \int_{G} (\eta(B_{1}) - \epsilon) \psi(h)~dm(h) = - \delta\eta(B_{1}) + \delta\epsilon.
\]
Hence
$\delta \eta(B_{1}) \leq \epsilon(2\eta(B_{1}) + \delta)$ and so
\[
\delta \eta(B) \leq \delta (\eta(B_{1}) + \epsilon) \leq 2\epsilon(\eta(B_{1}) + \delta) \leq 2\epsilon(\eta(B) + \delta).
\]
Since $\delta$ is fixed and this holds for all $\epsilon > 0$, $\eta(B) = 0$ contradicting that $\pi$ is not relatively contractive.
\end{proof}


The above is a generalization of a result of Furstenberg and Glasner \cite{FG10} that
a $G$-space is contractive if and only if every continuous compact model of the space is contractible.

The following fact was used in the above proof and is step by step equivalent to the proof of \cite{FG10} Proposition 8.8 but relativized over a $G$-map:
\begin{proposition}\label{P:fghelp}
Let $\pi : (X,\nu) \to (Y,\eta)$ be a $G$-map of $G$-spaces.  Let $\psi_{n} \in C_{c}(G)$ be an approximate identity (the $\psi_{n}$ are nonnegative continuous functions with decreasing compact supports $V_{n}$ such that $\cap V_{n} = \{ e \}$ and $\int \psi_{n} dm = 1$ for $m$ a Haar measure on $G$).  Then for any measurable set $A \subseteq X$ and almost every $y \in Y$ such that $D_{\pi}(y)(A) > 0$,
\[
\lim_{n} \| \bbone_{A} * \psi_{n} \|_{L^{\infty}(X,D_{\pi}(y))} = 1.
\]
\end{proposition}
\begin{proof}
Take a continuous compact model for $\pi$.
Let
\[
B = \{ y \in Y : D_{\pi}(y)(A) > 0 \}.
\]
Fix $\delta > 0$ and choose $\epsilon_{y} > 0$ for each $y \in B$ such that $\epsilon_{y} < \frac{1}{4}\delta D_{\pi}(y)(A)$.

For each $y \in B$, let $C_{y} \subseteq A \subseteq U_{y}$ such that $C_{y}$ is closed and $U_{y}$ is open and $D_{\pi}(y)(U_{y} \setminus C_{y}) < \epsilon_{y}$ (possible since $D_{\pi}(y)$ is regular).  Let $V_{y}$ be a symmetric compact neighborhood of the identity in $G$ such that $D_{\pi}(y)(hC_{y} \symdiff C_{y}) < \epsilon_{y}$ and such that $hC_{y} \subseteq U_{y}$ for all $h \in V_{y}$ (possible since the $G$-action is continuous).  Then $\bbone_{C_{y}}(hx) = 0$ for all $x \notin U_{y}$ and $h \in V_{y}$.

Let $N_{y} \in \mathbb{N}$ such that $\mathrm{supp}~\psi_{n} \subseteq V_{y}$ for all $n \geq N_{y}$.  For $n \geq N_{y}$, set $f_{y,n} = \bbone_{C_{y}} * \psi_{n}$.  Then $f_{y,n}(x) = 0$ for $x \notin U_{y}$.  So
\begin{align*}
D_{\pi}(y)(f_{y,n}) &=
\int_{X} \int_{G} \bbone_{C_{y}}(hx) \psi_{n}(h)~dm(h)~dD_{\pi}(y)(x) = \int_{G} D_{\pi}(y)(h^{-1}C_{y}) \psi_{n}(h)~dm(h) \\
&\geq \int_{G} (D_{\pi}(y)(C_{y}) - \epsilon_{y}) \psi_{n}(h)~dm(h) = D_{\pi}(y)(C_{y}) - \epsilon_{y}.
\end{align*}
Define
$E_{y,n} = \{ x \in U : f_{y,n}(x) < 1 - \delta \}$.
Then
\begin{align*}
D_{\pi}(y)(C_{y}) - \epsilon_{y} &\leq \int_{X} f_{y,n}(x)~dD_{\pi}(y)(x) = \int_{U_{y}} f_{y,n}(x)~dD_{\pi}(y)(x) \\
&= \int_{E_{y,n}} f_{y,n}(x)~dD_{\pi}(y)(x) + \int_{U_{y} \setminus E_{y,n}} f_{y,n}(x)~dD_{\pi}(y)(x) \\
&\leq (1 - \delta) D_{\pi}(y)(E_{y,n}) + D_{\pi}(y)(U_{y} \setminus E_{y,n}) \\
&= D_{\pi}(y)(U_{y}) - \delta D_{\pi}(y)(E_{y,n}).
\end{align*}
Therefore
\[
\delta D_{\pi}(y)(E_{y,n}) \leq D_{\pi}(y)(U_{y} \setminus C_{y}) + \epsilon_{y} < 2\epsilon_{y}
\]
Hence $D_{\pi}(y)(E_{y,n}) < 2 \epsilon_{y} \delta^{-1} < \frac{1}{2} D_{\pi}(y)(A)$.  So, for $x \in U_{y} \setminus E_{y,n}$, $f_{y,n}(x) \geq 1 - \delta$ and $D_{\pi}(y)(U_{y} \setminus E_{y,n}) \geq \frac{1}{2}D_{\pi}(y)(A) > 0$.

Therefore
\[
(\bbone_{A} * \psi_{n})(x) \geq (\bbone_{C_{y}} * \psi(x)) \geq 1 - \delta
\]
for all $x$ in a $D_{\pi}(y)$-positive measure set.  Hence for $n \geq N_{y}$, $\| \bbone_{A} * \psi_{n} \|_{L^{\infty}(X,D_{\pi}(y))} \geq 1 - \delta$.
As this holds for all $\delta > 0$,
$\lim_{n} \| \bbone_{A} * \psi_{n} \|_{L^{\infty}(X,D_{\pi}(y))} = 1$
for all $y \in B$.
\end{proof}

\subsection{Examples of Relatively Contractive Maps}

Let $(X,\nu)$ and $(Y,\eta)$ be contractive $G$-spaces.  In general it need not hold that $(X \times Y, \nu \times \eta)$ is contractive (with the diagonal $G$-action), however:
\begin{theorem}\label{T:examplecont}
Let $(X,\nu)$ be a contractive $G$-space and $(Y,\eta)$ be a $G$-space.  The map $\mathrm{pr_{Y}} : (X \times Y, \nu \times \eta) \to (Y,\eta)$ is relatively contractive ($X\times Y$ has the diagonal $G$-action).
\end{theorem}
\begin{proof}
The disintegration measures $D_{\pi}(y)$ are supported on $X \times \delta_{y}$ and have the form $D_{\pi}(y) = \nu \times \delta_{y}$.  Clearly
\[
D_{\pi}^{(g)}(y) = g^{-1} (\nu \times \delta_{gy}) = g^{-1}\nu \times \delta_{y}
\]
and since $(X,\nu)$ is contractive then $\pi$ is relatively contractive.
\end{proof}

More generally, the following holds:
\begin{theorem}
Let $\pi : (X,\nu) \to (Y,\eta)$ be a relatively contractive $G$-map of $G$-spaces.  Let $(Z,\zeta)$ be a $G$-space.  The map $\pi \times \mathrm{id} : (X \times Z, \nu \times \zeta) \to (Y \times Z, \eta \times \zeta)$ is relatively contractive (where $X\times Z$ and $Y\times Z$ have the diagonal $G$-action).
\end{theorem}
\begin{proof}
Since the disintegration of the identity is point masses, for almost every $(y,z) \in Y \times X$, it holds that
$D_{\pi \times \mathrm{id}}^{(g)}(y,z) = D_{\pi}^{(g)}(y) \times \delta_{z}$.
Then $\pi$ being relatively contractive implies $\pi \times \mathrm{id}$ is relatively contractive.
\end{proof}

Additionally, proximal maps (see Furstenberg and Glasner \cite{FG10} for a definition) are relatively contractive and if $\Gamma < G$ is a lattice and $(X,\nu)$ is a a contractive $\Gamma$-space then the natural projection map from the induced $G$-action $G \times_{\Gamma} X$ to $G/\Gamma$ is relatively contractive.  We leave the details to the reader as these facts are not necessary in the sequel.

\subsection{Factorization of Contractive Maps}

We now prove that if a composition of $G$-maps is relatively contractive then each of the maps is also relatively contractive.  This fact will be an important ingredient in the proof of the uniqueness of relatively contractive maps.

\begin{lemma}\label{L:compdisint}
Let $\pi : (X,\nu) \to (Y,\eta)$ and $\varphi : (Y,\eta) \to (Z,\rho)$ be $G$-maps of $G$-spaces.  Then for almost every $z \in Z$,
\[
\pi_{*}D_{\varphi \circ \pi}(z) = D_{\varphi}(z) \quad\quad\text{and}\quad\quad D_{\varphi\circ\pi}(z) = \int_{Y} D_{\pi}(y)~dD_{\varphi}(z)(y)
\]
\end{lemma}
\begin{proof}
Observe that the support of $\pi_{*}D_{\varphi\circ\pi}(z)$ is
\[
\pi((\varphi\circ\pi)^{-1}(z)) = \pi(\pi^{-1}(\varphi^{-1}(z))) = \varphi^{-1}(z).
\]
Also for $f \in L^{\infty}(Y,\eta)$, using the definition of disintegration over $\varphi\circ\pi$,
\begin{align*}
\int_{Z} \int_{Y} f(y)~d\pi_{*}&D_{\varphi\circ\pi}(z)(y)~d\zeta(z)
= \int_{Z} \int_{X} f(\pi(x))~dD_{\varphi\circ\pi}(z)(x)~d\zeta(z) \\
&= \int_{X} f(\pi(x))~d\nu(x) = \int_{Y} f(y)~d\pi_{*}\nu(y) = \int_{Y} f(y)~d\eta(y)
\end{align*}
and therefore by uniqueness of disintegration the first claim is proved.

Similarly, the support of $\int_{Y} D_{\pi}(y)~dD_{\varphi}(z)(y)$ is
\[
\bigcup_{y \in \varphi^{-1}(z)} \pi^{-1}(y) = \pi^{-1}(\varphi^{-1}(z)) = (\varphi\circ\pi)^{-1}(z)
\]
and also for $f \in L^{\infty}(X,\nu)$, using the definition of disintegration,
\begin{align*}
\int_{Z} \int_{X} f(x)~d\Big{(}\int_{Y} &D_{\pi}(y)~dD_{\varphi}(z)(y)\Big{)}(x)~d\zeta(z) \\
&= \int_{Z} \int_{Y} \int_{X} f(x)~dD_{\pi}(y)(x)~dD_{\varphi}(z)(y)~d\zeta(z) \\
&= \int_{Y} \int_{X} f(x)~dD_{\pi}(y)(x)~d\eta(y) = \int_{X} f(x)~d\nu(x)
\end{align*}
and therefore by uniqueness the second claim holds.
\end{proof}

\begin{theorem}\label{T:contractivecomp}
Let $\pi : (X,\nu) \to (Y,\eta)$ and $\varphi : (Y,\eta) \to (Z,\rho)$ be $G$-maps of $G$-spaces.  If $\varphi \circ \pi$ is relatively contractive then both $\varphi$ and $\pi$ are relatively contractive.
\end{theorem}
\begin{proof}
We use Theorem \ref{T:contractiveexttopo} and take a continuous compact model for $\pi$ to do so.  First observe, for all $g \in G$ and almost every $z$, that $\pi_{*}D_{\varphi\circ\pi}^{(g)}(z) = D_{\varphi}^{(g)}(z)$.  For such $z$ where also $\overline{\mathrm{conv}}~\{ D_{\varphi\circ\pi}^{(g)}(z) \} = P((\varphi\circ\pi)^{-1}(z))$ and every $x$ such that $\varphi(\pi(x)) = z$ there is $g_{n} \in G$ such that $D_{\varphi\circ\pi}^{(g_{n})}(z) \to \delta_{x}$.  Therefore
\[
D_{\varphi}^{(g_{n})}(z) = \pi_{*}D_{\varphi\circ\pi}^{(g_{n})}(z) \to \pi_{*}\delta_{x} = \delta_{\pi(x)}
\]
and so for every $y$ such that $\varphi(y) = z$ the point mass $\delta_{y}$ is a limit point of $D_{\varphi}^{(g)}(z)$.  Hence $\varphi$ is relatively contractive.

Suppose that $\pi$ is not relatively contractive.  Then, by the proof of Theorem \ref{T:contractiveexttopo}, there exists a continuous compact model for $\pi : X \to Y$ such that $f \mapsto |D_{\pi}^{(g)}(y)(f)|$ is not an isometry from $C(X)$ to $L^{\infty}(G)$ for a positive measure set of $y \in Y$.  

Observe that if the map is an isometry on a countable dense set $C_{0} \subseteq C(X)$ then for any $f \in C(X)$ there exists $f_{n} \in C_{0}$ with $f_{n} \to f$ in sup norm, hence
\[
|D_{\pi}^{(g)}(y)(f)| = |D_{\pi}^{(g)}(y)(f - f_{n}) + D_{\pi}^{(g)}(f_{n})| \geq |D_{\pi}^{(g)}(y)(f_{n})| - \| f - f_{n} \|_{\infty}.
\]
For $\epsilon > 0$, choose $n$ such that $\| f - f_{n} \|_{\infty} < \epsilon$.  Then choose $g$ such that $|D_{\pi}^{(g)}(y)(f_{n})| > \| f_{n} \| - \epsilon$.  Then
\[
|D_{\pi}^{(g)}(y)(f)| > \| f_{n} \| - \epsilon - \epsilon > \| f \| - 3\epsilon
\]
and so the map is an isometry for $f$ as well.

Therefore, there is a positive measure set of $y$ such that the map $f \mapsto |D_{\pi}^{(g)}(y)(f)|$ is not an isometry on $C_{0}$.  Hence, since $C_{0}$ is countable, there is some $f \in C_{0}$ and a positive measure set of $y$ such that $\sup_{g} |D_{\pi}^{(g)}(y)(f)| < \| f \|_{L^{\infty}(X,D_{\pi}(y))}$.  So there is some $\delta > 0$ and a measurable set $A \subseteq Y$ with $\eta(A) > 0$ such that $\sup_{g} |D_{\pi}^{(g)}(y)(f)| < \| f \|_{L^{\infty}(X,D_{\pi}(y))} - \delta$ for all $y \in A$.  We may assume (by taking a subset) that $A$ is closed.
Since $\eta$ is a Borel measure, it is regular, hence we may assume $A$ is closed (by taking a subset).

Now there exists a positive measure set $B \subseteq Z$ on which $D_{\varphi}(z)(A) > 0$ for $z \in B$.  For $z \in B$ such that $z$ is in the measure one set on which $\varphi \circ \pi$ contracts to point masses,
\begin{align*}
D_{\varphi \circ \pi}^{(g)}&(z)(f) \\
&= \int_{\varphi^{-1}(z)} D_{\pi}^{(g)}(y)(f)~dD_{\varphi}^{(g)}(z)(y) \\
&= \int_{\varphi^{-1}(z) \cap A} D_{\pi}^{(g)}(y)(f)~dD_{\varphi}^{(g)}(z)(y) + \int_{\varphi^{-1}(z) \setminus A} D_{\pi}^{(g)}(y)(f)~dD_{\varphi}^{(g)}(z)(y) \\
&\leq \int_{\varphi^{-1}(z) \cap A} \| f \|_{L^{\infty}(X,D_{\pi}(y))} - \delta~dD_{\varphi}^{(g)}(z)(y) + \int_{\varphi^{-1}(z) \setminus A} \| f \|_{L^{\infty}(X,D_{\pi}(y))}~dD_{\varphi}^{(g)}(z)(y) \\
&\leq \| f \|_{L^{\infty}(X,D_{\varphi \circ \pi}(z))} - \delta D_{\varphi}^{(g)}(z)(A).
\end{align*}

Now for any $x \in (\varphi \circ \pi)^{-1}(z)$, there exists $g_{n}$ such that $D_{\varphi \circ \pi}^{(g_{n})}(z) \to \delta_{x}$.  Hence also $D_{\varphi}^{(g_{n})}(z) \to \delta_{\pi(x)}$.  Choose $x \in \pi^{-1}(A) \cap (\varphi\circ\pi)^{-1}(z)$ such that $f(x) = \| f \|_{L^{\infty}(X,D_{\varphi \circ \pi}(z))}$ (possible since $\pi^{-1}(A) \cap (\varphi\circ\pi)^{-1}(z)$ is closed, hence compact, and $f$ is continuous).
Then
\begin{align*}
f(x) &= \lim_{n} D_{\varphi \circ \pi}^{(g_{n})}(z)(f)
\leq \lim_{n}  \| f \|_{L^{\infty}(X,D_{\varphi \circ \pi}(z))} - \delta D_{\varphi}^{(g_{n})}(z)(A) \\
&= \| f \|_{L^{\infty}(X,D_{\varphi \circ \pi}(z))} - \delta \delta_{\pi(x)}(A) = \| f \|_{L^{\infty}(X,D_{\varphi \circ \pi}(z))} - \delta
\end{align*}
is a contradiction.  Hence $\pi$ is relatively contractive.
\end{proof}

\subsection{Uniqueness of Relatively Contractive Maps}

Our next result generalizes a result of the first author and Shalom \cite{CS14} that factors of contractive spaces are unique (the case when $(Y,\eta)$ in the theorem below is trivial).

\begin{theorem}\label{T:relcontractiveunique}
Let $(X,\nu)$ be a contractive $G$-space and $(Y,\eta)$ be a measure-preserving $G$-space.  Let $\psi : (X \times Y, \nu \times \eta) \to (Y,\eta)$ be the natural projection map (treating $(X \times Y, \nu \times \eta)$ as $G$-space with the diagonal action).  Let $\pi : (X \times Y, \nu \times \eta) \to (Z,\alpha)$ be a $G$-map of $G$-spaces and let $\pi^{\prime} : (X \times Y, \nu \times \eta) \to (Z,\beta)$ be a $G$-map of $G$-spaces such that $\alpha$ is in the same measure class as $\beta$.  Let $\varphi : (Z,\alpha) \to (Y,\eta)$ and $\varphi^{\prime} : (Z,\beta) \to (Y,\eta)$ be $G$-maps such that $\varphi \circ \pi = \psi$ and $\varphi^{\prime} \circ \pi^{\prime} = \psi$.  Assume that the disintegrations $D_{\varphi}(y)$ of $\alpha$ over $\eta$ via $\varphi$ and the disintegrations $D_{\varphi^{\prime}}(y)$ of $\beta$ over $\eta$ via $\varphi^{\prime}$ have the property that $D_{\varphi}(y)$ and $D_{\varphi^{\prime}}(y)$ are in the same measure class almost surely.  Then $\pi = \pi^{\prime}$ almost everywhere, $\varphi = \varphi^{\prime}$ almost everywhere and $\alpha = \beta$.
\end{theorem}
\begin{proof}
First we consider $\varphi$ and $\varphi^{\prime}$.
Define the Borel set
\[
B = \{ z \in Z : \varphi(z) \ne \varphi^{\prime}(z) \}.
\]
Then for every $y \in Y$, it holds that $B \cap \varphi^{-1}(y) \cap (\varphi^{\prime})^{-1}(y) = \emptyset$.  Since $D_{\varphi}(y)$ is in the same measure class as $D_{\varphi^{\prime}}(y)$ almost everywhere and since $D_{\varphi^{\prime}}(y)((\varphi^{\prime})^{-1}(y)) = 1$, for almost every $y$ it holds that
\[
D_{\varphi}(y)(B) = D_{\varphi}(y)(B \cap \varphi^{-1}(y)) = D_{\varphi}(y)(B \cap \varphi^{-1}(y) \cap (\varphi^{\prime})^{-1}(y)) = D_{\varphi}(y)(\emptyset) = 0.
\]
Therefore $\zeta(B) = 0$.  Likewise, $\zeta^{\prime}(B) = 0$.  Hence $\varphi = \varphi^{\prime}$ almost everywhere.

Now we consider $\pi$ and $\pi^{\prime}$.  Suppose that
\[
\nu \times \eta (\{ (x,y) \in X \times Y : \pi(x,y) \ne \pi^{\prime}(x,y) \}) > 0.
\]
Fix compact models for $X$, $Y$ and $Z$ and let $d$ be a compatible metric on $Z$ and observe that
\[
\{ (x,y) \in X \times Y : \pi(x,y) \ne \pi^{\prime}(x,y) \} = \bigcup_{\delta > 0} \{ (x,y) \in X \times Y : d(\pi(x,y), \pi^{\prime}(x,y)) \geq \delta \}
\]
which is a decreasing union and therefore there is some $\delta > 0$ such that
\[
\nu \times \eta (\{ (x,y) \in X \times Y : d(\pi(x,y), \pi^{\prime}(x,y)) > \delta \}) > 0.
\]
By Fubini's Theorem there is then some $x_{0} \in X$ such that
\[
A = \{ y \in Y : d(\pi(x_{0},y),\pi^{\prime}(x_{0},y)) > \delta \}
\]
has $\eta(A) > 0$.

Since $(X,\nu)$ is contractive, there exists a sequence $\{ g_{n} \}$ in $G$ such that $g_{n}^{-1}\nu \to \delta_{x_{0}}$.  Observe that for almost every $y \in Y$,
\[
D_{\varphi}^{(g_{n})}(y) = \pi_{*}D_{\psi}^{(g_{n})}(y) = \pi_{*}(g_{n}^{-1}(\nu \times \delta_{g_{n}y})) = \pi_{*}(g_{n}^{-1}\nu \times \delta_{y})
\to \pi_{*}(\delta_{x_{0}} \times \delta_{y}) = \delta_{\pi(x_{0},y)}
\]
and likewise that
\[
D_{\varphi^{\prime}}^{(g_{n})}(y) \to \delta_{\pi^{\prime}(x_{0},y)}.
\]

Define the set
\[
U = \{ z \in Z : d(\pi(x_{0},\varphi(z)),z) < \frac{1}{2}\delta \} \cap \varphi^{-1}(A).
\]
Note that $U \cap \varphi^{-1}(y)$ is open in $\varphi^{-1}(y)$ for all $y \in A$ since $d$ is compatible.
Moreover, for each $y \in A \cap \varphi(U)$, it holds that $\pi(x_{0},y)$ is in the interior of $U$.  
Observe that for $z \in U$ we have that $d(\pi(x_{0},\varphi(z)),z) < \frac{1}{2}\delta$ and
\[
d(\pi(x_{0},\varphi(z)),\pi^{\prime}(x_{0},\varphi(z))) > \delta
\]
(using that $\varphi = \varphi^{\prime}$) meaning that $d(\pi^{\prime}(x_{0},\varphi(z)),z) > \frac{1}{2}\delta$ and therefore we conclude that $\pi^{\prime}(x_{0},y) \in (\overline{U})^{C}$ for every $y \in A$.  Therefore $U$ is a continuity set for $\delta_{\pi(x_{0},y)}$ and $\delta_{\pi^{\prime}(x_{0},y)}$ for all $y \in A$.

Then for almost every $y \in A$,
\[
D_{\varphi}^{(g_{n})}(y)(U) \to \delta_{\pi(x_{0},y)}(U) = 1
\]
and
\[
D_{\varphi^{\prime}}^{(g_{n})}(y)(U) \to \delta_{\pi^{\prime}(x_{0},y)}(U) = 0
\]
since $U$ is a continuity set.

For $\epsilon > 0$, define the Borel sets
\[
A_{\epsilon,n} = \{ y \in A : \text{ for all $m \geq n$, } D_{\varphi}^{(g_{m})}(y)(U) > 1 - \epsilon \text{ and } D_{\varphi^{\prime}}^{(g_{m})}(y)(U) < \epsilon \}.
\]
The sets $A_{\epsilon,n}$ increase with $n$ for a fixed $\epsilon$ and, up to a null set,
\[
A = \bigcup_{n=1}^{\infty} A_{\epsilon,n}
\]
and therefore, for each $\epsilon > 0$ there exists $n$ such that $\eta(A_{\epsilon,n}) > \eta(A) - \epsilon$.

Now, using that $\eta$ is measure-preserving, for every $\epsilon > 0$,
\begin{align*}
\alpha(g_{n}U) &= \int_{Y} D_{\varphi}(y)(g_{n}U)~d\eta(y) \\
&= \int_{g_{n}A} D_{\varphi}(y)(g_{n}U)~d\eta(y) \\
&= \int_{A} D_{\varphi}^{(g_{n})}(y)(U)~d\eta(y) \\
&\geq \int_{A_{n,\epsilon}} D_{\varphi}^{(g_{n})}(y)(U)~d\eta(y) \\
&\geq (1 - \epsilon) \eta(A_{n,\epsilon}) \geq (1 - \epsilon)(\eta(A) - \epsilon).
\end{align*}
Similarly,
\[
\beta(g_{n}U) \leq \int_{A_{n,\epsilon}} D_{\varphi^{\prime}}^{(g_{n})}(y)(U)~d\eta(y) + \eta(A) - \eta(A_{n,\epsilon})
\leq \epsilon \eta(A_{n,\epsilon}) + \epsilon \leq 2\epsilon.
\]

By Lemma \ref{L:notabscont} (following the proof), then $\alpha$ and $\beta$ are not in the same measure class, a contradiction.  Therefore we conclude that $\pi = \pi^{\prime}$ almost everywhere and so $\alpha = \pi_{*}\nu = (\pi^{\prime})_{*}\nu = \beta$.
\end{proof}

\begin{lemma}\label{L:notabscont}
Let $Z$ be a compact metric space and $\alpha, \beta \in P(Z)$ be Borel probability measures on it.  If there exists $\delta > 0$ such that for every $\epsilon > 0$ there exists a Borel set $B_{\epsilon} \subseteq Z$ such that $\alpha(B_{\epsilon}) < \epsilon$ and $\beta(B_{\epsilon}) > (1 - \epsilon)\delta$ then $\alpha$ is not absolutely continuous with respect to $\beta$.
\end{lemma}
\begin{proof}
Observe that
\[
\alpha(\bigcap_{n=1}^{\infty}\bigcup_{m=n+1} B_{2^{-m}}) \leq \limsup_{n \to \infty} \sum_{m=n+1}^{\infty} \alpha(B_{2^{-m}}) \leq \limsup_{n\to\infty} \sum_{m=n+1}^{\infty} 2^{-m} = 0
\]
but that
\[
\beta(\bigcap_{n=1}^{\infty}\bigcup_{m=n+1} B_{2^{-m}}) \geq \liminf_{n \to \infty} \beta(B_{2^{-n-1}}) = \liminf_{n \to \infty} (1 - 2^{-n-1})\delta = \delta > 0.
\]
\end{proof}

\subsection{Relatively Contractive Maps and Finite Index Subgroups}

\begin{theorem}\label{T:rcfi}
Let $G$ be a locally compact second countable group and $H < G$ be a finite index subgroup.  Let $\pi : (X,\nu) \to (Y,\eta)$ be a relatively contractive $G$-map of ergodic $G$-spaces.  Then, restricting the actions to $H$ makes $\pi$ a relatively contractive $H$-map.
\end{theorem}
\begin{proof}
Fix continuous compact models of $X$, $Y$ and $\pi$.  Let $\ell_{1},\ell_{2},\ldots,\ell_{N}$ be a system of representatives for $H \backslash G$.  Define the set
\[
Q = \{ x \in X : \text{ there exists $\{ h_{n} \}$ in $H$ such that $D_{\pi}^{(h_{n})}(\pi(x)) \to \delta_{x}$ } \}.
\]
By Theorem \ref{T:contractiveexttopo}, for every $x \in X$ there exists $\{ g_{n} \}$ in $G$ such that $D_{\pi}^{(g_{n})}(\pi(x)) \to \delta_{x}$.  Write $g_{n} = h_{n}\ell_{j_{n}}$ for $h_{n} \in H$ and $j_{n} \in \{ 1, 2, \ldots, N \}$.  Since $N$ is finite, there exists a subsequence $\{ n_{t} \}$ along which $g_{n_{t}} = h_{n_{t}} \ell$ for some fixed $\ell$ in the system of representatives.  Define the sets, for $j = 1, 2, \ldots, N$,
\[
C_{j} = \{ x \in X : \text{ there exists $\{ g_{n} \}$ such that $g_{n} \in H \ell_{j}$ for all $n$ and $D_{\pi}^{(g_{n})}(y) \to \delta_{x}$ } \}.
\]
Then $X = \cup_{j=1}^{N} C_{j}$ by the above.  

Let $x \in C_{j}$.  Then, writing $g_{n} = h_{n}\ell_{j}$, it holds that $\ell_{j}^{-1} D_{\pi}^{(h_{n})}(\ell_{j} \pi(x)) = D_{\pi}^{(h_{n}\ell_{j})}(\pi(x)) \to \delta_{x}$ and so $D_{\pi}^{(h_{n})}(\pi(\ell_{j} x)) \to \ell_{j} \delta_{x} = \delta_{\ell_{j} x}$.  Therefore $\ell_{j} x \in Q$ and so we have that
\[
\bigcup_{j=1}^{N} \ell_{j}C_{j} \subseteq Q.
\]
Since $X = \cup_{j} C_{j}$, there is some $j$ such that $\nu(C_{j}) > 0$.  So $\nu(Q) \geq \nu(\ell_{j}C_{j}) > 0$ as $\nu$ is quasi-invariant.

Now let $x \in Q$ and $h \in H$.  There exists $\{ h_{n} \}$ in $H$ such that $D_{\pi}^{(h_{n})}(\pi(x)) \to \delta_{x}$ and therefore
\[
D_{\pi}^{(h_{n}h^{-1})}(\pi(hx)) = h D_{\pi}^{(h_{n})}(h^{-1}\pi(hx)) = h D_{\pi}^{(h_{n})}(\pi(x)) \to h \delta_{x} = \delta_{hx}
\]
meaning that $hx \in Q$.  Since $(X,\nu)$ is $G$-ergodic, it is also $H$-ergodic ($H$ being finite index) and therefore $\nu(Q) = 1$ meaning precisely that $\pi$ is a relatively contractive $H$-map.
\end{proof}

\subsection{The Intermediate Contractive Factor Theorem}

The main result we will need in the sequel is Theorem \ref{T:groupoidcontractive1}.  The next result is a simple case of that Theorem and provides motivation both in terms of the statement and the proof.

\begin{theorem}\label{T:contractiveIFT}
Let $\Gamma < G$ be a lattice in a locally compact second countable group and let $\Lambda$ contain and commensurate $\Gamma$ and be dense in $G$.  Let $(X,\nu)$ be a contractive $G$-space which is also contractive as a $\Gamma$-space and $(Y,\eta)$ be a measure-preserving $G$-space.  Let $\pi : (X \times Y, \nu \times \eta) \to (Y,\eta)$ be the natural projection map from the product space with the diagonal action.
Let $(Z,\zeta)$ be a $\Lambda$-space such that there exist $\Gamma$-maps $\varphi : (X\times Y,\nu\times\eta) \to (Z,\zeta)$ and $\rho : (Z,\zeta) \to (Y,\eta)$ with $\rho \circ \varphi = \pi$.  Then $\varphi$ and $\rho$ are $\Lambda$-maps and $(Z,\zeta)$ is $\Lambda$-isomorphic to a $G$-space and over this isomorphism the maps $\varphi$ and $\rho$ become $G$-maps.
\end{theorem}
\begin{proof}
Write $(W,\rho) = (X \times Y, \nu \times \eta)$.
Fix $\lambda \in \Lambda$.  Define the maps $\varphi_{\lambda} : W \to Z$ and $\rho_{\lambda} : Z \to Y$ by $\varphi_{\lambda}(w) = \lambda^{-1}\varphi(\lambda w)$ and $\rho_{\lambda}(z) = \lambda^{-1}\rho(\lambda z)$.  Then $\rho_{\lambda} \circ \varphi_{\lambda}(w) = \lambda^{-1} \rho(\lambda \lambda^{-1}\varphi(\lambda w)) = \lambda^{-1}\rho(\varphi(\lambda w)) = \lambda^{-1}\pi(\lambda w) = \pi(w)$ since $\pi$ is $\Lambda$-equivariant.  Let $\Gamma_{0} = \Gamma \cap \lambda^{-1}\Gamma\lambda$.  Then for $\gamma_{0} \in \Gamma_{0}$, write $\gamma_{0} = \lambda^{-1}\gamma\lambda$ for some $\gamma \in \Gamma$ and we see that $\varphi_{\lambda}(\gamma_{0}w) = \lambda^{-1}\varphi(\lambda\gamma_{0}w) = \lambda^{-1}\varphi(\gamma \lambda w) = \lambda^{-1}\gamma\varphi(\lambda w) = \gamma_{0} \lambda^{-1}\varphi(\lambda w) = \gamma_{0} \varphi_{\lambda}(w)$ meaning that $\varphi_{\lambda}$ is $\Gamma_{0}$-equivariant.  Likewise $\rho_{\lambda}$ is $\Gamma_{0}$-equivariant.  Hence $\varphi$, $\varphi_{\lambda}$, $\rho$ and $\rho_{\lambda}$ are all $\Gamma_{0}$-equivariant.  

Since $(X,\nu)$ is a contractive $\Gamma$-space, by Theorem \ref{T:relcontractiveunique} applied to $\Gamma$, we can conclude that $\varphi_{\lambda} = \varphi$ and that $\rho_{\lambda} = \rho$ provided we can show that the disintegration measures $D_{\rho}(y)$ and $D_{\rho_{\lambda}}(y)$ are in the same measure class for almost every $y$.  Assuming this for the moment, we then conclude that $\varphi$ is $\Lambda$-equivariant since $\varphi_{\lambda} = \varphi$ for each $\lambda$.  The $\sigma$-algebra of pullbacks of measurable functions on $(Z,\zeta)$ form a $\Lambda$-invariant sub-$\sigma$-algebra of $L^{\infty}(W,\rho)$ which is therefore also $G$-invariant (because $\Lambda$ is dense in $G$) and so $(Z,\zeta)$ has a point realization as a $G$-space \cite{Ma62} and likewise $\varphi$ and $\rho$ as $G$-maps.

It remains only to show that the disintegration measures have the required property.  First note that $D_{\rho}(y) = \varphi_{*}D_{\rho\circ\varphi}(y)$ by the uniqueness of the disintegration measure and likewise that $D_{\rho_{\lambda}}(y) = (\varphi_{\lambda})_{*}D_{\rho_{\lambda}\circ\varphi_{\lambda}}(y) = \lambda^{-1}\varphi_{*}\lambda D_{\rho\circ\varphi}(y) = \lambda^{-1}\varphi_{*}D_{\rho\circ\varphi}^{(\lambda^{-1})}(\lambda y)$.  Now $\rho\circ\varphi = \pi$ is a $\Lambda$-map so $D_{\rho\circ\varphi}^{(\lambda^{-1})}(\lambda y)$ is in the same measure class as $D_{\rho\circ\varphi}(\lambda y)$.  Therefore $D_{\rho_{\lambda}}(y)$ is in the same measure class as $\lambda^{-1}\varphi_{*}D_{\rho\circ\varphi}(\lambda y) = \lambda^{-1} D_{\rho}(\lambda y)$.  Now $\lambda^{-1} D_{\rho}(\lambda y)$ disintegrates $\lambda^{-1}\zeta$ over $\lambda^{-1}\eta$ via $\rho$ and $\lambda^{-1}\zeta$ is in the same measure class as $\zeta$ since $(Z,\zeta)$ is a $\Lambda$-space.  Therefore, by Lemma \ref{L:abscontmeas}, $\lambda^{-1}D_{\rho}(\lambda y)$ and $D_{\rho}(y)$ are in the same measure class for almost every $y$.  Hence $D_{\rho_{\lambda}}(y)$ and $D_{\rho}(y)$ are in the same measure class for almost every $y$ as needed.
\end{proof}

%


\begin{theorem}\label{T:groupoidcontractive1}
Let $\Gamma$ be a group and let $\Lambda$ be a group that contains and commensurates $\Gamma$.
Let $(W,\rho)$ be a $\Lambda$-space such that the action restricted to $\Gamma$ on $(W,\rho)$ is contractive and let $(X,\nu)$ be a measure-preserving $\Lambda$-space.  Set $(Y,\eta) = (W \times X, \rho\times\nu)$ to be the product space with the diagonal action.
Let $p : (Y,\eta) \to (X,\nu)$ be the natural projection map.
Let $(Z,\zeta)$ be a $\Gamma$-space and $\pi : (Y,\eta) \to (Z,\zeta)$ and $\varphi : (Z,\zeta) \to (X,\nu)$ be $\Gamma$-maps such that $\varphi \circ \pi = p$.

Assume that $Z$ is orbital over $X$: for any $\gamma \in \Gamma$ and $x \in X$ such that $\gamma x = x$, if $z \in Z$ such that $\varphi(z) = x$ then $\gamma z = z$.

Fix $\lambda \in \Lambda$, define the Borel set
\[
E = E_{\lambda} = \{ x \in X : \lambda x \in \Gamma x \},
\]
and define the map $\theta_{\lambda} : \varphi^{-1}(E) \to Z$ as follows: for $z \in \varphi^{-1}(E)$ choose $\gamma \in \Gamma$ such that $\lambda \varphi(z) = \gamma \varphi(z)$ and define $\theta_{\lambda}(z) = \gamma z$ (this is well-defined since $Z$ is orbital).

Then $\pi(\lambda y) = \theta_{\lambda}(\pi(y))$ for almost every $y \in p^{-1}(E)$.  In particular, for almost every $y$ such that $\lambda p(y) = p(y)$ we have that $\pi(\lambda y) = \pi(y)$.
\end{theorem}

The proof of the theorem will proceed as a series of Propositions.  Retain the notation above throughout:
\begin{proposition}
$\theta_{\lambda}(\varphi^{-1}(E)) = \varphi^{-1}(\lambda E)$.
\end{proposition}
\begin{proof}
Let $z \in \theta_{\lambda}(\varphi^{-1}(E))$.  Then $z = \theta_{\lambda}(w)$ for some $w \in \varphi^{-1}(E)$ so $\lambda\varphi(w) = \gamma\varphi(w)$ for some $\gamma \in \Gamma$ hence $z = \gamma w$ by the definition of $\theta_{\lambda}$.  Then $\varphi(z) = \varphi(\gamma w) = \gamma \varphi(w) = \lambda \varphi(w) \in \lambda E$.  Therefore $\theta_{\lambda}(\varphi^{-1}(E)) \subseteq \varphi^{-1}(\lambda E)$.

Conversely, let $z \in \varphi^{-1}(\lambda E)$.  Then $\varphi(z) = \lambda x$ for some $x \in E$ so there exists $\gamma \in \Gamma$ such that $\varphi(z) = \lambda x = \gamma x$.  Now $\varphi(\gamma^{-1}z) = \gamma^{-1}\varphi(z) = x \in E$ and $\lambda x = \gamma x$ so $\theta_{\lambda}(\gamma^{-1}z) = \gamma (\gamma^{-1}z) = z$.  Therefore $z = \theta_{\lambda}(\gamma^{-1}z) \in \theta_{\lambda}(\varphi^{-1}(E))$ so $\varphi^{-1}(\lambda E) \subseteq \theta_{\lambda}(\varphi^{-1}(E))$.
\end{proof}

\begin{proposition}
$\theta_{\lambda}$ is invertible: there exists $\theta_{\lambda}^{-1} : \theta_{\lambda}(\varphi^{-1}(E)) \to E$ such that $\theta_{\lambda}^{-1}\theta_{\lambda}$ is the identity on $\varphi^{-1}(E)$ and $\theta_{\lambda}\theta_{\lambda}^{-1}$ is the identity on $\theta_{\lambda}(\varphi^{-1}(E))$.
\end{proposition}
\begin{proof}
Let $w \in \theta_{\lambda}(\varphi^{-1}(E))$.  Then $w = \theta_{\lambda}(z)$ for some $z \in \varphi^{-1}(E)$ so $w = \gamma z$ for some $\gamma \in \Gamma$ such that $\lambda \varphi(z) = \gamma \varphi(z)$.  Note that if $\gamma,\gamma^{\prime} \in \Gamma$ are both such that $\lambda \varphi(z) = \gamma \varphi(z) = \gamma^{\prime} \varphi(z)$ then $\gamma^{-1}\gamma^{\prime} \varphi(z) = \varphi(z)$ so as $Z$ is orbital then $\gamma^{-1}\gamma^{\prime} z = z$.  Define $\theta_{\lambda}^{-1}(w) = \gamma^{-1}w$.  This is then well-defined since $(\gamma^{\prime})^{-1}w = (\gamma^{\prime})^{-1}\gamma \gamma^{-1}w = (\gamma^{\prime})^{-1}\gamma z = z = \gamma^{-1}w$ because $\gamma^{-1}\gamma^{\prime} z = z$.  Then $\theta_{\lambda}^{-1}(\theta_{\lambda}(z)) = \theta_{\lambda}^{-1}(w) = z$ and $\theta_{\lambda}(\theta_{\lambda}^{-1}(w)) = \theta_{\lambda}(z) = w$ hence the proof is complete (since $\theta_{\lambda}$ maps onto its image).
\end{proof}

\begin{proposition}
$\Gamma_{0} = \Gamma \cap \lambda^{-1} \Gamma \lambda$ is a lattice in $\Gamma$ and $E$ is $\Gamma_{0}$-invariant.
\end{proposition}
\begin{proof}
$\Gamma_{0}$ has finite index in $\Gamma$ since $\Lambda$ commensurates $\Gamma$ hence is a lattice.
Observe that for $\gamma_{0} \in \Gamma_{0}$ and $x \in E$, writing $\gamma_{0} = \lambda^{-1}\gamma\lambda$ for some $\gamma \in \Gamma$ we have that
\[
\lambda \gamma_{0} x = \lambda \lambda^{-1} \gamma \lambda x = \gamma \lambda x \in \gamma \Gamma x = \Gamma x
\]
and therefore the set $E$ is $\Gamma_{0}$-invariant, that is $\lambda \gamma_{0}x \in \Gamma x$ whenever $\lambda x \in \Gamma x$.
\end{proof}

\begin{proposition}
Define the map $\pi_{\lambda} : Y \to Z$ as follows: for $y \in p^{-1}(E)$ set $\pi_{\lambda}(y) = \theta_{\lambda}^{-1}(\pi(\lambda y))$ and for $y \notin p^{-1}(E)$ set $\pi_{\lambda}(y) = \pi(y)$.  Likewise define the map $\varphi_{\lambda} : Z \to X$ by $\varphi_{\lambda}(z) = \lambda^{-1}\varphi(\theta_{\lambda} (z))$ for $z$ such that $\varphi(z) \in E$ and $\varphi_{\lambda}(z) = \varphi(z)$ for $z$ such that $\varphi(z) \notin E$.

Then $\varphi_{\lambda} \circ \pi_{\lambda} = \varphi \circ \pi = p$ and both $\pi_{\lambda}$ and $\varphi_{\lambda}$ are $\Gamma_{0}$-equivariant.
\end{proposition}
\begin{proof}
Note that in fact $\varphi_{\lambda} = \varphi$ since for $z \in \varphi^{-1}(E)$ and $\gamma \in \Gamma$ such that $\lambda\varphi(z) = \gamma\varphi(z)$ we have that $\lambda^{-1} \varphi(\gamma z) = \lambda^{-1}\gamma \varphi(z) = \varphi(z)$ but we will find it helpful to distinguish these maps since the measures $\pi_{*}\eta$ and $(\pi_{\lambda})_{*}\eta$ may be distinct and we will be treating $\varphi_{\lambda}$ as a map $(Z,(\pi_{\lambda})_{*}\eta) \to (X,\nu)$ and $\varphi$ as a map $(Z,\pi_{*}\eta) \to (X,\nu)$.

Now for $y$ such that $p(y) \in E$, observe that
\[
\varphi_{\lambda}(\pi_{\lambda}(y)) = \lambda^{-1}\varphi(\theta_{\lambda} \theta_{\lambda}^{-1} \pi(\lambda y))
= \lambda^{-1} \varphi(\pi(\lambda y)) = \lambda^{-1} p(\lambda y) = p(y)
\]
since $p$ is $\Lambda$-equivariant.  Clearly for $y$ such that $p(y) \notin E$ we have $\varphi_{\lambda}(\pi_{\lambda}(y)) = \varphi_{\lambda}(\pi(y)) = \varphi(\pi(y)) = p(y)$.  Hence $\varphi_{\lambda} \circ \pi_{\lambda} = p$.

Observe that for $\gamma_{0} \in \Gamma_{0}$, writing $\gamma_{0} = \lambda^{-1}\gamma\lambda$ for some $\gamma \in \Gamma$, we have that for $y$ such that $p(y) \in E$, also $p(\gamma_{0}y) = \gamma_{0} p(y) \in E$ since $E$ is $\Gamma_{0}$-invariant, so
\[
\pi_{\lambda}(\gamma_{0}y) = \theta_{\lambda}^{-1} \pi(\lambda \gamma_{0}y)
= \theta_{\lambda}^{-1} \pi(\gamma \lambda y)
= \theta_{\lambda}^{-1} \gamma \pi(\lambda y)
= \theta_{\lambda}^{-1} \gamma \theta_{\lambda} \theta_{\lambda}^{-1} \pi(\lambda y)
= \theta_{\lambda}^{-1} \gamma \theta_{\lambda} \pi_{\lambda}(y).
\]
Now observe that for $z$ such that $\varphi(z) \in E$ (which includes $\pi_{\lambda}(y)$ for $p(y) \in E$), write $\gamma^{\prime} \in \Gamma$ such that $\theta_{\lambda} z = \gamma^{\prime}z$ and observe that then $\lambda \varphi(z) = \gamma^{\prime} \varphi(z)$ and so
\[
\gamma \gamma^{\prime} \gamma_{0}^{-1} \varphi(\gamma_{0} z)
= \gamma \gamma^{\prime} \varphi(z) = \gamma \lambda \varphi(z) = \lambda \gamma_{0} \varphi(z)
= \lambda \varphi(\gamma_{0} z)
\]
which in turn means that
\[
\theta_{\lambda} (\gamma_{0} z) = \gamma \gamma^{\prime} \gamma_{0}^{-1} (\gamma_{0} z) = \gamma \gamma^{\prime} z = \gamma \theta_{\lambda} (z)
\]
and therefore
\[
\pi_{\lambda}(\gamma_{0}y) =  \theta_{\lambda}^{-1} \gamma \theta_{\lambda} \pi_{\lambda}(y) 
= \theta_{\lambda}^{-1}\theta_{\lambda}\gamma_{0}\pi_{\lambda}(y) = \gamma_{0} \pi_{\lambda}(y).
\]
Of course, for $y$ such that $p(y) \notin E$ we have that $\gamma_{0} y \notin E$ and so
\[
\pi_{\lambda}(\gamma_{0} y) = \pi(\gamma_{0} y) = \gamma_{0} \pi(y) = \gamma_{0} \pi_{\lambda}(y)
\]
and we conclude that $\pi_{\lambda}$ is $\Gamma_{0}$-equivariant.  Note that $\varphi_{\lambda} = \varphi$ so $\varphi_{\lambda}$ is likewise $\Gamma_{0}$-equivariant.
\end{proof}

\begin{proposition}
The maps $\pi,\varphi,\pi_{\lambda},\varphi_{\lambda}$ are all relatively contractive $\Gamma_{0}$-maps.
\end{proposition}
\begin{proof}
$p$ is a relatively contractive $\Gamma$-map hence is a relatively contractive $\Gamma_{0}$-map since $\Gamma_{0}$ has finite index in $\Gamma$ (Theorem \ref{T:rcfi}).  Since $\varphi_{\lambda} \circ \pi_{\lambda} = \varphi \circ \pi = p$ then the maps are all relatively contractive (Theorems \ref{T:contractivecomp} and \ref{T:examplecont}).
\end{proof}

\begin{proposition}
$\zeta_{\lambda}$ is in the same measure class as $\zeta$.
\end{proposition}
\begin{proof}
Let $B \subseteq Z$ be measurable such that $B \cap \varphi^{-1}(E) = \emptyset$.  Then $\pi^{-1}(B) \cap p^{-1}(E) = \pi^{-1}(B \cap \varphi^{-1}(E)) = \emptyset$ and $\pi_{\lambda}^{-1}(B) \cap p^{-1}(E) = \pi_{\lambda}^{-1}(B \cap \varphi_{\lambda}^{-1}(E)) = \emptyset$ since $\varphi = \varphi_{\lambda}$ pointwise.  So $\pi_{\lambda}(y) = \pi(y)$ for $y \in \pi^{-1}(B)$ and for $y \in \pi_{\lambda}^{-1}(B)$.  Then
\[
\zeta_{\lambda}(B) = \eta(\pi_{\lambda}^{-1}(B)) \leq \eta(\pi_{\lambda}^{-1}(\pi(\pi^{-1}(B)))) = \eta(\pi^{-1}(B)) = \zeta(B)
\]
and likewise
\[
\zeta(B) = \eta(\pi^{-1}(B)) \leq \eta(\pi^{-1}(\pi_{\lambda}(\pi_{\lambda}^{-1}(B)))) = \eta(\pi_{\lambda}^{-1}(B)) = \zeta_{\lambda}(B)
\]
hence $\zeta(B) = \zeta_{\lambda}(B)$ for $B \subseteq \varphi^{-1}(E^{C})$.

Now let $B \subseteq Z$ be measurable such that $B \subseteq \varphi^{-1}(E)$.  For $x \in E$, measurably choose $\gamma_{x} \in \Gamma$ such that $\lambda x = \gamma_{x} x$.  Write $F_{\gamma} = \{ x \in E : \gamma_{x} = \gamma \}$.  Define the disjoint sets
\[
B_{\gamma} = B \cap \varphi^{-1}(B_{\gamma}).
\]
Then $\theta_{\lambda}(B_{\gamma}) = \gamma B_{\gamma}$ by the definition of $\theta_{\lambda}$.

Suppose first that $\zeta(B) = 0$ but that $\zeta_{\lambda}(B) > 0$.  Then
\[
0 < \zeta_{\lambda}(B) = \eta(\lambda^{-1}\pi^{-1}(\theta_{\lambda}(B))) = \lambda\eta(\pi^{-1}(\theta_{\lambda}(B)))
\]
so, since $\eta$ is in the same measure class as $\lambda\eta$,
\[
0 < \eta(\pi^{-1}(\theta_{\lambda}(B))) = \zeta(\theta_{\lambda}(B)).
\]
Now
\begin{align*}
\zeta(\theta_{\lambda}(B)) &= \zeta(\theta_{\lambda}(\bigsqcup_{\gamma}B_{\gamma})) = \sum_{\gamma} \zeta(\gamma B_{\gamma}) = \sum_{\gamma} \gamma^{-1}\zeta(B_{\gamma})
\end{align*}
and therefore there exists $\gamma \in \Gamma$ such that $\gamma^{-1}\zeta(B_{\gamma}) > 0$.  Since $\zeta$ is $\Gamma$-quasi-invariant then $\zeta(B_{\gamma}) > 0$ for some $\gamma \in \Gamma$.  But then $\zeta(B) \geq \zeta(B_{\gamma}) > 0$ contradicting that $\zeta(B) = 0$.

Suppose now that $\zeta(B) > 0$ but that $\zeta_{\lambda}(B) = 0$.  Observe that
\begin{align*}
\zeta_{\lambda}(B) &= (\pi_{\lambda})_{*}\eta(B) = \eta(\lambda^{-1}\pi^{-1}(\theta_{\lambda}(B))) \\
&= \lambda\eta(\pi^{-1}(\theta_{\lambda}(\bigsqcup_{\gamma} B_{\gamma}))) = \sum_{\gamma} \gamma^{-1}\lambda\eta(\pi^{-1}(B_{\gamma}))
\end{align*}
and therefore $\gamma^{-1}\lambda\eta(\pi^{-1}(B_{\gamma})) = 0$ for all $\gamma \in \Gamma$.  By the $\Lambda$-quasi-invariance of $\eta$, then $\eta(\pi^{-1}(B_{\gamma})) = 0$ for all $\gamma \in \Gamma$.  But then
\[
\zeta(B) = \eta(\pi^{-1}(B)) = \eta(\bigsqcup_{\gamma} \pi^{-1}(B_{\gamma})) = 0
\]
contradicting that $\zeta(B) > 0$.
\end{proof}

\begin{proof}[Proof of Theorem \ref{T:groupoidcontractive1}]  We are now in the situation of having $\pi : (Y,\eta) \to (Z,\zeta)$, $\varphi : (Z,\zeta) \to (X,\nu)$, $\pi_{\lambda} : (Y,\eta) \to (Z,\zeta_{\lambda})$ and $\varphi_{\lambda} : (Z,\zeta_{\lambda}) \to (X,\nu)$ all $\Gamma_{0}$-maps of $\Gamma_{0}$-spaces such that $\varphi \circ \pi = \varphi_{\lambda} \circ \pi_{\lambda} = p$ is a relatively contractive $\Gamma_{0}$-map and such that the disintegration measures $D_{\varphi}(x)$ and $D_{\varphi_{\lambda}}(x)$ are in the same measure class for almost every $x$ (which follows from the previous proposition and Lemma \ref{L:abscontmeas}).  By Theorem \ref{T:relcontractiveunique}, as $(Y,\eta)$ is a product of a contractive space and a measure-preserving space, then $\pi = \pi_{\lambda}$ almost surely and $\zeta_{\lambda} = \zeta$.
Therefore for almost every $y$ such that $p(y) \in E$ we have that
\[
\pi(\lambda y) = \theta_{\lambda}\theta_{\lambda}^{-1}\pi(\lambda y) = \theta_{\lambda} \pi_{\lambda}(y) = \theta_{\lambda} \pi(y).
\]
\end{proof}

\section{Weak Amenability of Actions of Lattices}\label{S:weakamenlatt}

A key fact in our study of stabilizers of commensurators and lattices is that if an action of the commensurator has infinite stabilizers then the restriction of the action to the lattice is weakly amenable (the equivalence relation corresponding to the action of the lattice is amenable).

\begin{theorem}\label{T:main}
Let $\Gamma < G$ be a lattice in a locally compact second countable group and let $\Lambda$ be a countable dense subgroup of $G$ such that $\Gamma <_{c} \Lambda$.

Assume that for every ergodic measure-preserving action of $G$ either the restriction of the action to $\Lambda$ has finite stabilizers or the restriction of the action to $\Gamma$ has finite orbits.

Let $\Lambda \actson (X,\nu)$ be an ergodic measure-preserving action.
Then either $\Lambda$ has finite stabilizers or the restriction of the action to $\Gamma$ is weakly amenable.
\end{theorem}
\begin{proof}
Let $(B,\rho)$ be any Poisson boundary of $G$ with respect to a measure whose support generates $G$.  Then $G \actson (B,\rho)$ is contractive (Jaworski \cite{Ja94}) and amenable (Zimmer \cite{Zi78}).  Then $\Gamma \actson (B,\rho)$ amenably since $\Gamma$ is closed in $G$ (Zimmer \cite{Zi84} Theorem 4.3.5).  Let $A$ be an affine orbital $\Gamma$-space over $(X,\nu)$ (see \cite{Zi84} Section 4.3).  Then there exists $\Gamma$-maps $\pi : B \times X \to A$ and $p : A \to X$ such that $p \circ \pi$ is the natural projection to $X$ (Zimmer \cite{Zi84} Section 4.3).  Since $(B,\rho)$ is a Poisson boundary of $G$, by Proposition 2.4 in \cite{CS14}, the restriction of the action on $(B,\rho)$ to $\Gamma$ makes it a $\Gamma$-contractive space.

By the Intermediate Contractive Factor Theorem (the piecewise version--Theorem \ref{T:groupoidcontractive1}), for almost every $x \in X$ and $b \in B$ and for any $\lambda \in \mathrm{stab}_{\Lambda}(x)$ we have $\pi(\lambda(b,x)) = \pi(b,x)$.

As in the proof of existence of continuous compact models (Lemma \ref{L:compactmodels}; see also \cite{Zi84} Appendix B), there exist Borel models for the spaces $B$, $X$ and $A$, and the maps $\pi$ and $\varphi$.  Moreover, there is a Borel section $X \to (B \to A)$ for $\pi$: for almost every $x$ there is a Borel map $\pi_{x} : B \times \{ x \} \to A_{x}$ where $A_{x} = \varphi^{-1}(\{ x \})$.  The conclusion of the Intermediate Contractive Factor Theorem is that $\pi_{x} \circ \lambda = \pi_{x}$ for almost every $x$ and all $\lambda \in \mathrm{stab}_{\Lambda}(x)$.  Since composition is weakly continuous on the space of Borel maps, treating the $\Lambda$ action as a Borel map $B \to B$, then $\pi_{x} \circ g = \pi_{x}$ for almost every $x$ and all $g \in \overline{\mathrm{stab}_{\Lambda}(x)}$; that is, $\pi(gb,x) = \pi(b,x)$ for almost every $x$ and $b$ and all $g \in \overline{\mathrm{stab}_{\Lambda}(x)}$.

Define the map $s : X \to S(G)$, where $S(G)$ is the Borel space of closed subgroups of $G$ equipped with the conjugation action by $G$, by $s(x) = \overline{\mathrm{stab}_{\Lambda}(x)}$.  Observe that $s(\lambda x) = \overline{\lambda\hspace{.1em}\mathrm{stab}_{\Lambda}(x)\hspace{.03em}\lambda^{-1}} = \lambda \cdot s(x)$ so $s$ is a $\Lambda$-map.  Let $\eta \in P(S(G))$ be $\eta = s_{*}\nu$.  

Let $(\widetilde{X},\widetilde{\nu})$ be an action of $G$ giving rise to the invariant random subgroup $\eta$.  Such an action exists by Theorem \ref{T:gaussian}.  Then $(S(G),\eta)$ is a $G$-factor of $(\widetilde{X},\widetilde{\nu})$ and $\eta = \widetilde{s}_{*}\widetilde{\nu}$ where $\widetilde{s}(\widetilde{x}) = \mathrm{stab}_{G}(\widetilde{x})$.  Then anything true of the stabilizer $\mathrm{stab}_{G}(\widetilde{x})$ of almost every $\widetilde{x} \in \widetilde{X}$ is also true of the closure of the stabilizer $\overline{\mathrm{stab}_{\Lambda}(x)}$ of almost every $x \in X$.

Since $\Lambda$ acts ergodically on $(X,\nu)$ and $(S(G),\eta)$ is a $\Lambda$-factor of $(X,\nu)$ then $\Lambda$ acts ergodically on $(S(G),\eta)$.  Since $\Lambda$ is dense in $G$, $G$ acts ergodically on $(S(G),\eta)$.  Therefore we may assume $G$ acts ergodically on $(\widetilde{X},\widetilde{\nu})$ by Proposition \ref{P:gaussian}.

By hypothesis, the $G$-action on $\widetilde{X}$ either has finite orbits when restricted to $\Gamma$ or the restriction to $\Lambda$ of the action has finite stabilizers.  Suppose first that the action is such that $\Lambda\spacedcap\mathrm{stab}_{G}(\widetilde{x})$ is finite for almost every $\widetilde{x}$ (for some affine orbital $\Gamma$-space over $(X,\nu)$).  Then $\overline{\mathrm{stab}_{\Lambda}(x)}\spacedcap\Lambda$ is finite for almost every $x$ and therefore $\mathrm{stab}_{\Lambda}(x)$ is finite for almost every $x$ meaning the $\Lambda$-action on $(X,\nu)$ has finite stabilizers, in which case the proof is complete.

So assume instead that $G \actson (\widetilde{X},\widetilde{\nu})$ has finite orbits when restricted to $\Gamma$ (for every affine orbital $\Gamma$-space over $(X,\nu)$).  Then $\Gamma \cap \mathrm{stab}_{G}(\widetilde{x})$ has finite index in $\Gamma$ for $\widetilde{\nu}$-almost every $\widetilde{x}$ (since the $\Gamma$-orbits are finite almost surely).  Therefore $\Gamma \cap \overline{\mathrm{stab}_{\Lambda}(x)}$ has finite index in $\Gamma$ for $\nu$-almost every $x$.  Let $\Gamma_{x} = \Gamma \cap \overline{\mathrm{stab}_{\Lambda}(x)}$ be this lattice.  Note that $\pi(\gamma b,x) = \pi(b,x)$ for every $\gamma \in \Gamma_{x}$ and almost every $b \in B$.

For each such $x$, let $A_{x}$ be the fiber over $x$ in $A$ and define the map $\pi_{x} : B \to A_{x}$ by $\pi_{x}(b) = \pi(b,x)$.  Now $(B,\beta)$ is Poisson boundary of $G$ hence is a contractive $\Gamma_{x}$-space (by Proposition 2.4 in \cite{CS14} since $\Gamma_{x}$ is a lattice in $G$) and we will now treat $(A_{x},(\pi_{x})_{*}\beta)$ as a $\Gamma_{x}$-space with the trivial action.  Observe that for any $\gamma \in \Gamma_{x}$ and almost every $b \in B$,
\[
\pi_{x}(\gamma b) = \pi(\gamma b,x) = \pi(b,x) = \pi_{x}(b) = \gamma \pi_{x}(b)
\]
and therefore $\pi_{x}$ is a $\Gamma_{x}$-map meaning that $(A_{x},(\pi_{x})_{*}\beta)$ is a contractive $\Gamma_{x}$-space.  Since the $\Gamma_{x}$-action on it is trivial, $(\pi_{x})_{*}\beta$ must be a point mass.  Let $c_{x} \in A_{x}$ be the point $(\pi_{x})_{*}\beta$ is supported on.  Then $\pi(b,x) = c_{x}$ for almost every $b$ so the mapping $x \mapsto c_{x}$ inverts $\varphi$.  Moreover, this map provides an invariant section for $A$ since for any $\gamma \in \Gamma$ we have that $c_{\gamma x} = \pi(b,\gamma x)$ for almost every $b \in B$ and so $c_{\gamma x} = \pi(\gamma b, \gamma x) = \gamma \pi(b,x) = \gamma c_{x}$ for almost every $b \in B$ and $x \in X$ so $x \mapsto c_{x}$ is $\Gamma$-equivariant.

As this holds for all affine orbital $\Gamma$-spaces over $(X,\nu)$ the action of $\Gamma$ on $(X,\nu)$ is weakly amenable.
\end{proof}


\begin{corollary}\label{C:main}
Let $\Gamma < G$ be a lattice in a locally compact second countable group with property $(T)$ and let $\Lambda$ be a countable dense subgroup of $G$ such that $\Gamma <_{c} \Lambda$.

Assume that for every ergodic measure-preserving action of $G$ either the restriction of the action to $\Lambda$ has finite stabilizers or the restriction of the action to $\Gamma$ has finite orbits.

Then any ergodic measure-preserving action $\Lambda \actson (X,\nu)$ either has finite stabilizers or the restriction of the action to $\Gamma$ has finite orbits.
\end{corollary}
\begin{proof}
By Theorem \ref{T:main}, if the action of $\Lambda$ does not have finite stabilizers then the restriction of the action to $\Gamma$ is weakly amenable.  Since the action is weakly amenable, the associated equivalence relation is weakly amenable (Zimmer \cite{Zi77}).  By the definition of amenable equivalence relation, the equivalence relation on almost every ergodic component is amenable.  Therefore the action on almost every ergodic component is weakly amenable by Zimmer \cite{Zi77}.  Since $\Gamma$ has property $(T)$, by Lemma 1.5 in \cite{SZ94} (an easy consequence of Connes-Feldman-Weiss \cite{CFW}), the action on almost every ergodic component is essentially transitive.  Since $\Gamma$ is discrete, for an ergodic component $(Y,\eta)$ with $y \in Y$ such that $\eta(\Gamma y) = 1$ it follows that $\eta(\gamma y) > 0$ for some $\gamma \in \Gamma$ and then by the invariance of $\eta$ that $\eta(\gamma y)$ is constant and nonzero over $\Gamma$ hence $\Gamma y$ must be a finite set.  As this holds for almost every ergodic component then almost every $\Gamma$-orbit in $X$ must be finite (though the size of each orbit can vary over the ergodic components).
\end{proof}

\section{The One-One Correspondence}\label{S:oneone}

We obtain a correspondence between invariant random subgroups of $\Lambda$ and of the relative profinite completion (see Section \ref{sec:rpf} below) using the previous corollary.

\subsection{Invariant Random Subgroups of Commensurators}

We can restate our previous corollary in terms of invariant random subgroups:
\begin{corollary}\label{C:rsmain}
Let $\Gamma < G$ be a lattice in a locally compact second countable group with property $(T)$ and let $\Lambda$ be a countable dense subgroup of $G$ such that $\Gamma <_{c} \Lambda$.

Assume that for every ergodic measure-preserving action of $G$ either the restriction of the action to $\Lambda$ has finite stabilizers or the restriction of the action to $\Gamma$ has finite orbits.

Then any ergodic invariant random subgroup $\eta \in P(S(\Lambda))$ of $\Lambda$ is either finite ($\eta$-almost every $H \in S(\Lambda)$ is finite) or $\eta$ contains $\Gamma$ up to finite index: for $\eta$-almost every $H \in S(\Lambda)$, we have $[\Gamma : H \cap \Gamma]$ is finite.
\end{corollary}
\begin{proof}
An ergodic invariant random subgroup can always be realized as the stabilizer subgroups of a measure-preserving $\Lambda$-action (Theorem \ref{T:gaussian}).  By Corollary \ref{C:main} this action either has finite stabilizers, in which case the invariant random subgroup is finite, or has finite $\Gamma$-orbits which means that a finite index subgroup of $\Gamma$ fixes each point.
\end{proof}

\subsubsection{The Relative Profinite Completion}\label{sec:rpf}

We recall the construction and some basic facts about the relative profinite completion of commensurated subgroups.  This construction has appeared in the study of commensurated subgroups \cite{Sch80}, \cite{Tz00}, \cite{Tz03}, \cite{capracemonod} and the reader is referred to \cite{SW09} and \cite{CrD11} for more information and proofs of the following basic facts.

\begin{definition}
Let $\Gamma <_{c} \Lambda$ be countable groups with $\Lambda$ commensurating a subgroup $\Gamma$.  Define the map $\tau : \Lambda \to \Symm(\Lambda / \Gamma)$ to be the natural mapping of $\Lambda$ to the symmetry group of the coset space.  Endow $\Symm(\Lambda / \Gamma)$ with the topology of pointwise convergence.  Then $\Symm(\Lambda / \Gamma)$ is a Polish group since $\Lambda$ is countable but in general is not locally compact.

The \textbf{relative profinite completion} of $\Lambda$ over $\Gamma$ is 
\[
\rpf{\Lambda}{\Gamma} := \overline{\tau(\Lambda)}.
\]
\end{definition}

\begin{theorem}
Let $\Gamma <_{c} \Lambda$.  The relative profinite completion $\rpf{\Lambda}{\Gamma}$ is a totally disconnected locally compact group and $\overline{\tau(\Gamma)}$ is a compact open subgroup.
\end{theorem}

\begin{proposition}
Let $\Gamma <_{c} \Lambda$.  Then $\rpf{\Lambda}{\Gamma}$ is compact if and only if $[\Lambda : \Gamma] < \infty$.  In particular, $\rpf{\Lambda}{\Gamma}$ is finite if and only if it is compact.
\end{proposition}

\begin{proposition}\label{P:discreterpf}
Let $\Gamma <_{c} \Lambda$.  Then $\rpf{\Lambda}{\Gamma}$ is discrete if and only if there exists $\Gamma_{0} < \Gamma$ such that $[\Gamma : \Gamma_{0}] < \infty$ and $\Gamma_{0} \normal \Lambda$.
\end{proposition}

\begin{proposition}\label{P:rpfK}
Let $\Gamma <_{c} \Lambda$.  Then $\tau(\Lambda) \cap \overline{\tau(\Gamma)} = \tau(\Gamma)$ and $\tau^{-1}(\overline{\tau(\Gamma)}) = \Gamma$.
\end{proposition}

\begin{proposition}\label{P:idrpf}
Let $H$ be a totally disconnected locally compact group and $K$ be a compact open subgroup of $H$.  Define $\tau_{H,K} : H \to \Symm(H / K)$ as before ($K$ is necessarily commensurated by $H$).  Then $\tau_{H,K}$ is a continuous open map with closed range.

Moreover $\rpf{H}{K}$ is isomorphic to $H / \ker(\tau_{H,K})$ and in fact $\ker(\tau_{H,K})$ is the largest normal subgroup of $H$ that is contained in $K$.
\end{proposition}

\begin{proposition}\label{P:alpha}
Let $B < A$ be any countable groups.  Let $H$ be a locally compact totally disconnected group and $K < H$ a compact open subgroup.  Let $\varphi : A \to H$ be a homomorphism such that $\varphi(A)$ is dense in $H$ and $\varphi^{-1}(K) = B$.

Then $B <_{c} A$ and $\rpf{B}{A}$ is isomorphic to $\rpf{H}{K}$.
\end{proposition}

\subsection{The One-One Correspondence of Invariant Random Subgroups}

\begin{theorem}\label{T:oneone}
Let $\Gamma < G$ be a lattice in a locally compact second countable group with property $(T)$ and let $\Lambda$ be a countable dense subgroup of $G$ such that $\Gamma <_{c} \Lambda$.

Assume that for every ergodic measure-preserving action of $G$ either the restriction of the action to $\Lambda$ has finite stabilizers or the restriction of the action to $\Gamma$ has finite orbits.

Then there is a one-one, onto correspondence between commensurability classes of infinite ergodic invariant random subgroups of $\Lambda$ and commensurability classes of open ergodic invariant random subgroups of $\rpf{\Lambda}{\Gamma}$.
\end{theorem}

We prove some preliminary facts before proving the theorem.

\begin{proposition}\label{P:oneone1}
Let $\Gamma <_{c} \Lambda$ such that every infinite ergodic invariant random subgroup of $\Lambda$ contains $\Gamma$ up to finite index.  Let $\tau : \Lambda \to \Symm(\rpf{\Lambda}{\Gamma})$ be the map defining the relative profinite completion (see Section \ref{sec:rpf}).

The map $c : S(\Lambda) \to S(\rpf{\Lambda}{\Gamma})$ given by $c(L) = \overline{\tau(L)}$ is a $\Lambda$-equivariant map taking infinite ergodic invariant random subgroups of $\Lambda$ to open ergodic invariant random subgroups of $\rpf{\Lambda}{\Gamma}$.
\end{proposition}
\begin{proof}
For notational purposes, write
\[
H = \rpf{\Lambda}{\Gamma} = \overline{\tau(\Lambda)}\quad\quad\text{and}\quad\quad K = \overline{\tau(\Gamma)}
\]
and note that $K$ is a compact open subgroup of $H$.

Let $\nu \in P(S(\Lambda))$ be an infinite ergodic invariant random subgroup of $\Lambda$.  By hypothesis, $\nu$ contains $\Gamma$ up to finite index almost surely.
For $L \in S(\Lambda)$, let
\[
K_{L} = \overline{\tau(L \cap \Gamma)}.
\]
Since $L \cap \Gamma$ has finite index in $\Gamma$ almost surely, we have that $K_{L}$ has finite index in $K$ almost surely:
$[\tau(\Gamma) : \tau(L \cap \Gamma)] \leq [\Gamma : L \cap \Gamma] < \infty$
so $[\overline{\tau(\Gamma)} : \overline{\tau(L \cap \Gamma)}] < \infty$ since finite index passes to closures.
Therefore $K_{L}$ is a compact open subgroup (since $K$ is a compact open subgroup of the locally compact totally disconnected group $H$).  In particular, $c(L)$ contains $K_{L}$ and therefore $c(L)$ is an open subgroup of $H$ almost surely.

Therefore $c$ maps $S(\Lambda)$ to open subgroups of $H$.  Recall that $H \actson S(H)$ by conjugation and therefore $\Lambda \actson S(H)$ by $\lambda \cdot L = \tau(\lambda) L \tau(\lambda)^{-1}$.
For $\lambda \in \Lambda$ and $L \in S(\Lambda)$
\[
c(\lambda \cdot L) = \overline{\tau(\lambda L \lambda^{-1})} = \tau(\lambda) \overline{\tau(L)} \tau(\lambda)^{-1} = \lambda \cdot c(L)
\]
and therefore this mapping is $\Lambda$-equivariant.  Let $\eta \in P(S(H))$ be the pushforward of $\nu$ under this map.  Then $\eta$ is $\tau(\Lambda)$-invariant hence $H$-invariant since $\tau(\Lambda)$ is dense in $H$ and $H$ acts continuously on $S(H)$.  Since $\nu$ is ergodic, so is $\eta$.
\end{proof}

\begin{proposition}\label{P:oneone2}
The map $d : S(\rpf{\Lambda}{\Gamma}) \to S(\Lambda)$ by $d(M) = \tau^{-1}(M \cap \tau(\Lambda))$ has the following properties:
\begin{enumerate}[\hspace{2em}\emph{(}i\emph{)}\hspace{1em}]
\item $c(d(M)) = M$ for all open $M \in S(\rpf{\Lambda}{\Gamma})$;
\item $d(M \cap Q) = d(M) \cap d(Q)$ for all $M,Q \in S(\rpf{\Lambda}{\Gamma})$;
\item $L < d(c(L))$ for all $L \in S(\Lambda)$;
\item $[d(c(L)) : L] < \infty$ for all $L \in S(\Lambda)$ such that $[\Gamma : \Gamma \cap L] < \infty$; and
\item for open $M,Q \in S(\rpf{\Lambda}{\Gamma})$ with $Q < M$, if $[M : Q] < \infty$ then $[d(M) : d(Q)] < \infty$.
\end{enumerate}
\end{proposition}
\begin{proof}
Let $M$ be an open subgroup of $\rpf{\Lambda}{\Gamma}$.  Then
\[
c(d(M)) = \overline{\tau(\tau^{-1}(M \cap \tau(\Lambda)))} = \overline{M \cap \tau(\Lambda)} = M
\]
since $M$ is open (hence also closed) and $\tau(\Lambda)$ is dense in $\rpf{\Lambda}{\Gamma}$, proving the first statement.

Now let $M,Q \in S(\rpf{\Lambda}{\Gamma})$.  Then
\[
d(M) \cap d(Q) = \tau^{-1}(M \cap \tau(\Lambda)) \cap \tau^{-1}(Q \cap \tau(\Lambda)) = \tau^{-1}(M \cap Q \cap \tau(\Lambda)) = d(M \cap Q)
\]
proving the second statement.

Let $L \in S(\Lambda)$.  Then
$d(c(L)) = \tau^{-1}(\overline{\tau(L)} \cap \tau(\Lambda))$
and $\tau(L) \subseteq \tau(\Lambda)$ so $L$ is a subgroup of $d(c(L))$, proving the third statement.
Now let $L$ be an infinite subgroup of $\Lambda$.  Define the group
\[
Q = c(L) \cap \tau(\Lambda).
\]
Then $\tau(L)$ is dense in $Q$ and $K = \overline{\tau(\Gamma)}$ is open in $H = \overline{\tau(\Lambda)}$ so
$Q \subseteq K \tau(L)$.
Let $h \in Q$, then $h = kn$ for some $k \in K$ and $n \in \tau(L)$.  Therefore $hn^{-1} \in K$ and also $hn^{-1} \in \tau(\Lambda)$.  By Proposition \ref{P:rpfK}, $K \cap \tau(\Lambda) = \tau(\Gamma)$ so we have that $hn^{-1} \in \tau(\Gamma)$.  Hence
\[
Q \subseteq \tau(\Gamma)\tau(L) = \tau(\Gamma L).
\]

We will use the notation $[A:B]$ when $A$ and $B$ are merely subsets (and not necessarily subgroups) to refer to the smallest number of elements of $A$ such that the left translates of $B$ by those elements cover $A$.  Observe that, since $L$ contains $\Gamma$ up to finite index,
\[
[Q : \tau(L)] \leq [\tau(\Gamma L) : \tau(L)] \leq [\Gamma L : L] = [\Gamma : \Gamma \cap L] < \infty
\]
so $Q$ is a finite index extension of $\tau(L)$.

Now write $R = \tau^{-1}(Q) = \tau^{-1}(\overline{\tau(L)} \cap \tau(\Lambda))$.  Then $\tau(R) = Q$.  Write $R_{0} = R \cap \ker(\tau)$ and $L_{0} = L \cap \ker(\tau)$.  Since $R_{0} \subseteq \ker(\tau)$ and $\ker(\tau) \subseteq \Gamma$, by the isomorphism theorems we have that
\[
[R_{0} : L_{0}] \leq [\ker(\tau) : L \cap \ker(\tau)] = [L \ker(\tau) : L] \leq [L \Gamma : L] = [\Gamma : \Gamma \cap L] < \infty.
\]
By Lemma \ref{L:indextricks} below,
\[
[R : L] \leq [\tau(R) : \tau(L)] [R_{0} : L_{0}] = [Q : \tau(L)] [R_{0} : L_{0}] < \infty
\]
since $Q$ is a finite index extension of $\tau(L)$.  Therefore $L$ has finite index in $\tau^{-1}(\overline{\tau(L)} \cap \tau(\Lambda)) = d(c(L))$ proving the fourth statement.

Now let $M,Q$ be open subgroups of $\rpf{\Lambda}{\Gamma}$ such that $[M : Q] < \infty$.
Observe that $Q \cap \overline{\tau(\Gamma)}$ is then open so $\overline{\tau(\Gamma)} / Q \cap \overline{\tau(\Gamma)}$ is both compact and discrete, hence finite.  Since $\overline{\tau(d(Q))} = Q$ then we have $[\overline{\tau(\Gamma)} : \overline{\tau(d(Q))} \cap\overline{\tau(\Gamma)}] < \infty$.  
Therefore $[\tau(\Gamma) : \tau(d(Q)) \cap \tau(\Gamma)] < \infty$.

Since $\mathrm{ker}(\tau) = \tau^{-1}(\{ e \}) = d(\{ e \})$,
\[
[\mathrm{ker}(\tau) : d(Q) \cap \mathrm{ker}(\tau)]
= [d(\{ e \}) : d(Q) \cap d(\{ e \})] = [d(\{ e \}) : d(Q \cap \{ e \})] = 1
\]
hence by Lemma \ref{L:indextricks},
\[
[\Gamma : \Gamma \cap d(Q)] \leq [\tau(\Gamma) : \tau(\Gamma \cap d(Q))][\mathrm{ker}(\tau) : d(Q) \cap \mathrm{ker}(\tau)]
= [\tau(\Gamma) : \tau(\Gamma \cap d(Q))] < \infty.
\]
Similarly, $[d(M) \cap \mathrm{ker}(\tau) : d(Q) \cap \mathrm{ker}(\tau)] = 1$, so
by Lemma \ref{L:indextricks},
\[
[d(M) : d(Q)] \leq [\tau(d(M)) : \tau(d(Q))][d(M) \cap \mathrm{ker}(\tau) : d(Q) \cap \mathrm{ker}(\tau)] < \infty
\]
since $[\overline{\tau(d(M))} : \overline{\tau(d(Q))}] = [c(d(M)) : c(d(Q))] = [M : Q] < \infty$, proving the final statement.
\end{proof}

\begin{lemma}\label{L:indextricks}
Let $\phi : C \to D$ be a group homomorphism and $A \subseteq C$ and $B \subseteq A$ be subsets.  Define $[A : B]$ to be the smallest number $n$ such that there exists $a_{1},\ldots,a_{n} \in A$ with $A \subseteq \cup_{j} a_{j}B$.  Then
\[
[A : B] \leq [\phi(A) : \phi(B)] [\mathrm{ker}(\phi) : B \cap \mathrm{ker}(\phi)].
\]
\end{lemma}
\begin{proof}
Assume both indices on the right are finite, otherwise there is nothing to prove.  Let $X$ be a finite system of representatives for $\phi(A) / \phi(B)$ (that is, $\phi(A) \subseteq \cup_{x \in X} x \phi(B)$).  Let $Y$ be a finite system of representatives for $\mathrm{ker}(\phi) / B \cap \mathrm{ker}(\phi)$.  Let $\widetilde{X}$ contain one element $\widetilde{x}$ for each $x \in X$ such that $\phi(\widetilde{x}) = x$ so $|\widetilde{X}| = |X|$.

Let $a \in A$.  Then $\phi(a) = x \phi(b)$ for some $x \in X$ and $b \in B$.  So $\phi(\widetilde{x}^{-1}ab^{-1}) = e$ hence $\widetilde{x}^{-1}ab^{-1} \in \mathrm{ker}(\phi)$ and therefore $\widetilde{x}^{-1}ab^{-1} = y k$ for some $y \in Y$ and some $k \in B \cap \mathrm{ker}(\phi)$.  Then $a = \widetilde{x} y k b$.  Now $kb \in B$ and there are at most $|\widetilde{X}| |Y| = |X| |Y|$ choices for $\widetilde{x}y$ so the claim follows.
\end{proof}

\begin{proof}[Proof of Theorem \ref{T:oneone}]
Let $c$ and $d$ denote the maps in the previous propositions.  The correspondence will be given by the map $c$ on commensurability classes.  By Corollary \ref{C:rsmain}, any infinite ergodic invariant random subgroup $\nu$ of $\Lambda$ contains $\Gamma$ up to finite index almost surely.  By Proposition \ref{P:oneone1}, $c_{*}\nu$ is then an open ergodic invariant random subgroup of $\rpf{\Lambda}{\Gamma}$.

Let $\nu_{1}$ and $\nu_{2}$ be infinite ergodic invariant random subgroups of $\Lambda$ such that $\nu_{1}$ and $\nu_{2}$ are commensurate invariant random subgroups.  Let $\alpha \in P(S(\Lambda) \times S(\Lambda))$ be a joining of $\eta_{1}$ and $\eta_{2}$ witnessing the commensuration.  Define $\beta \in P(S(\rpf{\Lambda}{\Gamma}) \times S(\rpf{\Lambda}{\Gamma}))$ by $\beta = (c \times c)_{*}\alpha$.  Then $\beta$ is a joining of $c_{*}\nu_{1}$ and $c_{*}\nu_{2}$ that is clearly measure-preserving.  Since, in general $\overline{X \cap Y} \subseteq \overline{X} \cap \overline{Y}$, for any $H,L \in S(\Lambda)$,
\[
[c(H) : c(H) \cap c(L)] = [\overline{\tau(H)} : \overline{\tau(H)} \cap \overline{\tau(L)}]
\leq [\overline{\tau(H)} : \overline{\tau(H) \cap \tau(L)}].
\]
For $\alpha$-almost every $H,L$, we have that $[H : H \cap L] < \infty$ and since $\tau$ is a homomorphism then $[\tau(H) : \tau(H) \cap \tau(L)] < \infty$.  Therefore $[c(H) : c(H) \cap c(L)] < \infty$ since finite index passes to closures.  Likewise, $[c(L) : c(H) \cap c(L)] < \infty$.  

Hence for $\beta$-almost every $M,Q$, the subgroup $M \cap Q$ has finite index in both $M$ and $Q$.  Therefore $\beta$ makes $c_{*}\eta_{1}$ and $c_{*}\eta_{2}$ commensurate invariant random subgroups.  Hence $c$ defines a correspondence from commensurability classes of infinite ergodic invariant random subgroups of $\Lambda$ to commensurability classes of open ergodic invariant random subgroups of $\rpf{\Lambda}{\Gamma}$.

Now let $\nu_{1}$ and $\nu_{2}$ be infinite ergodic invariant random subgroups of $\Lambda$ such that $c_{*}\nu_{1}$ and $c_{*}\nu_{2}$ are commensurate open ergodic invariant random subgroups of $\rpf{\Lambda}{\Gamma}$.  Let $\beta \in P(S(\rpf{\Lambda}{\Gamma}) \times S(\rpf{\Lambda}{\Gamma}))$ be a joining of $c_{*}\nu_{1}$ and $c_{*}\nu_{2}$ such that for $\beta$-almost every $M,Q$, the subgroup $M \cap Q$ has finite index in $M$ and $Q$.  Define $\nu_{3} = d_{*}c_{*}\nu_{1}$.  Then by Proposition \ref{P:oneone2} (iv), $[d(c(L)) : L] < \infty$ for $\nu_{1}$-almost every $L \in S(\Lambda)$.  Define $\rho \in P(S(\Lambda) \times S(\Lambda))$ by
\[
\rho = \int_{L} \delta_{L} \times \delta_{d(c(L))}~d\nu_{1}(L).
\]
Then $\rho$ is a joining of $\nu_{1}$ and $\nu_{3}$ and clearly $L \cap d(c(L)) = L$ has finite index in both $L$ and $d(c(L))$ almost surely so $\rho$ makes $\nu_{1}$ and $\nu_{3}$ commensurate invariant random subgroups.  Likewise $\nu_{2}$ and $\nu_{4} = d_{*}c_{*}\nu_{2}$ are commensurate invariant random subgroups.  Since commensurability is an equivalence relation (Proposition \ref{T:commequiv}), it is enough to show that $\nu_{3}$ and $\nu_{4}$ are commensurate.

Define $\alpha \in P(S(\Lambda) \times S(\Lambda))$ by $\alpha = (d \times d)_{*}\beta$.  Then $\alpha$ is a joining of $d_{*}c_{*}\nu_{1} = \nu_{3}$ and $d_{*}c_{*}\nu_{2} = \nu_{4}$.  By Proposition \ref{P:oneone2} (ii), for open $M,Q \in S(\rpf{\Lambda}{\Gamma})$,
$d(M) \cap d(Q) = d(M \cap Q)$.
Observe that $\nu_{3}$ and $\nu_{4}$ are infinite ergodic invariant random subgroup of $\Lambda$ hence $d(M)$ and $d(Q)$ both contain $\Gamma$ up to finite index almost surely.  Then $d(M \cap Q)$ contains $\Gamma$ up to finite index almost surely.  For $\beta$-almost every $M,Q$ we also know that $[M : M \cap Q] < \infty$.
So by Proposition \ref{P:oneone2} (v),
\[
[d(M) : d(M) \cap d(Q)] = [d(M) : d(M \cap Q)] < \infty
\]
almost surely.  Hence for $\alpha$-almost every $H,L$ the subgroup $H \cap L$ has finite index in both so $\nu_{3}$ and $\nu_{4}$ are commensurate invariant random subgroups.  Therefore the correspondence is one-one.

Let $\eta \in P(S(\rpf{\Lambda}{\Gamma}))$ be an open ergodic invariant random subgroup of $\rpf{\Lambda}{\Gamma}$.  For $M$ an open subgroup of $\rpf{\Lambda}{\Gamma}$ we have that $d(M) = \tau^{-1}(M \cap \tau(\Lambda))$ is infinite since otherwise $M \cap \tau(\Lambda)$ is finite but $\tau(\Lambda)$ is dense.  Therefore $d_{*}\eta$ is an infinite invariant random subgroup of $\Lambda$ and must be ergodic since $c_{*}d_{*}\eta = \eta$ by Proposition \ref{P:oneone2} (i).  Therefore the correspondence is onto.
\end{proof}

\subsection{The Dichotomy for Actions of Commensurators}

We now are ready to state the conclusion of our study of stabilizer subgroups that will be the main ingredient in the various consequences we prove in the rest of the paper:
\begin{corollary}\label{T:useful}
Let $\Gamma < G$ be a lattice in a locally compact second countable group with property $(T)$ and let $\Lambda$ be a countable dense subgroup of $G$ such that $\Gamma <_{c} \Lambda$.

Assume that for every ergodic measure-preserving action of $G$ either the restriction of the action to $\Lambda$ has finite stabilizers or the restriction of the action to $\Gamma$ has finite orbits.

Assume that every ergodic measure-preserving action of $\rpf{\Lambda}{\Gamma}$ with open stabilizer subgroups is necessarily on the trivial space.

Then any ergodic measure-preserving action of $\Lambda$ on a probability space either has finite orbits or has finite stabilizers.
\end{corollary}
\begin{proof}
Let $\Lambda \actson (X,\nu)$ be an ergodic measure-preserving action that does not have finite stabilizers.  By the one-one correspondence theorem, the invariant random subgroup of stabilizer subgroups corresponds to an ergodic open invariant random subgroup $\eta$ of $H = \rpf{\Lambda}{\Gamma}$.  This invariant random subgroup corresponds to an ergodic action of $H$ with open stabilizer groups and so by hypothesis then $\eta = \delta_{H}$ meaning $\overline{\tau(\mathrm{stab}_{\Lambda}(x))} = H$ for almost every $x$.

By the one-one correspondence construction we then have that
\[
[\Lambda : \mathrm{stab}_{\Lambda}(x)] = [\tau^{-1}(\overline{\mathrm{stab}_{\Lambda}(x)} \cap \tau(\Lambda)) : \mathrm{stab}_{\Lambda}(x)] = [d(c(\mathrm{stab}_{\Lambda}(x))) : \mathrm{stab}_{\Lambda}(x)] < \infty
\]
for almost every $x$.  
This means that almost every $\Lambda$-orbit is finite so by ergodicity $(X,\nu)$ consists of exactly one such orbit.
\end{proof}

\section{Howe-Moore Groups}\label{S:howemoore}

We now discuss the properties one can impose on the ambient group $G$ to ensure that for every nontrivial ergodic measure-preserving action of $G$ the restriction of the action to $\Lambda$ has finite stabilizers.  The main property we impose on the ambient group will be the Howe-Moore property.

\subsection{Actions of Subgroups of Simple Lie Groups}

The results in this subsection are consequences of the Stuck-Zimmer Theorem \cite{SZ94} and also follow from earlier work by Zimmer, \cite{zimmer3} Lemma 6, and of Iozzi, \cite{iozzi} Proposition 2.1, showing that the stabilizers of any nontrivial irreducible action of a semisimple real Lie group are discrete.  However, we opt to include the following elementary argument proving what we need directly.

\begin{theorem}\label{T:lieactionssubgroups}
Let $G$ be a connected (real) Lie group with trivial center and let $\Lambda < G$ be any countable subgroup.  Let $G \actson (X,\nu)$ be a faithful weakly mixing measure-preserving action.  Then the restriction of the action to $\Lambda$ is essentially free.
\end{theorem}
\begin{proof}
For $x \in X$ let $C(x)$ be the connected component of the identity in the stabilizer subgroup $\mathrm{stab}_{G}(x)$.  Let $n(x)$ be the dimension of $C(x)$.  Then $n(gx)$ is the dimension of $C(gx) = gC(x)g^{-1}$ hence $n(x)$ is $G$-invariant.  By ergodicity then $n(x) = n$ is constant almost surely.

Since the action of $G$ is weakly mixing, the diagonal action $G \actson (X^{2},\nu^{2})$ is ergodic.  Let $n_{1}(x,y)$ be the dimension of the connected component of the identity $C(x,y)$ in $\mathrm{stab}_{G}(x,y)$ and then $n_{1}(x,y) = n_{1}$ is constant almost surely by ergodicity.

When $n = 0$, the stabilizer subgroup $\mathrm{stab}_{G}(x)$ is discrete for almost every $x$ (since the stabilizer subgroup is closed).  Assume now that $n \ne 0$.  Since $\mathrm{stab}_{G}(x,y) = \mathrm{stab}_{G}(x) \cap \mathrm{stab}_{G}(y)$ we have that $C(x,y) = C(x) \cap C(y)$.  Suppose that $n = n_{1}$.  Then for almost every $x$ and $y$ we have that $C(x,y) = C(x) \cap C(y)$ has the same dimension as $C(x)$ and $C(y)$.

If $C(x,y) = C(x) \cap C(y)$ has the same dimension as $C(x)$ and $C(y)$ then in fact $C(x) = C(y)$ since both are connected. So if $n > 0$ this then means there is a positive dimension subgroup in the kernel of the action contradicting that the action is faithful.  So if $n = n_{1}$ then $n = n_{1} = 0$.

So instead we have that $n_{1} < n$.  Proceeding by induction, since $G$ acts ergodically on $(X^{m},\nu^{m})$ for any $m \in \mathbb{N}$, we conclude that for almost every $\tilde{x} \in X^{n+1}$ the stabilizer subgroup $\mathrm{stab}_{G}(\tilde{x})$ is discrete.

Since the action $\Lambda \actson (X,\nu)$ is essentially free if and only if the diagonal action $\Lambda \actson (X^{n+1},\nu^{n+1})$ is essentially free, the conclusion now follows from the following proposition.
\end{proof}

\begin{proposition}
Let $G$ be a connected locally compact second countable group and $\Lambda < G$ a countable subgroup such that $\Lambda$ does not intersect the center of $G$.  Let $G \actson (X,\nu)$ be a measure-preserving action such that almost every stabilizer subgroup is discrete.  Then the restriction of the action to $\Lambda$ is essentially free.
\end{proposition}
\begin{proof}
Suppose the $\Lambda$-action is not essentially free.  Then there exists $\lambda \in \Lambda$, $\lambda \ne e$, such that $E = \{ x \in X : \lambda x = x \}$ has positive measure.  Since $G$ is connected and $\lambda \notin Z(G)$, the centralizer subgroup of $\lambda$ cannot contain an open neighborhood the identity in $G$.  Therefore there exists $g_{n} \to e$ in $G$ such that $g_{n}\lambda g_{n}^{-1} \ne \lambda$ for all $n$.  Note that $\nu(g_{n}E \symdiff E) \to 0$ since $G$ acts continuously.  Take a subsequence along which $\nu(g_{n}E \symdiff E) < 2^{-n-1} \nu(E)$.  Then
\[
\nu(E \cap \bigcap_{n} g_{n}E) = \nu(E) - \nu(E \symdiff \bigcap_{n} g_{n}E) \geq \nu(E) - \sum_{n=1}^{\infty} \nu(E \symdiff g_{n}E) > \frac{1}{2}\nu(E) > 0
\]
For $x \in E \cap (\cap_{n} g_{n}E)$ we have that $\lambda x = x$ and $g_{n}\lambda g_{n}^{-1} x = x$ for all $n$, hence $g_{n}\lambda g_{n}^{-1} \in \mathrm{stab}_{G}(x)$ and $\lambda \in \mathrm{stab}_{G}(x)$.  But $g_{n}\lambda g_{n}^{-1} \ne \lambda$ and $g_{n}\lambda g_{n}^{-1} \to \lambda$ contradicting that $\mathrm{stab}_{G}(x)$ is discrete almost everywhere.
\end{proof}

\begin{theorem}\label{T:productGconnected}
Let $G$ be a product of noncompact connected simple locally compact second countable groups with the Howe-Moore property.  Let $\Lambda < G$ be a countable subgroup of $G$ such that the $\Lambda$ intersection with any proper subproduct of $G$ is finite.  Then the restriction of any nontrivial ergodic measure-preserving action of $G$ to $\Lambda$ has finite stabilizers.
\end{theorem}
\begin{proof}
Let $G \actson (X,\nu)$ be a nontrivial ergodic action.  Then the kernel of the action is some subproduct $G^{\prime}$ of the $G$-factors (as all are normal).  Let $G_{0} = G / G^{\prime}$.  By a result of Rothman \cite{rothman}, each simple factor of $G_{0}$, being a Howe-Moore group that is simple and connected, is necessarily a simple Lie group.  Hence $G_{0}$ is a minimally almost periodic group, being a semisimple Lie group without compact factors, so any ergodic action of $G_{0}$ is weakly mixing.  Since each factor is simple, $G_{0}$ has trivial center.  Therefore $G_{0} \actson (X,\nu)$ is a faithful weakly mixing action and so Theorem \ref{T:lieactionssubgroups} implies that $\mathrm{proj}_{G_{0}}~\Lambda$ acts essentially freely.  Therefore $\mathrm{proj}_{G_{0}}~\mathrm{stab}_{\Lambda}(x) = \{ e \}$ for almost every $x$ hence $\mathrm{stab}_{\Lambda}(x) \subseteq G^{\prime}$ and $|\Lambda \cap G^{\prime}| < \infty$ by hypothesis.
\end{proof}

\subsection{Actions of Lattices in Howe-Moore Groups}

In the case when $G$ is not connected we have a similar result:
\begin{definition}
A countable discrete group $\Gamma$ is \textbf{locally finite} when every finitely generated subgroup is finite.
\end{definition}

\begin{theorem}\label{T:mixing1}
Let $G$ be a locally compact second countable group and $\Gamma < G$ be a lattice in $G$.  Let $G \actson (X,\nu)$ be an ergodic measure-preserving action of $G$ such that the restriction of the action to $\Gamma$ is mixing.  Then either the kernel of the $G$-action is noncompact or $\mathrm{stab}_{\Gamma}(x) = \mathrm{stab}_{G}(x) \cap \Gamma$ is locally finite almost surely.
\end{theorem}
\begin{proof}
Let $E = \{ x \in X : \mathrm{stab}_{\Gamma}(x) \text{ is locally finite} \}$.  Then $E$ is $\Gamma$-invariant.  Assume that $\nu(E) < 1$.  By ergodicity then $\nu(E) = 0$.  Then $\mathrm{stab}_{\Gamma}(x)$ contains a finitely generated infinite subgroup for almost every $x$.  Since there are countably many finitely generated infinite subgroups of $\Gamma$, there exists an infinite finitely generated subgroup $\Gamma_{0} < \Gamma$ and a positive measure set $F \subseteq X$ such that $\Gamma_{0} < \mathrm{stab}_{\Gamma}(x)$ for each $x \in F$.

Since the action of $\Gamma$ on $(X,\nu)$ is mixing, we have that $\nu(F) = 1$ (as $\Gamma_{0}$ is infinite so is unbounded in $\Gamma$ and therefore must also be mixing but $\Gamma_{0}$ acts trivially on $F$).  Therefore $\Gamma_{0}$ is contained in the kernel of the $G$-action which is therefore noncompact (as $\Gamma$ is a lattice so $\Gamma_{0}$ is unbounded in $G$).
\end{proof}

\begin{corollary}\label{C:howemoorea}
Let $G$ be a noncompact  locally compact second countable group with the Howe-Moore property and $\Gamma < G$ be a lattice.  Let $G \actson (X,\nu)$ be a nontrivial ergodic measure-preserving action.  Then $\mathrm{stab}_{\Gamma}(x)$ is locally finite for almost every $x$.
\end{corollary}
\begin{proof}
The Howe-Moore property applied to the Koopman representation for $G \actson (X,\nu)$ implies that $G \actson (X,\nu)$ is mixing (Schmidt \cite{schmidtmixing} Theorem 3.6).  Since $\Gamma$ is a lattice in $G$ then $\Gamma \actson (X,\nu)$ by restriction is also mixing.  If $\mathrm{stab}_{\Gamma}(x)$ is not locally finite almost surely then by Theorem \ref{T:mixing1} the kernel $N$ of the $G$-action is a noncompact closed normal subgroup.  Since $G$ has Howe-Moore any proper normal subgroup is compact so the kernel is all of $G$.
\end{proof}

The following argument is due to R.~Tucker-Drob \cite{tuckerdrob} and we are grateful to him for allowing us to present it here:
\begin{corollary}\label{C:howemoore}
Let $G$ be a noncompact  locally compact second countable group with the Howe-Moore property and $\Gamma < G$ be a lattice.  Let $G \actson (X,\nu)$ be a nontrivial ergodic measure-preserving action.  Then $\mathrm{stab}_{\Gamma}(x)$ is finite for almost every $x$.
\end{corollary}
\begin{proof}(Tucker-Drob \cite{tuckerdrob})
By the previous corollary, the stabilizer subgroups are locally finite almost surely.  Hall and Kulatilaka \cite{hallkulatilaka} showed that any infinite locally finite group contains an infinite abelian subgroup.

Since $G$ has Howe-Moore, the action is mixing and has compact kernel $K$.  Let $G^{\prime} = G / K$ and $\Gamma^{\prime} = \Gamma / \Gamma \cap K$.  Then $G^{\prime} \actson (X,\nu)$ is a faithful mixing action and $\Gamma^{\prime}$ is a lattice in $G^{\prime}$.

Let $\gamma \in \Gamma^{\prime}$, $\gamma \ne e$, such that there exists an infinite abelian subgroup $A < \Gamma^{\prime}$ with $\gamma \in A$.  Let $E_{\gamma} = \{ x \in X : \gamma x = x \}$.  Then for $a \in A$, $\gamma a x = a \gamma x = a x$ so $E_{\gamma}$ is an $A$-invariant set.  Since the action is mixing and faithful, and since $A$ is infinite and discrete, $\nu(E_{\gamma}) = 0$.

Let $F = \{ x \in X : \mathrm{stab}_{\Gamma^{\prime}}(x) \text{ contains an infinite abelian subgroup} \}$ and suppose $\nu(F) > 0$.  Since $\Gamma^{\prime}$ is countable there then exists some $\gamma \ne e$ such that $\nu\{ x \in F : \gamma x = x \} > 0$.  But this contradicts the above since $\gamma$ is then contained in an infinite abelian subgroup.

Hence $\mathrm{stab}_{\Gamma^{\prime}}(x)$ is finite for almost every $x$ and since $\Gamma \cap K$ is finite, then $\mathrm{stab}_{\Gamma}(x)$ is also finite for almost every $x$.
\end{proof}

\subsection{A Normal Subgroup of the Commensurator}

\begin{proposition}\label{P:upgrade}
Let $\Gamma$ be a finitely generated countable group that is not virtually abelian and let $\Lambda$ be a countable group such that $\Gamma <_{c} \Lambda$.  Let $\Lambda \actson (X,\nu)$ be a measure-preserving action such that $\mathrm{stab}_{\Lambda}(x)$ is infinite almost surely.  If $\mathrm{stab}_{\Gamma}(x)$ are finite on a positive measure set then $\Lambda$ contains an infinite normal subgroup $N \normal \Lambda$ such that $[\Gamma : \Gamma \cap N] = \infty$.
\end{proposition}
\begin{proof}
Since there are only countably many finite subgroups of $\Gamma$, let us assume there exists some finite subgroup $\Sigma < \Gamma$ such that $\mathrm{stab}_{\Gamma}(x) = \Sigma$ for all $x \in E$ where $\nu(E) > 0$.

For $\lambda \in \Lambda$ define the set
\[
E_{\lambda} = \{ x \in X : \lambda x = x \} \cap E
\]
and denote by $\Gamma_{\lambda} = \Gamma \cap \lambda \Gamma \lambda^{-1}$ the subgroup with finite index in $\Gamma$ and $\lambda \Gamma \lambda^{-1}$.

By hypothesis, $\nu(E_{\lambda}) > 0$ for infinitely many $\lambda \in \Lambda$ (since otherwise $\mathrm{stab}_{\Lambda}(x)$ is finite for all $x \in E$ which has positive measure, see \cite{vershiktotallynonfree}).  For such $\lambda$, define $\overline{E_\lambda} \subset X$ to be $\cup_{\gamma \in \Gamma_{\lambda}} \gamma E_\lambda$.  For any $\epsilon > 0$ there exists a finite set $F \subseteq \Gamma_{\lambda}$ such that $\nu(\overline{E_{\lambda}}) - \nu(\cup_{f \in F} f E_{\lambda}) < \epsilon$.  Take $\epsilon = \nu(E_{\lambda})$.  Then there is a finite set $F \subseteq \Gamma_{\lambda}$ such that $\nu(\cup_{f \in F} f E_{\lambda}) > \nu(\overline{E_{\lambda}}) - \nu(E_{\lambda})$.

Then for each $\gamma \in \Gamma_{\lambda}$ there exists $f \in F$ such that $\nu(\gamma E_{\lambda} \cap f E_{\lambda}) > 0$ because $\gamma E_{\lambda} \subseteq \overline{E_{\lambda}}$ and $\nu(\gamma E_{\lambda}) = \nu(E_{\lambda})$.  For $x \in f^{-1} \gamma E_{\lambda} \cap E_{\lambda}$ we have that $x \in E$, $\lambda x = x$ and $\lambda \gamma^{-1}fx = \gamma^{-1}fx$, therefore
\[
\lambda^{-1} (f^{-1} \gamma) \lambda (f^{-1} \gamma)^{-1} x = \lambda^{-1} (f^{-1}\gamma) (f^{-1}\gamma)^{-1}x = \lambda^{-1} x = x
\]
and so $\lambda^{-1} (f^{-1} \gamma) \lambda (f^{-1} \gamma)^{-1} \in \Sigma$.  This in turn means that $\gamma \lambda \gamma^{-1} \in f \lambda \Sigma f^{-1} \subseteq F \lambda \Sigma F^{-1}$.  Since $F$ and $\Sigma$ are finite then the centralizer
\[
C_{\Gamma_{\lambda}}(\lambda) = \{ \gamma \in \Gamma_{\lambda} : \gamma \lambda \gamma^{-1} = \lambda \}
\]
has finite index in $\Gamma_{\lambda}$.  Therefore $[\Gamma : C_{\Gamma}(\lambda)] < \infty$ since $\Gamma_{\lambda}$ has finite index in $\Gamma$.

Consider the subgroup
\[
N = \{ \lambda \in \Lambda : [\Gamma : C_{\Gamma}(\lambda)] < \infty \}
\]
which is infinite by the above (it is a subgroup since $C_{\Gamma}(\lambda_{1}) \cap C_{\Gamma}(\lambda_{2}) \subseteq C_{\Gamma}(\lambda_{1}\lambda_{2})$).  Since $\Gamma <_{c} \Lambda$, for $\lambda_{0} \in \Lambda$,
\begin{align*}
\lambda_{0} N \lambda_{0}^{-1} &= \{ \lambda \in \Lambda : [\Gamma : C_{\Gamma}(\lambda_{0}^{-1}\lambda\lambda_{0})] < \infty \} \\
&= \{ \lambda \in \Lambda : [\Gamma : \lambda_{0}^{-1}C_{\Gamma}(\lambda)\lambda_{0}] < \infty \} \\
&= \{ \lambda \in \Lambda : [\lambda_{0} \Gamma \lambda_{0}^{-1} : C_{\Gamma}(\lambda)] < \infty \} \\
&= \{ \lambda \in \Lambda : [\Gamma : C_{\Gamma}(\lambda)] < \infty \} = N
\end{align*}
where the last line follows since $\Gamma \cap \lambda_{0} \Gamma \lambda_{0}^{-1}$ has finite index in $\Gamma$ by commensuration.

Therefore $N$ is an infinite normal subgroup of $\Lambda$.  If $[\Gamma : \Gamma \cap N] < \infty$ then there exists $\Gamma_{0} = \Gamma \cap N$ of finite index in $\Gamma$ such that for every $\gamma \in \Gamma_{0}$ we have that $[\Gamma_{0} : C_{\Gamma_{0}}(\gamma)] < \infty$.  Hence for any finite set $F \subseteq \Gamma_{0}$ we have that $[\Gamma_{0} : C_{\Gamma_{0}}(F)] < \infty$.  As $\Gamma$ is finitely generated so is $\Gamma_{0}$ so let $S$ be a finite generating set of $\Gamma_{0}$ and then $C_{\Gamma_{0}}(S)$ has finite index in $\Gamma_{0}$.  But $C_{\Gamma_{0}}(S)$ commutes with $\Gamma_{0}$ so $\Gamma_{0}$ is virtually abelian hence so is $\Gamma$.
\end{proof}

\subsection{Ensuring Actions of the Ambient Group ``Behave"}

\begin{theorem}\label{T:howemooreplan}
Let $G$ be a noncompact compactly generated locally compact second countable group with the Howe-Moore property.  Let $\Gamma < G$ be a finitely generated lattice that is not virtually abelian and let $\Lambda < G$ be a dense subgroup such that $\Gamma <_{c} \Lambda$.  Assume that for every compact normal subgroup $M \normal G$ we have that $|M \cap \Lambda| < \infty$.

Then for any nontrivial ergodic measure-preserving action of $G$ the restriction of the action to $\Lambda$ has finite stabilizers.
\end{theorem}
\begin{proof}
By Corollary \ref{C:howemoore}, since $G$ has the Howe-Moore property almost every $\Gamma$-stabilizer is finite, hence the restriction of the action to $\Gamma$ has finite stabilizers.  Suppose the restriction of the action to $\Lambda$ does not have finite stabilizers.  Then by Proposition \ref{P:upgrade} there exists an infinite normal subgroup $N \normal \Lambda$ such that $[\Gamma : \Gamma \cap N] < \infty$.  But $\Gamma < \Lambda < G$ satisfy the hypotheses of the Normal Subgroup Theorem for Commensurators due to Shalom and the first author \cite{CS14}, as any proper closed normal subgroup is compact by Howe-Moore, so for any normal subgroup $N \normal \Lambda$ we have that  $[\Gamma : \Gamma \cap N] < \infty$ or $|N| < \infty$.  This contradiction means the $\Lambda$-action has finite stabilizers.
\end{proof}

\subsection{Ensuring Actions of the Relative Profinite Completion ``Behave"}

To handle invariant random subgroups coming from the relative profinite completion we also need:
\begin{theorem}\label{T:howemooreplan2}
Let $H$ be a simple nondiscrete locally compact second countable totally disconnected group with the Howe-Moore property.  If $H \actson (X,\nu)$ is an ergodic measure-preserving action with open stabilizer subgroups then $(X,\nu)$ is trivial.
\end{theorem}
\begin{proof}
Suppose $(X,\nu)$ is nontrivial so that $\mathrm{stab}_{H}(x) \ne H$ almost surely.
For almost every $x \in X$, since $\mathrm{stab}_{H}(x)$ is open in $H$ and $H$ has Howe-Moore then $\mathrm{stab}_{H}(x)$ is compact almost surely.  There are only countably many compact open subgroups of $H$ (as $H$ is second countable) so there exists $E \subseteq X$ with $\nu(E) > 0$ and $K_{0}$ a compact open subgroup such that $\mathrm{stab}_{H}(x) = K_{0}$ for all $x \in E$.  Now for $h \in H \setminus N_{H}(K_{0})$ we have that $hE \cap E = \emptyset$ (since $\mathrm{stab}_{H}(hx) = hK_{0}h^{-1}$ for $x \in E$).  As $\nu(E) > 0$ and $\nu$ is preserved by $H$ there exists a finite collection $h_{1},h_{2},\ldots,h_{n} \in H$ such that $X = \sqcup_{j=1}^{n} h_{j}E$.  Then $K = \cap_{j=1}^{n} h_{j}K_{0}h_{j}^{-1}$ is a compact open subgroup and $K < \mathrm{stab}_{H}(x)$ for almost every $x$ hence $K$ is in the kernel of the $H$-action.  As $H$ is simple and $K$ is nontrivial (since $H$ is nondiscrete) then the kernel is all of $H$ so $X$ is trivial.
\end{proof}

\begin{proposition}\label{P:usefulproduct}
Let $H = H_{1} \times \cdots \times H_{m}$ be a product of locally compact second countable groups where each $H_{j}$ has the property that any ergodic measure-preserving action of $H_{j}$ with open stabilizer subgroups is necessarily on the trivial space.  Then any ergodic measure-preserving action of $H$ with open stabilizer subgroups is necessarily on the trivial space.
\end{proposition}
\begin{proof}
Let $H \actson (X,\nu)$ be an ergodic measure-preserving action.  Let $(S(H),\eta)$ be the ergodic open invariant random subgroup corresponding to the $(X,\nu)$ stabilizers.  Fix $j$ and consider the map $p_{j} : S(H) \to S(H_{j})$ by $p_{j}(L) = L \cap H_{j}$ (meaning that $p_{1}(L) = L \cap H_{1} \times \{ e \} \times \cdots \times \{ e \}$).  Treat $S(H_{j})$ as an $H$-space where $H_{i}$ acts trivially on $S(H_{j})$ for $i \ne j$.  Then $p_{j}$ is an $H$-map from $(S(H),\eta)$ to $(S(H_{j}),\eta_{j})$ where $\eta_{j} = (p_{j})_{*}\eta$.

Since $\eta$-almost every $L \in S(H)$ is open so is $\eta_{j}$-almost every $L_{j} < H_{j}$.  Since $\eta$ is $H$-ergodic, $\eta_{j}$ is $H_{j}$-ergodic hence corresponds to an ergodic action of $H_{j}$ with open stabilizer subgroups  (Theorem \ref{T:gaussian}).  By hypothesis then $\eta_{j} = \delta_{H_{j}}$.  As this holds for each $j$, for $\eta$-almost every $L < H$ we have that $L \cap H_{j} = H_{j}$ hence $\langle H_{1}, \cdots, H_{m} \rangle \subseteq L$ and therefore $\eta = \delta_{H}$.  So $\overline{\mathrm{stab}_{\Lambda}(x)} = H$ for almost every $x$ and therefore $(X,\nu)$ is the trivial space.
\end{proof}

\begin{proposition}\label{P:usefulproduct2}
Let $H$ be a restricted infinite product $\prod^{\prime} H_{j}$ of locally compact second countable groups where each $H_{j}$ has the property that any ergodic measure-preserving action of $H_{j}$ with open stabilizer subgroups is necessarily on the trivial space.  Then any ergodic measure-preserving action of $H$ with open stabilizer subgroups is necessarily on the trivial space.
\end{proposition}
\begin{proof}
This follows exactly as in the proof of Proposition \ref{P:usefulproduct}.
\end{proof}

\section[Actions of Commensurators in Howe-Moore (T) Groups]{Actions of Commensurators in Howe-Moore $(T)$ Groups}\label{S:howemoorecomm}

Here we apply the results of the previous sections to derive concrete consequences about actions of commensurators in groups with Howe-Moore and property $(T)$:
\begin{corollary}\label{C:maintheorem1}
Let $G$ be a noncompact locally compact second countable group with the Howe-Moore property and property $(T)$.  Let $\Gamma < G$ be a lattice and $\Lambda < G$ be a countable dense subgroup such that $\Gamma <_{c} \Lambda$ and such that $\Lambda$ has finite intersection with every compact normal subgroup of $G$.

Then any ergodic measure-preserving action of $\Lambda$ either has finite stabilizers or the restriction of the action to $\Gamma$ has finite orbits.
\end{corollary}
\begin{proof}
This follows from Corollary \ref{C:main} and Theorem \ref{T:howemooreplan} (note that $\Gamma$ inherits property $(T)$ so must be finitely generated and cannot be virtually abelian).
\end{proof}

\begin{corollary}\label{C:maintheorem2}
Let $G$ be a noncompact locally compact second countable group with the Howe-Moore property and property $(T)$.  Let $\Gamma < G$ be a lattice and $\Lambda < G$ be a countable dense subgroup such that $\Gamma <_{c} \Lambda$ and such that $\Lambda$ has finite intersection with every compact normal subgroup of $G$.

The commensurability classes of infinite ergodic invariant random subgroups of $\Lambda$ are in one-one, onto correspondence with the commensurability classes of open ergodic invariant random subgroups of $\rpf{\Lambda}{\Gamma}$.
\end{corollary}
\begin{proof}
This follows from Theorem \ref{T:oneone} and Theorem \ref{T:howemooreplan}.
\end{proof}

\begin{corollary}\label{C:actualmain}
Let $G$ be a noncompact locally compact second countable group with the Howe-Moore property and property $(T)$.  Let $\Gamma < G$ be a lattice and $\Lambda < G$ be a countable dense subgroup such that $\Gamma <_{c} \Lambda$ and such that $\Lambda$ has finite intersection with every compact normal subgroup of $G$.

Assume that $\rpf{\Lambda}{\Gamma}$ is isomorphic to a finite (or restricted infinite) product $\prod H_{j}$ such that each $H_{j}$ is a simple nondiscrete locally compact second countable group with the Howe-Moore property.

Then any ergodic measure-preserving action of $\Lambda$ either has finite orbits or has finite stabilizers.
\end{corollary}
\begin{proof}
By Theorem \ref{T:howemooreplan}, any nontrivial ergodic action of $G$ has finite stabilizers when restricted to $\Lambda$.  Theorem \ref{T:howemooreplan2} applied to each $H_{j}$ says that open ergodic invariant random subgroups of $H_{j}$ correspond to the trivial space.  Theorem \ref{T:useful} combined with Propositions \ref{P:usefulproduct} and \ref{P:usefulproduct2} then gives the conclusion.
\end{proof}

\begin{corollary}\label{C:actualmain2}
Let $G$ be a product of connected noncompact locally compact second countable groups with the Howe-Moore property and property $(T)$.  Let $\Gamma < G$ be a lattice and $\Lambda < G$ be a countable dense subgroup such that $\Gamma <_{c} \Lambda$ and such that $\Lambda$ has finite intersection with every proper closed normal subgroup of $G$.

Assume that $\rpf{\Lambda}{\Gamma}$ is isomorphic to a finite (or restricted infinite) product $\prod H_{j}$ such that each $H_{j}$ is a simple nondiscrete locally compact second countable group with the Howe-Moore property.

Then any ergodic measure-preserving action of $\Lambda$ either has finite orbits or has finite stabilizers.
\end{corollary}
\begin{proof}
By Theorem \ref{T:productGconnected}, any nontrivial ergodic action of $G$ has finite stabilizers when restricted to $\Lambda$.  Theorem \ref{T:howemooreplan2} applied to the $H_{j}$ shows that any ergodic action of $H_{j}$ with open stabilizers is on the trivial space.  Theorem \ref{T:useful} combined with Propositions \ref{P:usefulproduct} and \ref{P:usefulproduct2} then gives the conclusion.
\end{proof}

\section[Actions of Lattices in Products of Howe-Moore (T) Groups]{Actions of Lattices in Products of Howe-Moore $(T)$ Groups}\label{S:howemoorelatt}

A consequence of the previous results is a generalization of the Bader-Shalom Normal Subgroup Theorem for Lattices in Product Groups to measure-preserving actions for certain product groups:
\begin{theorem}\label{T:BSbetter}
Let $G$ be a product of at least two simple nondiscrete noncompact locally compact second countable groups with the Howe-Moore property, at least one of which has property $(T)$, at least one of which is totally disconnected and such that every connected simple factor has property $(T)$.
Let $\Gamma < G$ be an irreducible lattice.

Then any ergodic measure-preserving action of $\Gamma$ either has finite orbits or has finite stabilizers.
\end{theorem}
\begin{proof}
Write $G_{0}$ to be the product of all the connected simple factors of $G$.  In the case when there are no connected simple factors instead take $G_{0}$ to be a simple factor with property $(T)$.  Write $H$ to be the product of all the simple factors not in $G_{0}$.  So $H$ is totally disconnected and nondiscrete.

Write $G = G_{0} \times H$ and let $K$ be a compact open subgroup of $H$.  Let $L = \Gamma \cap (G_{0} \times K)$.  Then $\mathrm{proj}_{K}~L$ is dense in $K$ since $\Gamma$ is irreducible.  $L$ is a lattice in $G_{0} \times K$ since $K$ is open.

Set $\Gamma_{0} = \proj_{G_{0}}~L$.  Since $K$ is compact, $\Gamma_{0}$ has finite covolume in $G_{0}$ since $L$ does in $G \times K$.  Moreover, $\Gamma_{0}$ is discrete since $L$ is discrete.  Therefore $\Gamma_{0}$ is a lattice in $G_{0}$.

Set $\Lambda_{0} = \proj_{G_{0}}~\Gamma$.  Then $\Lambda_{0}$ is dense in $G_{0}$ since $\Gamma$ is irreducible and $\Gamma_{0} <_{c} \Lambda_{0}$ since $K <_{c} H$.

By Propositions \ref{P:alpha} and \ref{P:idrpf}, $\rpf{\Gamma}{L}$ is isomorphic to $H / \mathrm{ker}(\tau_{H,K})$ since $\mathrm{proj} : \Gamma \to H$ is a homomorphism with dense image and pullback of $K$ equal to $L$.  Since $\mathrm{ker}(\tau_{H,K})$ is contained in $K$ and $H$ is semisimple then the kernel is trivial so $\rpf{\Gamma}{L}$ is isomorphic to $H$.

Set $N = \Gamma \cap \{ e \} \times H$ and write $M$ for the subgroup of $H$ such that $N = \{ e \} \times M$.  Then $N \normal \Gamma$ since $\{ e \} \times H \normal G \times H$ and $M$ is discrete in $H$ so $M = \overline{\mathrm{proj}_{H}~N} \normal \overline{\mathrm{proj}_{H}~\Gamma} = H$ by the irreducibility of $\Gamma$.  Since $H$ is simple, $M$ is trivial so $\Gamma \cap \{ e \} \times H$ is trivial.  This means that $\mathrm{proj}_{G} : \Gamma \to \Lambda_{0}$ is an isomorphism and so
$\rpf{\Lambda_{0}}{\Gamma_{0}} \simeq H$.

By Corollary \ref{C:actualmain} or Corollary \ref{C:actualmain2} (depending on whether $G_{0}$ is a single factor or a product of connected factors) then any ergodic measure-preserving action of $\Lambda_{0}$ either has finite orbits or has finite stabilizers.  The same then holds for $\Gamma \simeq \Lambda_{0}$.
\end{proof}

We remark that the above construction, writing an irreducible lattice in a product of nondiscrete groups, at least one of which is totally disconnected, as the commensurator of a lattice in one of the groups can also be reversed:
\begin{theorem}
Let $\Gamma$ be a lattice in a locally compact second countable group $G$ and let $\Lambda$ be a subgroup of $G$ such that $\Gamma <_{c} \Lambda$.  Then $\Lambda$ sits diagonally as a lattice in $G \times (\rpf{\Lambda}{\Gamma})$.
\end{theorem}
\begin{proof}
Let $\tau : \Lambda \to \rpf{\Lambda}{\Gamma}$ be the map defining the relative profinite completion and let
\[
\Lambda_{0} = \{ (\lambda, \tau(\lambda)) : \lambda \in \Lambda \} < G \times (\rpf{\Lambda}{\Gamma})
\]
be the diagonal embedding of $\Lambda$.  

Let $F$ be a fundamental domain for $G / \Gamma$: $F$ is of finite volume, $F \cap \Gamma = \{ e \}$ and $\Gamma \cdot F = G$. Let $K = \overline{\tau(\Gamma)}$ be the canonical compact open subgroup.  Let $\lambda_{0} \in \Lambda_{0} \cap F \times K$.  Then $\lambda_{0} = (\lambda, \tau(\lambda))$ for some $\lambda \in \Lambda \cap F$ such that $\tau(\lambda) \in K$.  Now $K = \overline{\tau(\Gamma)}$ and by Proposition \ref{P:rpfK}, $K \cap \tau(\Lambda) = \tau(\Gamma)$ so $\tau(\lambda) \in \tau(\Gamma)$ meaning that $\lambda \in \Gamma$ (as the kernel of $\tau$ is contained in $\Gamma$).  But $\lambda \in \Lambda \cap F$ so $\lambda \in \Gamma \cap F = \{ e \}$.
Therefore $F \times K$ is a subset of $G \times (\rpf{\Lambda}{\Gamma})$ of finite volume such that $\Lambda_{0} \cap F \times K = \{ e \}$ and, in particular, $\Lambda_{0}$ is discrete in $G \times (\rpf{\Lambda}{\Gamma})$.

Let $(g,h) \in G \times \rpf{\Lambda}{\Gamma}$ be arbitrary.  Write $h = \tau(\lambda^{\prime}) k^{\prime}$ for some $\lambda^{\prime} \in \Lambda$ and $k^{\prime} \in K$.  Write $(\lambda^{\prime})^{-1}g = \gamma f$ for some $\gamma \in \Gamma$ and $f \in F$.  Set $\lambda = \lambda^{\prime}\gamma$.  Then $\tau(\lambda)\tau(\gamma^{-1})k^{\prime} = \tau(\lambda^{\prime})k^{\prime} = h$ and $k = \tau(\gamma^{-1})k^{\prime} \in K$.  
Also $g = \lambda^{\prime}\gamma f = \lambda f$.  Therefore $(g,h) = (\lambda, \tau(\lambda)) (f,k) \in \Lambda_{0} \cdot (F \times K)$.

Therefore $F \times K$ is a fundamental domain for $\Lambda_{0}$ hence $\Lambda_{0}$ is a lattice as claimed.
\end{proof}

We remark that if both $G$ and $\rpf{\Lambda}{\Gamma}$ are semisimple with finite center then $\Lambda$ sits as an irreducible lattice if and only if $\Gamma$ is irreducible and $\Lambda$ is dense.

A consequence of this reverse construction is a special case of the Normal Subgroup Theorem for Commensurators \cite{CS14} following immediately from the Bader-Shalom Normal Subgroup Theorem for lattices in products:
\begin{corollary}
Let $\Gamma$ be an irreducible integrable lattice in a just noncompact locally compact second countable group $G$ and let $\Lambda$ be a dense subgroup of $G$ such that $\Gamma <_{c} \Lambda$.  Assume that $\rpf{\Lambda}{\Gamma}$ is just noncompact.  Then $\Lambda$ is just infinite.
\end{corollary}
\begin{proof}
Write $\Lambda$ as an irreducible lattice in the product $G \times \rpf{\Lambda}{\Gamma}$.  Observe that $\Lambda$ will be integrable (as a lattice) since $\Gamma$ is.  As the relative profinite completion is totally disconnected, it is not isomorphic to $\mathbb{R}$ and therefore the Bader-Shalom Normal Subgroup Theorem implies that $\Lambda$ has no nontrivial normal subgroups of infinite index.
\end{proof}

\section{Commensurators and Lattices in Lie Groups}\label{S:liegroups}

The primary example of a class of groups our results apply to is commensurators of lattices and lattices in higher-rank Lie groups.

\subsection{Actions of Commensurators in Semisimple Higher-Rank Lie Groups}

\begin{theorem}\label{T:semisimplecomm}
Let $G$ be a semisimple Lie group (real or $p$-adic or both) with finite center where each simple factor has rank at least two.  Let $\Gamma < G$ be an irreducible lattice.  Let $\Lambda < G$ be a countable dense subgroup such that $\Gamma <_{c} \Lambda$ and that $\Lambda$ has finite intersection with every proper subfactor of $G$.

Then any ergodic measure-preserving action of $\Lambda$ either has finite stabilizers or the restriction of the action to $\Gamma$ has finite orbits.

Moreover, the commensurability classes of infinite ergodic invariant random subgroups of $\Lambda$ are in one-one, onto correspondence with the commensurability classes of open ergodic invariant random subgroups of $\rpf{\Lambda}{\Gamma}$.
\end{theorem}
\begin{proof}
First note that if we show that the commensurability classes are in one-one, onto correspondence then any ergodic measure-preserving action of $\Lambda$ that does not have finite stabilizers must have finite $\Gamma$-orbits since any infinite ergodic invariant random subgroup of $\Lambda$ then contains $\Gamma$ up to finite index.  So we need only prove the one-one correspondence.

The case when $G$ is a real Lie group follows from Corollary \ref{C:maintheorem2} combined with Theorem \ref{T:productGconnected} and the case when $G$ is simple follows from Corollary \ref{C:maintheorem2} directly since every factor in $G$ has Howe-Moore and property $(T)$.

So we may assume that $G$ has at least two factors, at least one of which is totally disconnected.  By Theorem \ref{T:BSbetter}, any measure-preserving ergodic action of $\Gamma$ either has finite stabilizers or has finite orbits.  In particular, if $G \actson (X, \nu)$ is a measure-preserving ergodic action such that the restriction to $\Gamma$ does not have finite orbits then there must exist a positive measure $\Gamma$-invariant subset where the $\Gamma$-stabilizers are finite.  Since $\Lambda \actson (X, \nu)$ is ergodic we may then apply Proposition \ref{P:upgrade} and the Normal Subgroup Theorem for Commensurators to conclude that the restriction of the action to $\Lambda$ has finite stabilizers.  So $\Gamma <_{c} \Lambda < G$ satisfy the hypotheses of Theorem \ref{T:oneone} and the result follows.
\end{proof}

\subsection{Relative Profinite Completions of Arithmetic Lattices}

\begin{theorem}\label{T:rpfadelic}
Let $K$ be a global field, let $\mathcal{O}$ be the ring of integers, let $V$ be the set of places (inequivalent valuations) on $K$, let $V_{\infty}$ be the infinite places (archimedean valuations in the case of a number field) and let $K_{v}$ be the completion of $K$ over $v \in V$.

Let $V_{\infty} \subseteq S \subseteq V$ be any collection of valuations and let $\mathcal{O}_{S}$ be the ring of $S$-integers: $\mathcal{O}_{S} = \{ k \in K : v(k) \geq 0 \text{ for all $v \notin S$} \}$.   Let $\mathcal{O}_{v} = \mathcal{O}_{V_{\infty} \cup \{ v \}}$ be the ring of $v$-integers and let $\overline{\mathcal{O}_{v}}$ be the closure of the $v$-integers in $K_{v}$.

Let $\mathbf{G}$ be a simple algebraic group defined over $K$.  Then for $V_{\infty} \subseteq S^{\prime} \subseteq S$, the relative profinite completion $\rpf{\mathbf{G}(\mathcal{O}_{S})}{\mathbf{G}(\mathcal{O_{S^{\prime}}})}$ is isomorphic to the restricted product
\[
\prod_{v \in S \setminus S^{\prime}}^{\prime} \mathbf{G}(K_{v}) = \{ (g_{v})_{v \in S \setminus S^{\prime}} : g_{v} \notin \mathbf{G}(\overline{\mathcal{O}_{v}}) \text{ for only finitely many $v \in S \setminus S^{\prime}$} \}.
\]
\end{theorem}
\begin{proof}
That $\mathbf{G}(\mathcal{O_{S^{\prime}}})$ is commensurated by $\mathbf{G}(\mathcal{O}_{S})$ follows from the fact that any fixed element in $\mathcal{O}_{S}$ has a negative valuation on only finitely many valuations in $S \setminus S^{\prime}$.  Let $\varphi : \mathbf{G}(\mathcal{O}_{S}) \to \prod_{v \in S \setminus S^{\prime}}^{\prime} \mathbf{G}(K_{v})$ be the diagonal embedding.  Then $\varphi(\mathbf{G}(\mathcal{O}_{S}))$ is dense (since $S^{\prime}$ contains $V_{\infty}$) and $\varphi^{-1}(\prod_{v \in S \setminus S^{\prime}}^{\prime} \mathbf{G}(\overline{\mathcal{O}_{v}})) = \mathbf{G}(\mathcal{O_{S^{\prime}}})$.  By Proposition \ref{P:alpha}, $\rpf{\mathbf{G}(\mathcal{O}_{S})}{\mathbf{G}(\mathcal{O_{S^{\prime}}})}$ is isomorphic to $\rpf{\prod_{v \in S \setminus S^{\prime}}^{\prime} \mathbf{G}(K_{v})}{\prod_{v \in S \setminus S^{\prime}}^{\prime} \mathbf{G}(\overline{\mathcal{O}_{v}})}$. By Proposition \ref{P:idrpf}, it is isomorphic to $\prod_{v \in S \setminus S^{\prime}}^{\prime} \mathbf{G}(K_{v}) / M$ where $M$ is the largest closed normal subgroup contained in $\prod_{v \in S \setminus S^{\prime}}^{\prime} \mathbf{G}(\overline{\mathcal{O}_{v}})$.  Since $\mathbf{G}$ is simple, the only normal subgroups are of the form $\prod_{v \in S^{\prime\prime}}^{\prime} \mathbf{G}(K_{v})$ for $S^{\prime\prime} \subsetneq S \setminus S^{\prime}$.  But $M$ is contained in the $v$-integers for each $v \in S \setminus S^{\prime}$ so $M$ must be trivial.
\end{proof}

\begin{corollary}
Let $\mathbf{G}$ be a simple algebraic group over $\mathbb{Q}$ and let $S^{\prime} \subseteq S$ be sets of primes containing $\infty$.  Then the relative profinite completion $\rpf{\mathbf{G}(\mathbb{Z}_{S})}{\mathbf{G}(\mathbb{Z_{S^{\prime}}})}$ is isomorphic to the restricted product $\prod_{p \in S \setminus S^{\prime}}^{\prime} \mathbf{G}(\mathbb{Q}_{p})$.
In particular $\rpf{\mathbf{G}(\mathbb{Q})}{\mathbf{G}(\mathbb{Z})}$ is isomorphic to $\prod_{p \in \mathbb{P}}^{\prime} \mathbf{G}(\mathbb{Q}_{p})$ where $\mathbb{P}$ is the set of all primes.
\end{corollary}

\subsection{Actions of Lattices in Semisimple Higher-Rank Lie Groups}

\begin{corollary}\label{C:latticeLie}
Let $G$ be a semisimple Lie group (real or $p$-adic or both) with no compact factors, finite center, at least one factor with rank at least two and such that each real simple factor has rank at least two.  Let $\Gamma < G$ be an irreducible lattice.  Then any ergodic measure-preserving action of $\Gamma$ either has finite orbits or has finite stabilizers.
\end{corollary}
\begin{proof}
By Margulis' $S$-Arithmeticity Theorem \cite{Ma91}, and using that finite kernels and commensuration do not affect the conclusion, we may assume $\Gamma = \mathbf{G}(\mathbb{Z}_{S})$ and $G = \prod_{p \in S \cup \{ \infty \}} \mathbf{G}(\mathbb{Q}_{p})$ where $\mathbf{G}$ is a semisimple algebraic group over $\mathbb{Q}$ and $S$ is a finite set of primes containing $\infty$.

First observe that since $\Gamma$ is irreducible, the intersection of $\Gamma$ with any proper subfactor of $G$ is finite (that is, $\mathbf{G}(\mathbb{Z}_{S}) \cap \prod_{p \in S \setminus Q} \mathbf{G}(\mathbb{Q}_{p})$ is finite for any nonempty $Q \subseteq S$ since the lattice is embedded diagonally).  Also, $\Gamma$ has finite intersection with any compact factor because the original choice of $G$ has no compact factors.

The case when every simple factor of $G$ is real reduces to the Stuck-Zimmer Theorem \cite{SZ94}.  So instead we may assume there is some $p$-adic factor in $G$.  When $G$ is a simple $p$-adic group, the methods of Stuck-Zimmer combined with the more general Nevo-Zimmer Intermediate Factor Theorem \cite{nevozimmerpreprint} (for local fields of characteristic zero) give the conclusion 
since in this case there is no issue with action of $G$ being nonirreducible (however, in the case when there are two factors, such an issue does arise and their work does not apply).

Therefore we may assume that there are at least two noncompact simple factors, one of which is totally disconnected, so combined with the fact that each noncompact simple factor of $\mathbf{G}(\mathbb{Q}_{p})$ is Howe-Moore, Theorem \ref{T:BSbetter} then implies the conclusion.
\end{proof}

\begin{corollary}\label{C:latticeLie2}
Let $G$ be a semisimple Lie group (real or $p$-adic or both) with no compact factors, trivial center, at least one factor with rank at least two and such that each real simple factor has rank at least two.  Let $\Gamma < G$ be an irreducible lattice.  Then any ergodic measure-preserving action of $\Gamma$ on a nonatomic probability space is essentially free.
\end{corollary}
\begin{proof}
Let $\Gamma \actson (X,\nu)$ be an ergodic measure-preserving action on a nonatomic probability space.
Corollary \ref{C:latticeLie} implies that the action either has finite stabilizers or has finite orbits.  The case of finite orbits is ruled out by the space being nonatomic so the action has finite stabilizers.  For a finite subgroup $F < \Gamma$ let $E_{F} = \{ x \in X : \mathrm{stab}(x) = F \}$.  Since there are only countably many finite groups there is some $F$ with $\nu(E_{F}) > 0$.  Since $gE_{F} = E_{gFg^{-1}}$, $\nu(E_{gFg^{-1}}) = \nu(E_{F})$ for all $g$ so there are at most a finite number of finite groups appearing as stabilizers.  Therefore there is a subgroup $\Gamma_{0} < \Gamma$ of finite index such that $\Gamma_{0}$ normalizes the finite subgroup $F$.  Then $F \normal \Gamma_{0}$ and $\Gamma_{0}$ is a lattice in $G$ hence by Margulis' Normal Subgroup Theorem \cite{Ma91}, $F$ is contained in the center of $G$.  Therefore $F$ is trivial so all the stabilizer groups are trivial.
\end{proof}

\subsection{Actions of Rational Groups in Simple Higher-Rank Lie Groups}

\begin{corollary}\label{C:rationalLie}
Let $\mathbf{G}$ be a simple algebraic group defined over $\mathbb{Q}$ with $p$-rank at least two for some prime $p$, possibly $\infty$, such that $\mathbf{G}(\mathbb{R})$ is either compact or has rank at least two.  Let $S$ be any (finite or infinite) set of primes containing $\infty$ and $p$.
Then every ergodic measure-preserving action of $\mathbf{G}(\mathbb{Z}_{S})$ either has finite orbits or has finite stabilizers.
\end{corollary}
\begin{proof}
The case when $S$ contains only one prime $q$, possibly $\infty$, such that $\mathbf{G}(\mathbb{Q}_{q})$ is noncompact is a consequence of Corollary \ref{C:latticeLie}.  So assume $S$ contains more than one such prime.
Let $S^{\prime} = \{ p, \infty \}$.
By Theorem \ref{T:rpfadelic}, the relative profinite completion $\rpf{\mathbf{G}(\mathbb{Z}_{S})}{\mathbf{G}(\mathbb{Z_{S^{\prime}}})}$ is isomorphic to $\prod_{p \in S \setminus S^{\prime}}^{\prime} \mathbf{G}(\mathbb{Q}_{p})$.  The above facts about Lie groups imply that each factor of the relative profinite completion has Howe-Moore.  Therefore Corollary \ref{C:actualmain} applied to $\mathbf{G}(\mathbb{Z_{S^{\prime}}}) <_{c} \mathbf{G}(\mathbb{Z}_{S}) < \prod_{p \in S^{\prime}} \mathbf{G}(\mathbb{Q}_{p})$ (recall $\mathbb{Q}_{\infty} = \mathbb{R}$) implies the result.
\end{proof}

\begin{corollary}\label{C:algQ}
Let $\mathbf{G}$ be a simple algebraic group defined over $\mathbb{Q}$ with $p$-rank at least two for some prime $p$, possibly $\infty$, such that $\mathbf{G}(\mathbb{R})$ is either compact or has rank at least two.
Then every nontrivial ergodic measure-preserving action of $\mathbf{G}(\mathbb{Q})$ is essentially free.
\end{corollary}
\begin{proof}
Since $\mathbf{G}$ is simple as an algebraic group over $\mathbb{Q}$ the group $\mathbf{G}(\mathbb{Q})$ has no finite normal subgroups and therefore the previous corollary implies the conclusion.
\end{proof}

\subsection{Actions of Rational Groups in Simple Higher-Rank Groups}

\begin{theorem}
Let $K$ be a global field, let $\mathcal{O}$ be the ring of integers, let $V$ be the set of places (inequivalent valuations) on $K$, let $V_{\infty}$ be the infinite places (archimedean valuations in the case of a number field), let $K_{v}$ be the completion of $K$ over $v \in V$ and let $\mathcal{O}_{v}$ be the ring of $v$-integers.  Let $V_{\infty} \subseteq S \subseteq V$ and let $\mathcal{O}_{S}$ be the ring of $S$-integers.

Let $\mathbf{G}$ be a simple algebraic group defined over $K$ such that $\mathbf{G}$ has $v_{0}$-rank at least two for some $v_{0} \in S$ (possibly in $V_{\infty}$), $\mathbf{G}(K_{v})$ is noncompact for some $v \in S$, $v \ne v_{0}$, and $\mathbf{G}(K_{v_{\infty}})$ is compact or of higher-rank for all $v_{\infty} \in V_{\infty}$.  Then every ergodic measure-preserving action of $\mathbf{G}(\mathcal{O}_{S})$ either has finite orbits or has finite stabilizers.
\end{theorem}
\begin{proof}
Let $S^{\prime} = V_{\infty} \cup \{ v_{0} \}$.
Let $\Gamma = \mathbf{G}(\mathcal{O_{S^{\prime}}})$, let $\Lambda = \mathbf{G}(\mathcal{O}_{S})$ and let $G = \prod_{v \in S^{\prime}}\mathbf{G}(K_{v})$.  Then $\Gamma <_{c} \Lambda$ and $\Lambda$ is dense in $G$ (since $S$ contains some valuation $v \ne v_{0}$ where $\mathbf{G}(K_{v})$ is noncompact).  $\Gamma$ is an irreducible lattice in $G$ and each simple factor of $G$ has property $(T)$ and the Howe-Moore property.

By Theorem \ref{T:rpfadelic}, $\rpf{\Lambda}{\Gamma}$ is isomorphic to $\prod_{v \in S \setminus S^{\prime}}^{\prime} \mathbf{G}(K_{v})$ which is a product of simple locally compact groups with the Howe-Moore property.
Corollary \ref{C:actualmain} applied to $\Gamma <_{c} \Lambda < G$ then implies the result.
\end{proof}

\begin{corollary}\label{C:algK}
Let $\mathbf{G}$ be a simple algebraic group defined over a global field $K$ with $v$-rank at least two for some place $v$ such that the $v_{\infty}$-rank is at least two for every infinite place $v_{\infty}$.  Then every nontrivial ergodic measure-preserving action of $\mathbf{G}(K)$ is essentially free.
\end{corollary}

\dbibliography{references}


\newcommand{\etalchar}[1]{$^{#1}$}
\providecommand{\bysame}{\leavevmode\hbox to3em{\hrulefill}\thinspace}
\providecommand{\MR}{\relax\ifhmode\unskip\space\fi MR }
\providecommand{\MRhref}[2]{%
  \href{http://www.ams.org/mathscinet-getitem?mr=#1}{#2}
}
\providecommand{\href}[2]{#2}
\begin{thebibliography}{ABB{\etalchar{+}}11}

\bibitem[ABB{\etalchar{+}}11]{seven}
Mikl\'{o}s Abert, Nicolas Bergeron, Ian Biringer, Tsachik Gelander, Nikolay
  Nikolov, Jean Raimbault, and Iddo Samet, \emph{On the growth of {B}etti
  numbers of locally symmetric spaces}, Comptes Rendus Math\'{e}matique.
  Acad\'{e}mie des Sciences. Paris \textbf{349} (2011), no.~15--16, 831--835.

\bibitem[AEG94]{AEG}
S.~Adams, G.~Elliott, and T.~Giordano, \emph{Amenable actions of groups},
  Transactions of the American Mathematical Society \textbf{344} (1994), no.~2,
  803--822.

\bibitem[AGV14]{AGV12}
Mikl\'{o}s Abert, Yair Glasner, and B\'{a}lint Vir\'{a}g, \emph{Kesten's
  theorem for invariant random subgroups}, Duke Math J. \textbf{163} (2014),
  no.~3, 465--488.

\bibitem[AM66]{moore66}
L.~Auslander and C.~C. Moore, \emph{Unitary representations of solvable {L}ie
  groups}, Memoirs of the American Mathematical Society (1966), 66--77.

\bibitem[AS93]{adamsstuck}
S.~Adams and G.~Stuck, \emph{Splitting of nonnegatively curved leaves in
  minimal sets of foliations}, Duke Mathematical Journal \textbf{71} (1993),
  no.~1, 71--92.

\bibitem[Bek07]{bekka}
Bachir Bekka, \emph{Operator-algebraic supperrigidity for
  $\mathrm{SL}_{n}(\mathbb{Z})$, $n \geq 3$}, Inventiones Mathematica
  \textbf{169} (2007), no.~2, 401--425.

\bibitem[BG04]{gaboriau}
N.~Bergeron and D.~Gaboriau, \emph{Asymptotique des nombres de {B}etti,
  invariants $\ell^{2}$ et laminations [{A}symptotics of {B}etti numbers and
  $\ell^{2}$-invariants and laminations]}, Commentarii Mathematici Helvetici
  \textbf{79} (2004), no.~2, 362--395.

\bibitem[Bow12]{bowen12}
Lewis Bowen, \emph{Invariant random subgroups of the free group}, Preprint
  (2012), \arxiv{1204.5939}.

\bibitem[BS06]{BS04}
Uri Bader and Yehuda Shalom, \emph{Factor and normal subgroup theorems for
  lattices in products of groups}, Inventiones Mathematicae \textbf{163}
  (2006), no.~2, 415--454.

\bibitem[CFW81]{CFW}
A.~Connes, J.~Feldman, and B.~Weiss, \emph{An amenable equivalence relation is
  generated by a single transformation}, Ergodic Theory and Dynamical Systems
  \textbf{1} (1981), no.~4, 431--450.

\bibitem[CM09]{capracemonod}
Pierre-Emmanuel Caprace and Nicolas Monod, \emph{Isometry groups of
  non-positively curved spaces: Discrete subgroups}, Journal of Topology
  \textbf{2} (2009), no.~4, 701--746.

\bibitem[Cre11]{CrD11}
Darren Creutz, \emph{Commensurated subgroups and the dynamics of group actions
  on quasi-invariant measure spaces}, Ph.D. thesis, University of California:
  Los Angeles, 2011.

\bibitem[CS14]{CS14}
Darren Creutz and Yehuda Shalom, \emph{A normal subgroup theorem for
  commensurators of lattices}, Groups, Geometry and Dynamics \textbf{8} (2014),
  1--22.

\bibitem[CW85]{conneswood}
A.~Connes and E.~J. Woods, \emph{Approximately transitive flows and {IPTFI}
  factors}, Ergodic Theory and Dynamical Systems \textbf{5} (1985), 203--236.

\bibitem[Dye59]{dye1}
H.~A. Dye, \emph{On groups of measure-preserving transformations {I}}, American
  Journal of Mathematics \textbf{81} (1959), 119--159.

\bibitem[FG10]{FG10}
Hillel Furstenberg and Eli Glasner, \emph{Stationary dynamical systems},
  Dynamical Numbers---Interplay Between Dynamical Systems and Number Theory,
  Contemporary Mathematics, vol. 532, American Mathematical Society, 2010,
  pp.~1--28.

\bibitem[Fol94]{folland}
G.~Folland, \emph{A course in abstract harmonic analysis}, Studies in Advanced
  Mathematics, vol.~18, CRC Press, 1994.

\bibitem[Gla03]{glasner}
Eli Glasner, \emph{Ergodic theory via joinings}, Mathematical Surveys and
  Monographs, vol. 101, American Mathematical Society, 2003.

\bibitem[Gri11]{grig2}
R.I. Grigorchuk, \emph{Some topics in the dynamics of group actions on rooted
  trees}, Proceedings of the Steklov Institute of Mathematics \textbf{273}
  (2011), 64--175.

\bibitem[GS12]{grig1}
Rostislav Grigorchuk and Dmytro Savchuk, \emph{Self-similar groups acting
  essentially freely on the boundary of the binary rooted tree}, Preprint
  (2012), \arxiv{1212.0605}.

\bibitem[HK64]{hallkulatilaka}
P.~Hall and C.~R. Kulatilaka, \emph{A property of locally finite groups},
  Journal of the London Mathematical Society \textbf{39} (1964), 235--239.

\bibitem[HM79]{HM79}
R.E. Howe and C.C. Moore, \emph{Asymptotic properties of unitary
  representations}, Journal of Functional Analysis \textbf{32} (1979), 72--96.

\bibitem[Ioz94]{iozzi}
Alessandra Iozzi, \emph{Equivariant maps and purely atomic spectrum}, Journal
  of Functional Analysis \textbf{124} (1994), 211--227.

\bibitem[Jaw94]{Ja94}
Wojciech Jaworski, \emph{Strongly approximately transitive group actions, the
  {C}hoquet-{D}eny theorem, and polynomial growth}, Pacific Journal of
  Mathematics \textbf{165} (1994), no.~1, 115--129.

\bibitem[Jaw95]{Ja95}
\bysame, \emph{Strong approximate transitivity, polynomial growth, and spread
  out random walks on locally compact groups}, Pacific Journal of Mathematics
  \textbf{170} (1995), no.~2, 517--533.

\bibitem[JM13]{monod}
Kate Juschenko and Nicolas Monod, \emph{Cantor systems, piecewise translations
  and simple amenable groups}, Annals of Mathematics \textbf{178} (2013),
  775--787.

\bibitem[Jol05]{MR2122917}
Paul Jolissaint, \emph{The {H}aagerup property for measure-preserving standard
  equivalence relations}, Ergodic Theory and Dynamical Systems \textbf{25}
  (2005), no.~1, 161--174.

\bibitem[Kai02]{kaimanovichSAT}
Vadim Kaimanovich, \emph{{SAT} actions and ergodic properties of the horosphere
  foliation}, pp.~261--282, Springer, 2002.

\bibitem[LM92]{LM92}
Alexander Lubotzky and Shahar Mozes, \emph{Asymptotic properties of unitary
  representations of tree automorphisms}, Harmonic Analysis and Discrete
  Potential Theory, Plenum Press, 1992.

\bibitem[Mac62]{Ma62}
George Mackey, \emph{Point realizations of transformation groups}, Illinois
  Journal of Math \textbf{6} (1962), 327--335.

\bibitem[Mac66]{mackey}
\bysame, \emph{Ergodic theory and virtual groups}, Annals of Mathematics
  \textbf{166} (1966), 187--207.

\bibitem[Mar79]{Ma79}
Gregory Margulis, \emph{Finiteness of quotient groups of discrete subgroups},
  Funktsional\'nyi Analiz i Ego Prilozheniya \textbf{13} (1979), 28--39.

\bibitem[Mar91]{Ma91}
\bysame, \emph{Discrete subgroups of semisimple {L}ie groups}, Springer-Verlag,
  1991.

\bibitem[Mat06]{matsui}
Hiroki Matsui, \emph{Some remarks on topological full groups of {C}antor
  minimal systems}, International Journal of Mathematics \textbf{17} (2006),
  no.~2, 231--251.

\bibitem[NZ99a]{nevozimmerpreprint}
Amos Nevo and Robert Zimmer, \emph{A generalization of the intermediate factor
  theorem}, Preprint (1999).

\bibitem[NZ99b]{nevozimmer}
\bysame, \emph{Homogenous projective factors for actions of semi-simple lie
  groups}, Inventiones Mathematicae (1999), 229--252.

\bibitem[Pop06]{MR2215135}
Sorin Popa, \emph{On a class of type {$\mathrm{II}_1$} factors with {B}etti
  numbers invariants}, Annals of Mathematics. Second Series \textbf{163}
  (2006), no.~3, 809--899.

\bibitem[Ram71]{ramsay}
A.~Ramsay, \emph{Virtual groups and group actions}, Advances in Mathematics
  \textbf{6} (1971), 253--322.

\bibitem[Rob93]{robertson}
A.~Guyan Robertson, \emph{Property {$(T)$} for {$\mathrm{II}_1$} factors and
  unitary representations of {K}azhdan groups}, Mathematische Annalen
  \textbf{256} (1993), 547--555.

\bibitem[Rot80]{rothman}
Sheldon Rothman, \emph{The von {N}eumann kernel and minimally almost periodic
  groups}, Transactions of the American Mathematical Society (1980), 401--421.

\bibitem[Sch80]{Sch80}
G.~Schlichting, \emph{Operationen mit periodischen {S}tabilisatoren}, Archiv
  der Math. \textbf{34} (1980), 97--99.

\bibitem[Sch84]{schmidtmixing}
Klaus Schmidt, \emph{Asymptotic properties of unitary representations and
  mixing}, Proc. London Math. Soc. (3) \textbf{48} (1984), no.~3, 445--460.

\bibitem[Sch96]{schmidtcohomology}
\bysame, \emph{From infinitely divisible representations to cohomological
  rigidity, analysis, geometry and probability}, vol.~10, Hindustan Book
  Agency, Delhi, 1996.

\bibitem[SW13]{SW09}
Yehuda Shalom and George Willis, \emph{Commensurated subgroups of arithmetic
  groups, totally disconnected groups and adelic rigidity}, Geometry and
  Functional Analysis \textbf{23} (2013), 1631--1683.

\bibitem[SZ94]{SZ94}
Garrett Stuck and Robert Zimmer, \emph{Stabilizers for ergodic actions of
  higher rank semisimple groups}, The Annals of Mathematics \textbf{139}
  (1994), no.~3, 723--747.

\bibitem[TD12]{tuckerdrob}
Robin Tucker-Drob, \emph{Mixing actions of countable groups are almost free},
  Preprint (2012).

\bibitem[Tza00]{Tz00}
Kroum Tzanev, \emph{{$C^{*}$}-alg\'{e}bres de {H}ecke at {$K$}-theorie}, Ph.D.
  thesis, Universit\'{e} Paris 7 -- Denis Diderot, 2000.

\bibitem[Tza03]{Tz03}
\bysame, \emph{{H}ecke {$C^{*}$}-alg\'{e}bres and amenability}, Journal of
  Operator Theory \textbf{50} (2003), 169--178.

\bibitem[Var63]{vara}
V.S. Varadarajan, \emph{Groups of automorphisms of {B}orel spaces},
  Transactions of the American Mathematical Society \textbf{109} (1963), no.~2,
  191--220.

\bibitem[Ver11]{vershik11}
A.~M. Vershik, \emph{Nonfree actions of countable groups and their characters},
  Journal of Mathematical Science \textbf{174} (2011), no.~1, 1--6.

\bibitem[Ver12]{vershiktotallynonfree}
\bysame, \emph{Totally nonfree actions and the infinite symmetric group},
  Moscow Mathematical Journal \textbf{12} (2012), 193--212.

\bibitem[Zim77]{Zi77}
Robert Zimmer, \emph{Hyperfinite factors and amenable group actions},
  Inventiones Mathematicae \textbf{41} (1977), no.~1, 23--31.

\bibitem[Zim78]{Zi78}
\bysame, \emph{Amenable ergodic group actions and an application to {P}oisson
  boundaries of random walks}, J.~Functional Analysis \textbf{27} (1978),
  350--372.

\bibitem[Zim82]{zimmer2}
\bysame, \emph{Ergodic theory, semi-simple lie groups, and foliations by
  manifolds of negative curvature}, Publications Math\'{e}matique de l'IH\'{E}S
  (1982), 37--62.

\bibitem[Zim84]{Zi84}
\bysame, \emph{Ergodic theory and semisimple groups}, Birkhauser, 1984.

\bibitem[Zim87]{zimmer3}
\bysame, \emph{Split rank and semisimple automorphism group of a
  {G}-structure}, Journal of Differential Geometry \textbf{26} (1987),
  169--173.

\end{thebibliography}
\end{document}